\documentclass[aos,preprint]{imsart}

\RequirePackage[OT1]{fontenc}
\RequirePackage{amsthm,amsmath,color}
\RequirePackage[numbers]{natbib}
\RequirePackage[colorlinks,citecolor=blue,urlcolor=blue]{hyperref}

\arxiv{1408.0361}

\startlocaldefs
\numberwithin{equation}{section}
\theoremstyle{plain}

\endlocaldefs

\usepackage[utf8]{inputenc}			
\usepackage[T1]{fontenc}				
\usepackage{lmodern}					
\usepackage[english]{babel}				
\usepackage{amsmath,amsfonts,amscd}		
\usepackage{mathrsfs}
\usepackage{graphicx}					
\usepackage{amsthm}
\usepackage{enumitem}
\usepackage{amssymb}
\usepackage{bbold}
\usepackage{color}
\usepackage[table]{xcolor}
\usepackage{ulem}
\usepackage[toc,page]{appendix} 		
\usepackage{hyperref}					
\usepackage{caption}
\usepackage{subcaption}
\usepackage{float}
\usepackage{grffile}					

\newcommand{\lecc}{\preccurlyeq}
\newcommand{\idm}{\mathrm{I}}
\newtheorem{Def}{Definition}			
\newtheorem{Th}{Theorem}				
\newtheorem{Lem}{Lemma}					
\newtheorem{Cor}{Corollary}				
\newtheorem{Prop}{Proposition}			

\DeclareMathOperator{\tr}{tr}			
\DeclareMathOperator{\Image}{Im}		
			
\DeclareMathOperator{\Var}{Var}			
\DeclareMathOperator{\Bias}{Bias}		
\DeclareMathOperator{\Span}{span}		
\DeclareMathOperator{\var}{var}			
\DeclareMathOperator{\bias}{bias}		
\DeclareMathOperator{\Ker}{Ker} 		%
\DeclareMathOperator{\supp}{supp} 		%

\newcommand{\N}{\mathbb{N}}
\newcommand{\E}{\mathbb{E}}
\newcommand{\R}{\mathbb{R}}

\newcommand{\lec}{\preccurlyeq}

\renewcommand{\le}{\leqslant}

\renewcommand{\geq}{\geqslant}
\renewcommand{\epsilon}{\varepsilon }

\renewcommand{\H}{\mathcal{H}}
\newcommand{\x}{K_{x_n} \otimes K_{x_n}}
\newcommand{\X}{\mathcal{X}}

\newcommand{\lem}{\preccurlyeq}

\newcommand\nb[1]{\bar{\eta}_{#1}}
\newcommand\n[1]{{\eta}_{#1}}
\newcommand\tb[1]{{\overline{g}}_{#1}}
\renewcommand\t[1]{{g}_{#1}}

\renewcommand{\[}{\left[}
\renewcommand{\]}{\right]}
\renewcommand{\(}{\left(}
\renewcommand{\)}{\right)}

\renewcommand{\L}{\mathcal{L}^2_{\rho_{{X}}}}

\newcommand{\Ld}{{L}^2_{\rho_{{X}}}}

\newcommand{\T}{\mathcal{T}}

\newcommand{\Td}{T}
\renewcommand{\S}{\mathscr{S}}

\newcommand{\refa}[1]{\ref{#1}}
\newcommand{\refawc}[1]{\ref{#1}}

\begin{document}

\begin{frontmatter}
\title{Non-parametric Stochastic Approximation with Large Step-sizes}
\runtitle{Non-parametric Stochastic Approximation}

\begin{aug}
\author{\fnms{Aymeric} \snm{Dieuleveut}\ead[label=e1]{aymeric.dieuleveut@ens.fr}},
\author{\fnms{Francis} \snm{Bach}\ead[label=e2]{francis.bach@ens.fr}}

\runauthor{Dieuleveut and Bach}

\affiliation{D\'epartement d'Informatique de l'Ecole Normale Sup\'erieure, Paris, France}

\address{SIERRA Project-Team\\
23, avenue d'Italie\\
75013 Paris, France \\
\printead{e1}\\
\phantom{E-mail:\ }\printead*{e2}}
\end{aug}

\begin{abstract} \hspace*{.5cm}  
We  consider the random-design least-squares regression problem within the reproducing kernel Hilbert space (RKHS) framework.  Given a stream of independent and identically distributed input/output data, we aim to learn a regression function within an RKHS $\mathcal{H}$, even if the optimal predictor (i.e., the conditional expectation) is not in $\mathcal{H}$. In a stochastic approximation framework where the estimator is updated after each observation, we show that the averaged unregularized least-mean-square algorithm (a form of  stochastic gradient descent), given a sufficient large step-size, attains optimal rates of convergence for a variety of regimes for the smoothnesses of the optimal prediction function and the functions in $\mathcal{H}$.
\end{abstract}

\begin{keyword}[class=MSC]
\kwd[Primary ]{60K35}
\kwd{60K35}
\kwd[; secondary ]{60K35}
\end{keyword}

\begin{keyword}
\kwd{Reproducing kernel Hilbert space}
\kwd{Stochastic approximation}
\end{keyword}

\end{frontmatter}


\section{Introduction}
Positive-definite-kernel-based methods such as the support vector machine or kernel ridge regression are now widely used in many areas of science of engineering. They were first developed within the statistics community for non-parametric regression using splines, Sobolev spaces, and more generally reproducing kernel Hilbert spaces~(see, e.g.,~\cite{wah1990splines}). Within the machine learning community, they were extended in several interesting ways (see, e.g.,~\cite{smola-book,Cristianini2004}): (a) other problems were tackled using positive-definite kernels beyond regression problems, 
through the ``kernelization'' of classical unsupervised learning methods such as principal component analysis, canonical correlation analysis, or K-means, (b) efficient algorithms based on convex optimization have emerged, in particular for large sample sizes, and (c) kernels for non-vectorial data have been designed for objects like strings, graphs, measures, etc. A key feature is that  they allow the separation of the representation problem  (designing good kernels for non-vectorial data) and the algorithmic/theoretical problems (given a kernel, how to design, run efficiently and analyse estimation algorithms).

The theoretical analysis of non-parametric least-squares regression within the RKHS framework is well understood. 
In particular,  regression on input data in $\mathbb{R}^d$, $d \geqslant 1$, and so-called \emph{Mercer kernels} (continuous kernels over a compact set) that lead to dense subspaces of the space of square-integrable functions and non parametric estimation \cite{Tsy2008intro}, has been widely studied in the last decade starting with the works of Smale and Cucker \cite{Sma2001mathematical, sma2002best} and being further refined  \cite{vito2005model,sma2007learning} up to optimal rates \cite{cap2007optimal, ste2009optimal, bac2012sharp} for  Tikhonov regularization (batch iterative methods were for their part studied in \cite{bla2010optimal,rask2011early}).
 However, the kernel framework goes beyond Mercer kernels and non-parametric regression; indeed, kernels on non-vectorial data provide examples where the usual topological assumptions may not be natural, such as sequences, graphs and measures. Moreover, even finite-dimensional Hilbert spaces may need a more refined analysis when the dimension  of the Hilbert space is much larger than the number of observations: for example, in modern text and web applications, linear predictions are performed with a large number of covariates  which are equal to zero with high probability. The sparsity of the representation allows to reduce significantly the complexity of traditional optimization procedures; however, the finite-dimensional analysis which ignores the spectral structure of the data often leads to trivial guarantees because the number of covariates far exceeds the number of observations, while the analysis we carry out is meaningful (note that in these contexts sparsity of the underlying estimator is typically not a relevant assumption).
In this paper,  we consider minimal assumptions regarding the input space and the distributions, so that our non-asymptotic results may be applied to all the cases mentioned above.

In practice,   estimation algorithms based on regularized empirical risk minimization (e.g., penalized  least-squares) face two challenges: (a) using the correct regularization parameter and (b) finding an approximate solution of the convex optimization problems. In this paper, we consider these two problems jointly by following a stochastic approximation framework formulated directly in the RKHS, in which each observation is  used only once and overfitting is avoided by making only a single pass through the data--a form of \emph{early stopping}, which has been considered in other statistical frameworks such as boosting~\cite{zhang2005boosting}. While this framework has been considered before~\cite{ros2014regularisation,yin2008online,tar2011online}, the algorithms that are considered either (a) require two sequences of hyperparameters (the step-size in stochastic gradient descent and a regularization parameter) or (b) do not always attain the optimal rates of convergence for estimating the regression function. In this paper, we aim to remove simultaneously these two limitations.

Traditional online stochastic approximation algorithms, as  introduced by Robbins and Monro \cite{rob1951stochastic}, lead in finite-dimensional learning problems (e.g., parametric least-squares regression) to stochastic gradient descent methods with step-sizes decreasing with the number of observations $ n $, which are typically proportional to $ n^{-\zeta} $, with $ \zeta $   between  $1/2$ and 1. Short step-sizes ($\zeta=1$) are adapted to well-conditioned problems (low dimension, low correlations between covariates), while longer step-sizes ($\zeta=1/2$) are adapted to ill-conditioned problems (high dimension, high correlations) but with a worse convergence rate---see, e.g.,~\cite{shalev2011online,gradsto} and references therein.   More recently \cite{bac2013nonstrongly} showed that  constant step-sizes \emph{with averaging} could lead to the best possible convergence rate in Euclidean spaces (i.e., in finite dimensions).
In this paper, we show that using longer step-sizes with averaging also brings benefits to Hilbert space settings needed for non parametric regression.

With our analysis, based on positive definite kernels, under assumptions on both the objective function and the covariance operator of the RKHS, we derive improved rates of convergence \cite{cap2007optimal}, in both the finite horizon setting where the number of observations is known in advance and our bounds hold for the last iterate (with exact constants), and the online setting where our bounds hold for each iterate (asymptotic results only). It leads to an explicit choice of the step-sizes (which play the role of the regularization parameters) which may be used in stochastic gradient descent,  depending on the number of training examples we want to use and on the assumptions we make.

In this paper, we make the following contributions:
\begin{itemize}

\item[--] 
We review in Section~\ref{sec:rkhs} a general though simple algebraic framework for least-squares regression in RKHS, which encompasses all commonly encountered situations. This framework however makes unnecessary topological assumptions, which we relax in Section~\ref{Sec:minass} (with details in App.~\ref{App:rkhsnoproof}).

\item[--] We characterize in Section~\ref{sec:lms} the convergence rate of averaged least-mean-squares (LMS) and show how the proper set-up of the step-size leads to optimal convergence rates (as they were proved in \cite{cap2007optimal}), extending results from finite-dimensional~\cite{bac2013nonstrongly} to infinite-dimensional settings. The problem we solve here was stated as an open problem in \cite{ros2014regularisation,yin2008online}.  Moreover, our results apply as well in the usual finite-dimensional setting of parametric least-squares regression, showing adaptivity of our estimator to the spectral decay of the covariance matrix of the covariates (see Section~\ref{subsec:euclidien}).

\item[--] We compare our new results with existing work, both in terms of rates of convergence in Section~\ref{sec:links}, and with simulations on synthetic spline smoothing in Section~\ref{sec:experiments}.
\end{itemize}

Sketches of the proofs are given in Appendix~\ref{app_sketch}.
Complete proofs are given in Appendices~\ref{A_rkhs},~\ref{A_proofs}.

\section{Learning with positive-definite kernels}
\label{sec:rkhs}

In this paper, we consider a general random design regression problem, where  observations $(x_i,y_i)$ are  independent and identically distributed (i.i.d.) random variables in $\mathcal{X\times Y} $ drawn from a probability measure $\rho$ on $ \mathcal{X\times Y}  $. 
The set $\mathcal{X}$ may be any set equipped with a measure; moreover we consider for simplicity $\mathcal{Y} = \R$ and we measure the risk of a function $g: \mathcal{X} \to \R$, by the mean square error, that is,  $\epsilon(g):=\E_\rho \left[(g(X)-Y)^2 \right]$. 

The function $g$ that minimizes $\epsilon(g)$ over all measurable functions is known to be the conditional expectation, that is,  $g_\rho(X)=\E [Y|X]$.
 In this paper we consider formulations where our estimates  lie in a reproducing kernel Hilbert space (RKHS) $\mathcal{H}$ with positive definite kernel $K: \mathcal{X} \times \mathcal{X} \to \R$.


\subsection{Reproducing kernel Hilbert spaces}\label{subsec:rkhs}
Throughout this section, we make the following assumption:
 \begin{enumerate}
\item[\textbf{(A1)}]  $\X$ is a compact topological space and $\mathcal{H}$ is an  RKHS associated with a continuous kernel $K$ on  the set $\mathcal{X}$.
\end{enumerate}
RKHSs are well-studied Hilbert spaces which are particularly adapted to regression problems (see, e.g.,~\cite{berlinet2004reproducing,wah1990splines}). 
They satisfy the following properties:

\begin{enumerate}
\item  $\left(\H, \langle \cdot, \cdot \rangle_\mathcal{H}\right)$ is a separable Hilbert space of functions: $\mathcal{H}\subset \R^{\X}$.
\item $\mathcal{H}$ contains all functions $  K_x: t \mapsto K(x,t)$, for all $x$ in $ \mathcal{X} $.
\item For any $x\in \mathcal{X}$ and $f\in \mathcal{H}$, the reproducing property holds:
\begin{equation*}
f(x) =\langle f, K_x\rangle_{\mathcal{H}}. 
\end{equation*}
\end{enumerate} 	  
The reproducing property allows to treat non-parametric estimation in the same algebraic framework as parametric regression.
 The Hilbert space $\mathcal{H}$ is totally characterized by the positive definite kernel~$K: \mathcal{X} \times \mathcal{X} \to \mathbb{R}$, which simply needs to be a symmetric function on $\mathcal{X} \times \mathcal{X}$ such that for any finite family of points $(x_i)_{i \in I}$ in $\mathcal{X}$, the $|I| \!\times\!|I|$-matrix of kernel evaluations is positive semi-definite. We provide examples in Section~\ref{sec:examples}. 
 For simplicity, we have here made the assumption that $K$ is a Mercer kernel, that is, $\X$ is a compact set  and  $K : \X \times \X \to \R$  is continuous. See Section~\ref{Sec:minass} for an extension without topological assumptions.

\subsection{Random variables}\label{subsec:randomvar}
In this paper, we consider a set $\X$ and $\mathcal{Y}\subset \R$ and a distribution  $\rho$ on $\X\times \mathcal{Y}$. We denote by $\rho_X$ the marginal law on the space~$\X$ and by $\rho_{Y|X=x} $ the conditional probability measure on $ Y $ given $ x\in \mathcal{X} $.  We may use the notations $\E \left[f(X)\right] $ or $\E_{\rho_X}\left[f(\cdot) \right]$ for $\int_{\mathcal{X}} f(x) d\rho_X(x)$. Beyond the moment conditions stated below, we will always make the assumptions that the space $\Ld$ of square $\rho_X$-integrable functions defined below is separable (this is the case in most interesting situations; see~\cite{tho2000elementary} for more details). Since we will assume that $\rho_X$ has full support, we will make the usual simplifying identification of functions and their equivalence classes (based on equality up to a zero-measure set).  We denote by $\| \cdot\| _{\Ld}$ the norm:
 \begin{equation*}
\| f\| _{\Ld}^{2}=\int_\X |f(x)|^{2} d\rho_X(x) .
\end{equation*}
The space $\Ld$ is then a  Hilbert space with norm $\| \cdot \| _{\Ld}$, which we will always assume separable (that is, with a countable orthonormal system).

Throughout this section, we make the following simple assumption regarding finiteness of moments:
 \begin{enumerate}
 \item[\textbf{(A2)}]  $R^{2}:=\sup_{x\in \X} K(x,x)$ and $ \E [Y^2] $ are finite; $\rho_X$ has full support in $\X$.
 \end{enumerate}

Note that under these assumptions, any function in $\H$ in in $\Ld$; however this inclusion is strict in most interesting situations.

\subsection{Minimization problem}
We are interested in minimizing the following quantity, which is the \textit{prediction error} (or mean squared error) of a function $f$, defined for any function in $\Ld$ as: 
\begin{equation}
\epsilon(f)=\E \left[\left(f(X)-Y\right)^2\right].
\end{equation}

We are looking for a function with a low prediction error in the particular function space $\H$, that is we aim to minimize $\epsilon(f)$ over $f \in \H$.
We have for $f \in \Ld$: 
\begin{eqnarray}
\epsilon(f)&=& \| f\| ^2_{\Ld} - 2\left\langle f , \int_\mathcal{Y} y d\rho_{Y|X=\cdot}(y) \right\rangle_{\Ld} +  \E [Y^2] \\
&=&  \| f\| ^2_{\Ld} - 2\left\langle f , \E\left[Y|X=\cdot\right] \right\rangle_{\Ld} +  \E [Y^2]  \nonumber.
\end{eqnarray}
A minimizer $g$  of $\epsilon(g)$ over $\Ld$ is known to be such that $g(X)=\E[Y|X]$. Such a function is generally referred to as the regression function, and denoted $g_\rho$ as it only depends on $\rho$. It is moreover unique (as an element of $\Ld$). An important property of the prediction error is that
the excess risk may be expressed as a squared distance to $g_\rho$, i.e.:
 \begin{equation}
\forall f \in \Ld, \qquad \epsilon(f)-\epsilon(g_\rho) = \| f-g_\rho\| ^2_{\Ld}.
\end{equation}

A key feature of our analysis is that we only considered $\|f-g_\rho\|^{2}_{\Ld}$ as a measure of performance and do not consider convergences in stricter norms (which are not true in general). \emph{This allows us   to neither  assume that $g_\rho$ is in $\H$ nor that $\H$ is dense in $\Ld$.} 
We thus need to define a notion of the best estimator in $\H$. We first define the closure  $\overline{{F}}$ (with respect to $\| \cdot\| _{\Ld}$) of any set $F \subset \Ld$ as  the set of limits in $\Ld$ of sequences in ${F}$. The space $\overline{\H}$ is a closed and convex subset in $\Ld$.  We can thus define $g_\H=\arg\min_{f\in \,\overline{\H}} \epsilon(g)$, as the orthogonal projection  of $g_{\rho}$ on $\overline{\H}$, using the existence of the projection on any closed convex set in a Hilbert space. See Proposition~\ref{prop:def_approximation_function} in Appendix~\ref{App:rkhsnoproof} for details.
Of course we do not have $g_\H\in \H$, that is \emph{the minimum in $\H$ is in general not attained}.

Estimation from $n$ i.i.d.~observations builds a sequence $(g_n)_{n\in \mathbb{N}}$ in $\H$. We will prove under suitable conditions that such an estimator satisfies weak consistency, that is $g_n$ ends up predicting as well as $g_\mathcal{H}$: \begin{equation*}
 \E \left[\epsilon (g_n) -\epsilon(g_\H)\right]\xrightarrow{n\rightarrow \infty} 0
 \ \Leftrightarrow \ \| g_n - g_\H \|_{\L} \xrightarrow{n\rightarrow \infty} 0.
\end{equation*}

 Seen as a function of $f \in \H$, our loss function $\epsilon$ is not coercive (i.e., not strongly convex), as our covariance operator (see definition below)  $\Sigma$ has no minimal strictly positive eigenvalue (the sequence of eigenvalues decreases to zero). As a consequence, even if $g_\H \in \H$, $g_n$  may not converge to $g_\H$ in $\H$, and \emph{when $g_\H \notin \H $, we shall even have $\|  g_n\| _\H \rightarrow \infty$}.

\subsection{Covariance operator}\label{subsec:covoper} We now define the \textit{covariance operator} for the space $\H$ and probability distribution $\rho_X$. The spectral properties of such an operator have appeared to be a key point to characterize the convergence rates of estimators~\cite{Sma2001mathematical,sma2007learning,cap2007optimal}.

We implicitly  define (via Riesz' representation theorem) a linear  operator $\Sigma:   \H  \rightarrow   \H $ through
$$ \forall (f, g) \in \H^2,\quad \langle f, \Sigma g \rangle_\H  =  \E \left[ f(X)g(X) \right] =\int_{\X} f(x) g(x) d\rho_X(x).$$
This operator is the \textit{covariance operator} (defined on the Hilbert space $\H$). Using the reproducing property, we have:
$$
\Sigma =\E\left[K_X\otimes K_X\right],
$$
where for any elements $g, h \in \H $, we denote by $g \otimes h$ the operator from $\H$ to $\H$ defined as:
 $$
  g \otimes h~:  f \mapsto  \langle f, h \rangle_\H \ g. 
$$
Note that this expectation is formally defined as a Bochner expectation (an extension of Lebesgue integration theory to Banach spaces, see~\cite{mik2014bochner}) in $\mathcal{L}(\H)$ the set of endomorphisms of $\H$.

In finite dimension, i.e., $\H = \mathbb{R}^d$, for $g,h \in \mathbb{R}^d$,  $g \otimes h$ may be identified to a rank-one matrix, that is, $g \otimes h = g h^\top = \left(\left(g_i h_j\right)_{1\le i,j\le d}\right) \in \mathbb{R}^{d \times d}$ as for any~$f$, $(g h^\top) f = g (h^\top f) = \langle f,h \rangle_\H g $. In other words, $g \otimes h$ is a linear operator, whose image is included in $\text{Vect} (g)$, the linear space spanned by~$g$. Thus in finite dimension, $\Sigma$ is the usual (non-centered) covariance matrix.

We have defined the covariance operator on the Hilbert space $\mathcal{H}$. If $f \in \mathcal{H}$, we have for all $z \in \mathcal{X}$, using the reproducing property:
$$
\E [f(X) K(X,z)]  = \E [ f(X) K_z(X) ] =   \langle K_z, \Sigma f \rangle_\H = (\Sigma f)(z),
$$
which shows that the operator $\Sigma$ may be extended to any square-integrable function $f \in \Ld$.
In the following, we extend such an operator as an endomorphism $\Td$ from $\Ld$  to $\Ld$. 

\begin{Def}[Extended covariance operator]
Assume \textbf{(A1-2)}. We define the operator $\Td$ as follows:
\begin{eqnarray*}
\Td   :\ \  \Ld  & \rightarrow&  \Ld  \\
{g} &\mapsto & \int_\mathcal{X} g(t)\  K_t  \ d\rho_\mathcal{X}(t),
\end{eqnarray*}
so that for any $ z \in \mathcal{X}  $, $ \Td(g)(z)= \displaystyle \int_\mathcal{X} g(x)\  K(x,z)  \ d\rho_\mathcal{X}(t) = \E [ g(X) K(X,z)]. $
\end{Def}

From the discussion above, if $f \in \H \subset \Ld$, then $\Td f = \Sigma f$. We give here some of the most important properties of $\Td  $. The operator~$\Td$ (which is an endomorphism of the separable Hilbert space $\Ld$) may be reduced in some Hilbertian eigenbasis of $\Ld$. It allows us to define the power of such an operator $T^r$, which will be used to quantify the regularity of the function $g_\H$.  See proof in Appendix~\refa{subsec:app_covar}, Proposition~\refawc{propTdr}.

\begin{Prop}[Eigen-decomposition of $\Td$] \label{propTd}
Assume  \textbf{(A1-2)}. $\Td$ is a  bounded  self-adjoint semi-definite positive operator on $\Ld$, which is trace-class. There exists a Hilbertian eigenbasis $(\phi_i)_{i \in I}$ of the orthogonal supplement $S$ of the null space ${\rm Ker}(\Td)$, with summable strictly positive eigenvalues $(\mu_i)_{i \in I}$. That is:
\begin{itemize}
\item[--] $\forall i \in I , \   \Td \phi_i= \mu_i \phi_i $, $(\mu_i)_{i \in I} $ strictly positive  such that $\sum_{i \in I } \mu_i < \infty$.
\item[--] $\Ld= \Ker(\Td) \overset{\perp}{\oplus} S$, that is, $\Ld$ is the orthogonal direct sum of 
$\Ker(\Td)$ and $S$.
\end{itemize}
\end{Prop}
When the space $S$ has finite dimension, then $I$ has finite cardinality, while in general $I$ is countable. Moreover, the null space $\Ker(\Td) $ may be either reduced to $\{0\} $ (this is the more classical setting and such an assumption is often made), finite-dimensional (for example when the kernel has zero mean, thus constant functions are in $S$) or infinite-dimensional (e.g., when the kernel space only consists in even functions, the whole space of odd functions is in $S$).

Moreover, the linear operator $\Td$ allows to relate $\Ld$ and $\H$ in a very precise way. For example, when $g \in \H$, we immediately have $T g = \Sigma g \in \H$ and $\langle g, T g \rangle_\H = \E g(X)^2 = \| g\|_{\Ld}^2$. As we formally state in the following propositions, this essentially means that $T^{1/2}$ will be an isometry from  $\Ld$ to $\H$.
We first show that the linear operator $\Td$ happens to  have an image included in $\H$, and that the eigenbasis of $\Td$ in $\Ld$ may also be seen as eigenbasis of $\Sigma$ in $\H$ (See proof in Appendix~\refa{subsec:app_covar}, Proposition~\refawc{prop:app_decSigma}):

\begin{Prop}[Decomposition of $\Sigma$]\label{prop:decsigmainjec}
Assume  \textbf{(A1-2)}. $\Sigma: \H \to \H$ is injective. The image of $\Td$ is included in $\H$: $\text{Im}(\Td) \subset \H$. Moreover, for any $i \in I$, $\phi_i = \frac{1}{\mu_i} \Td \phi_i \in \H $ , thus $\(\mu_i^{1/2} \phi_i\right)_{i \in I}$ is an orthonormal eigen-system of $\Sigma$ and an Hilbertian basis of $\H$, i.e.,  for any $ i $ in $ I , \   \Sigma \phi_i= \mu_i \phi_i $.
\end{Prop}

This proposition will be generalized under relaxed assumptions (in particular as $\Sigma$ will no more be injective, see  Section~\ref{Sec:minass} and Appendix~\ref{App:rkhsnoproof}).

We may now  define all powers $\Td^r$ (they are always well defined because the sequence of eigenvalues is upper-bounded): 
 
 \begin{Def}[Powers of $\Td$]\label{def:tdr}
 We define, for any $ r \geqslant 0 $, $\Td^r:  \Ld      \rightarrow  \Ld$, for any $ h \in  \Ker(\Td)$ and $(a_i)_{i \in I}$ such that $\sum_{i \in I } a_i^2 < \infty$, through:
 $
 \Td^r\(h + \sum_{i \in I} a_i \phi_i \right) = \sum_{i \in I} a_i \mu_i^r\phi_i.
 $
 Moreover, for any $r>0$, $\Td^r$ may be defined as a bijection from $S$ into $\text{Im}(\Td^r)$. We may thus define its unique inverse
$\Td^{-r}  :\  \  \text{Im}(\Td^r)  \rightarrow  S.
$
\end{Def}

The following proposition is  a consequence of Mercer's theorem \cite{Sma2001mathematical,Aro1950theory}. It describes how the space $\H$ is related to the image of operator $\Td^{1/2}$.

\begin{Prop}[Isometry for Mercer kernels]\label{prop:isometrymercer}
Under assumptions \textbf{(A1,2)}, $\H=\Td^{1/2}\(\Ld\right)$ and $\Td^{1/2}: S \rightarrow \H$ is an isometrical isomorphism.
\end{Prop}

The proposition has the following consequences: \vspace{-0.2em} 
\begin{Cor} \label{cor_struct_main} Assume \textbf{(A1, A2)}:
\begin{itemize}
\item[--] For any $ r\geq 1/2,\  \Td^{r}(S)\subset \H$, because  $\Td^r(S) \subset \Td^{1/2}(S)$, that is, with large enough powers $r$, the image of $\Td^r$ is in the Hilbert space.
\item[--] $\forall r >0, \ \overline{\Td^{r} (\Ld)} = S= \overline{\Td^{1/2} (\Ld)}= \overline{\H}$, because (a) $ \Td^{1/2}(\Ld) ={\H}$ and (b) for any $r>0$, $\overline{\Td^r(\Ld)} = S$. In other words, elements of $ \overline{\H}$ (on which our minimization problem attains its minimum), may seen as limits (in $\Ld$) of elements of $\Td^{r}( \Ld)$, for any $r>0$.
\item[--]  \emph{$\H$ is dense in $\Ld $ if and only if $\Td$ is injective (which is equivalent to $\ker(\Td) = \{0 \}$)}
\end{itemize}
\end{Cor}

 The sequence of spaces $\lbrace \Td^r(\Ld)\rbrace_{r>0}$ is thus a decreasing (when $r$ is increasing) sequence of subspaces of $\Ld$ such that any of them is dense in $\overline{\H}$, and $\Td^{r}(\Ld)\subset \H$ if and only if $r\geq 1/2.$

In the following, the regularity of the function $g_\H$ will be characterized by the fact that $g_\H$ belongs to the space $\Td^r(\Ld)$ (and not only to its closure), for a specific $r>0$ (see Section~\ref{subsec:assumptions}). This space may be described depending on the eigenvalues and eigenvectors as $$\Td^r(\Ld) = \left\lbrace \sum_{i=1}^{\infty} b_i \phi_i  \mbox{ such that }   \sum_{i=1}^{\infty} \frac{b_i^2}{\mu_i^{2r}} <\infty \right\rbrace.$$
We may thus see the spaces $\Td^r(\Ld)$ as spaces of sequences with various decay conditions.

\subsection{Minimal assumptions}
	\label{Sec:minass}

In this section, we describe under which ``minimal'' assumptions the analysis may be carried. We prove that  the set $\mathcal{X}$  may only be assumed to be equipped with a measure, the kernel $K$ may only assumed to have bounded expectation $\E_\rho K(X,X)$ and the output $Y$  may only be assumed to have finite variance. That is:
 \begin{enumerate}
\item[\textbf{(A1')}]  $\mathcal{H}$ is a separable RKHS associated with kernel $K$ on  the set $\mathcal{X}$.
 \item[\textbf{(A2')}]    $ \E \left[K(X,X)\right] $ and $ \E [Y^2] $ are finite.
\end{enumerate}
   
In this section, we have to distinguish the set of square $\rho_X$-integrable functions~$\L$ and its quotient $\Ld$  that makes it a separable Hilbert space. We define $p$ the projection from $\L$ into $\Ld$ (precise definitions are given in Appendix~\ref{App:rkhsnoproof}). Indeed it is no more possible to identify the space $\H$, which is a subset of $\L$, and its canonical projection $p(\H)$ in $\Ld$.

\emph{Minimality:} 
 The separability assumption is necessary to be able to expand any element as an infinite sum, using a countable orthonormal family (this assumption is satisfied in almost all cases, for instance it is simple as soon as $\X$ admits a topology for which it is separable and functions in $\H$ are continuous, see~\citep{berlinet2004reproducing} for more details).
 Note that we do not make any topological assumptions regarding the set~$\mathcal{X}$. We only assume that it is equipped with a probability measure.

 Assumption \textbf{(A2')} is needed to ensure that every function in $\mathcal{H}$ is square-integrable, that is,  $\E [K(X,X)] < \infty  $ if and only if  $ \H \subset \L$; for example, for $f = K_z$, $z \in \mathcal{X}$, $\| K_z \|^2_{\Ld} = \E [ K(X,z)^2 ]  \leqslant K(z,z) \E K(X,X)$ (see more details in the Appendix~\refa{A_rkhs}, Proposition~\refawc{prop:inclusionRKHSL2}).

Our assumptions are sufficient to analyze the minimization of $\epsilon(f)$ with respect to $f \in \mathcal{H}$ and seem to allow the widest generality.

\emph{Comparison:}
These assumptions will include the previous setting, but also recover  measures without full support (e.g., when the data lives in a small subspace of the whole space) and kernels on discrete objects (with non-finite cardinality). 

Moreover, \textbf{(A1'),  \textbf{(A2')}} are stricly weeker than  \textbf{(A1)},  \textbf{(A2)}. 
In previous work, \textbf{(A2')} was sometimes replaced by the stronger assumptions $\sup_{x\in \X} K(x,x)<\infty$  \cite{ros2014regularisation,yin2008online,tar2011online} and $|Y|$ bounded \cite{ros2014regularisation,tar2011online}.  Note that in functional analysis, the weaker hypothesis  $\int_{\X\!\times\! \X} k(x,x')^2 d{\rho_X}(x) d\rho_X(x') < \infty$ is often used~\cite{bre1983analyse}, but it is not adapted to the statistical setting.

\emph{Main differences:}
The main difference here is that we cannot identify $\H$ and $p(\H)$:  there may exist functions  $f\in \H \setminus \lbrace 0\rbrace$ such that $\| f\| _{\L} =0$. 
This may for example occur if the support   of $\rho_X$ is strictly included in $\mathcal{X}$, and $f$ is zero on this support, but not identically zero. 
See the Appendix~\refawc{subsec:strongass} for more details. 

As a consequence, $\Sigma$ is no more injective and we do not have $\Image(T^{1/2})=\H$ any more. We thus denote $\S$  an  orthogonal supplement of the null space $\text{Ker}(\Sigma)$.
As we also need to be careful not to confuse $\L$ and $\Ld$, we define an extension $\T$ of $\Sigma$ from $\L$ into $\H$, then $\Td = p\circ \T$. We can define for $r\geq 1/2$ the power operator $\T^{r}$ of $\T$ (from $\Ld$ into $\H$), see App.~\ref{App:rkhsnoproof} for details.  
 
\emph{Conclusion:}
Our problem has the same behaviour under such assumptions. Proposition~\ref{propTd}  remains unchanged. Decompositions in Prop.~\ref{prop:decsigmainjec} and Corollary~\ref{cor_struct_main} must be slightly adapted (see Proposition~\ref{prop:decsigmamain} and Corollary~\ref{cor_struct} in Appendix~\ref{App:rkhsnoproof} for details). Finally, Proposition~\ref{prop:isometrymercer}   is generalized by the next proposition, which states that $p(\S)  = p (\H)$ and thus $S$ and   $p(\H)$ are isomorphic
  (see proof in  Appendix~\refa{subsec:app_covar}, Proposition~\refawc{propTdr}):

\begin{Prop}[Isometry between supplements]
\label{isomor}
$\T^{1/2}: S \rightarrow \S$ is an isometry. Moreover, $\Image(\Td^{1/2})=p(\H)$ and  $\Td^{1/2}: S \rightarrow p(\H)$ is an isomorphism.
\end{Prop}

We can also derive a version of Mercer's theorem, which does not make any more assumptions that are required for defining RKHSs. As we will not use it in this article, this proposition is only given in Appendix~\ref{App:rkhsnoproof}.

\emph{Convergence results:} In all convergence results stated below, \emph{assumptions \textbf{(A1, A2)} may be replaced by assumptions \textbf{(A1', A2')}}.

\subsection{Examples}
\label{sec:examples}
The property $\overline{\H}=S$, stated after Proposition~\ref{prop:isometrymercer}, is important to understand what the space $\overline{\H}$ is, as we are minimizing over this closed and convex set. As a consequence the  space $\H$ is dense in $\Ld$ if and only if $\Td$ is injective (or equivalently, ${\rm Ker}(\Td) = \{0\} \Leftrightarrow  S = \Ld$). We detail below a few classical situations in which different configurations for the ``inclusion'' $ \H \subset\overline{\H}\subset \Ld$ appear:

\begin{enumerate}
\item \textbf{Finite-dimensional setting with linear kernel}: 
in finite dimension, with $\X = \mathbb{R}^d$ and $K(x,y) = x^\top y$, we have $\H=\R^d$, with the scalar product in $\langle u,v \rangle_\H=\sum_{i=1}^{d} u_i v_i$. This corresponds to usual parametric least-squares regression.
If the support of $\rho_X$ has non-empty interior, then $\overline{\H}=\H$:  $g_\H$ is the best linear estimator. Moreover, we have $\H=\overline{\H} \varsubsetneq \Ld$  : indeed ${\rm Ker}(\Td)$ is the set of functions such that $\E X f(X)  = 0$ (which is a large space).

\item \textbf{Translation-invariant kernels} for instance the Gaussian kernel over $\X = \R^d$, with $X$ following a  distribution with full support in $\mathbb{R}^d$: in such a situation we have $\H \varsubsetneq\overline{\H} = \Ld$. This last equality holds more generally for all universal kernels, which include all kernels of the form $K(x,y) = q(x-y)$ where $q$ has a summable strictly positive Fourier transform~\cite{micchelli2006universal,sriperumbudur2011universality}. These kernels are exactly the kernels such that $\Td$ is an injective endomorphism of $\Ld$. 

\item \textbf{Splines over the circle:} When $X\sim \mathcal{U} [0;1]$ and $\H$ is the set of $m$-times periodic weakly differentiable functions (see Section~\ref{sec:experiments}), we have in general $\H \varsubsetneq\overline{\H} \varsubsetneq \Ld$. In such a case, $\ker(\Td)= \Span(x\mapsto 1) $, and $\overline{\H} \oplus \Span(x\mapsto 1) =\Ld $, that is we can approximate any zero-mean function.
\end{enumerate}

Many examples and more details may be found in \cite{Cristianini2004,Aro1950theory,ver2014kernel}. In particular, kernels on non-vectorial objects may be defined (e.g., sequences, graphs or measures).

\subsection{Convergence rates}\label{subsec:assumptions}

In order to be able to establish rates of convergence in this infinite-dimensional setting, we have to make assumptions on the objective function and on the covariance operator eigenvalues. In order to account for all cases (finite and infinite dimensions), we now consider eigenvalues ordered in \emph{non-increasing} order, that is, we assume that the set $I$ is either $\{1,\dots,d\}$ if the underlying space is $d$-dimensional or $\mathbb{N}^\ast$ if the underlying space has infinite dimension. 

\begin{enumerate}
\item[\textbf{(A3)}]
We denote $(\mu_i)_{i\in I}$ the sequence of non-zero eigenvalues of the operator~$\Td$, in decreasing order. We assume $\mu_i \le \frac{s^2}{i^\alpha}$  for some $\alpha > 1$ (so that $\tr (\Td) < \infty$), with $s\in \mathbb{R}_+$. 
\item[\textbf{(A4)}] $ g_\H \in  \Td^{r}\left(\Ld\right) $ with $  r \geqslant 0  $, and as a consequence $ \| \Td^{-r}(g_\H) \| _{\Ld}< \infty$.
\end{enumerate}

We chose such assumptions in order to make the comparison with the existing literature as easy as possible, for example \cite{cap2007optimal,yin2008online}. However, some other assumptions may be found as in \cite{bac2012sharp,hsu2014random}. 
\vspace{1em}

  \paragraph{Dependence on $\alpha$ and $r$}
  The two parameters $r$ and $\alpha$ intuitively parametrize the strengths of our assumptions:
  
\begin{itemize}
\item[--] In assumption \textbf{(A3)} a bigger $\alpha$  makes the assumption stronger: it means the reproducing kernel Hilbert space is smaller, that is
if \textbf{(A3)} holds with some constant $\alpha$, then it also holds for any $\alpha' < \alpha$. Moreover, if $\Td$ is reduced in the Hilbertian basis $(\phi_i)_i$ of $\Ld$, we have an effective search space $S = \big\{ \sum_{i=1}^{\infty} b_i \phi_i  /  \sum_{i=1}^{\infty} \frac{b_i^2}{\mu_i} <\infty \big\}$
: the smaller the eigenvalues, the smaller the space. Note that since $\tr (\Td)$ is finite, \textbf{(A3)} is always true for $\alpha=1$.

\item[--] In assumption \textbf{(A4)}, for a fixed $\alpha$,  a bigger $r$ makes the assumption stronger, that is the function $g_\H$ is actually smoother. Indeed, considering that \textbf{(A4)} may be rewritten $  g_\H \in \Td^{r} \big(\Ld \big)$ and  for any $  r <r' , \  \Td^{r'} \big(\Ld \big) \subset  \Td^{r} \big(\Ld \big) $. In other words, $\big\{ \Td^{r} \left(\Ld \right) \big\}_{r\ge 0}$ are decreasing ($r$ growing) subspaces of $\Ld $. 

For $r=1/2$, $\Td^{1/2} \big(\Ld \big)= \H $; moreover, for $r \geqslant 1/2$, our best approximation function $g_\H \in \overline{\H}$ is in fact in $\H$, that is the optimization problem in the RKHS $\H$ is attained by a function of finite norm. However for $r<1/2$ it is not attained.

\item[--]
Furthermore, it is worth pointing the stronger assumption which is often used in the finite dimensional context, namely
$ \tr \( \Sigma^{1/\alpha}\right)  = \sum_{i \in I} \mu_i^{1/\alpha}$ finite. It turns out that this is a stronger assumption, indeed, since we have assumed that the eigenvalues $(\mu_i)$ are arranged in non-increasing order, if $ \tr \( \Sigma^{1/\alpha}\right) $ is finite, then \textbf{(A3)} is satisfied for $s^2 =  \big[ 2 \tr \( \Sigma^{1/\alpha}\right) \big]^\alpha$.
Such an assumption appears for example in Corollary~\ref{cor:actrace}. 
\end{itemize}

 \paragraph{Related assumptions} The assumptions \textbf{(A3)} and \textbf{(A4)} are adapted to our theoretical results, but some stricter assumptions are often used, that make comparison with existing work more direct. For comparison purposes, we will also use:

\begin{enumerate}
\item[\textbf{(a3)}]
For any $i \in I = \mathbb{N}$,
  $u^2 \le i^\alpha \mu_i \le s^2$  for some $\alpha > 1$ and $u,s \in \mathbb{R}_+$.
\item[\textbf{(a4)}] We assume the coordinates $ (\nu_i)_{i \in \mathbb{N}}$  of $g_\H \in \Ld$ in the eigenbasis $(\phi_i)_{i \in \mathbb{N}}$ (for $\| .\| _{\Ld}$) of $\Td$ are such that $  \nu_i  i^{\delta/2} \le W$, for some $\delta>1$  and $W \in \mathbb{R}_+$ (so that $\| g_\H\| _{\Ld}< \infty$).
\end{enumerate}
 
 Assumption \textbf{(a3)} directly imposes that the eigenvalues of $\Td$ decay at rate $i^{-\alpha}$ (which imposes that there are infinitely many), and thus implies~\textbf{(A3)}. Together, assumptions \textbf{(a3)} and \textbf{(a4)}, imply assumptions \textbf{(A3)} and \textbf{(A4)}, with  any $  \delta >1 +2 \alpha r$. Indeed, we have
 $$ \| \Td^{-r} g_\H \|_{\Ld}^2 = \sum_{i \in \mathbb{N}} \nu_i^2 \mu_i^{-2r} \leqslant \frac{W^2}{ u^{4r} } \sum_{i \in \mathbb{N}} i^{-\delta + 2 \alpha r},$$ which is finite for $2 \alpha r - \delta < -1$. Thus, the supremum element of the set of~$r$ such that $\textbf{(A4)}$ holds is such that $\delta = 1 +2 \alpha r$. Thus, when comparing assumptions \textbf{(A3-4)} and \textbf{(a3-4)}, we will often make the identification above, that is, $\delta = 1 +2 \alpha r$. 
 
  The main advantage of the new assumptions is their interpretation when the basis $(\phi_i)_{i \in I}$ is common for several RKHSs (such as the Fourier basis for splines, see  Section~\ref{sec:experiments}): \textbf{(a4)} describes the decrease of the coordinates of the best function $g_\H \in \Ld$ \emph{independently of the chosen RKHS}. Thus, the parameter $\delta$ characterizes the prediction function, while the parameter~$\alpha$ characterizes the RKHS.

\section{Stochastic approximation in Hilbert spaces} 
\label{sec:lms}

In this section, we consider estimating a  prediction function $g \in \H$ from observed data, and we make the following assumption:
\begin{enumerate}
\item[\textbf{(A5)}] For $n \geqslant 1$, the random variables $(x_n,y_n) \in \X \times \R$ are independent and identically distributed with distribution $\rho$.
\end{enumerate}
Our goal is to estimate a function $g \in \H$ from data, such that $\varepsilon(g)  =\E (Y - g(X))^2$ is as small as possible. As shown in Section~\ref{sec:rkhs}, this is equivalent to minimizing $\| g - g_\H\|_{\Ld}^2$. The two main approaches to define an estimator is by regularization or by stochastic approximation (and combinations thereof). See also approaches by early-stopped gradient descent on the empirical risk in~\cite{yao2007early}.

\subsection{Regularization and linear systems}
Given $n$ observations, regularized empirical risk minimization corresponds to minimizing with respect to $g \in \H$ the following objective function:
$$
\frac{1}{n} \sum_{i=1}^n ( y_i - g(x_i) )^2 + {\lambda} \| g\|_\H^2.
$$
Although the problem is formulated in a potentially infinite-dimensional Hilbert space, through the classical representer theorem~\cite{smola-book,Cristianini2004,kimeldorf1971some}, the unique (if $\lambda >0$) optimal solution may be expressed as $\hat{g} = \sum_{i=1}^n a_i K_{x_i}$, and $a \in \mathbb{R}^n$ may be obtained by solving the linear system
$(\mathbf{K} + \lambda I) a = \mathbf{y}$, where $\mathbf{K} \in \mathbb{R}^{n \times n}$ is the kernel matrix, a.k.a.~the Gram matrix, composed of pairwise kernel evaluations $\mathbf{K}_{ij} = K(x_i,x_j)$, $i,j = 1,\dots,n$, and $\mathbf{y}$ is the $n$-dimensional vector of all $n$ responses $y_i$, $i=1,\dots,n$.

The running-time complexity to obtain $a \in \R^n$ is typically $O(n^3)$ if no assumptions are made, but several algorithms may be used to lower the complexity and obtain an approximate solution, such as conjugate gradient~\cite{golub83matrix} or column sampling (a.k.a.~Nystr\"om method)~\cite{fot_mahoney,williams2001using,bach2013sharp}.

In terms of convergence rates, assumptions $\textbf{(a3-4)}$ allow to obtain convergence rates that decompose 
$\epsilon(\hat{g}) - \epsilon(g_\H) = \| \hat{g} - g_\H\|_{\Ld}^2$ as the sum of two asymptotic terms~\cite{cap2007optimal,hsu2014random,bach2013sharp}:

\begin{itemize}
\item[--] \emph{Variance term}: $O\big(  {\sigma^2 }{ n^{-1}  \lambda^{-1/\alpha} }\big) $, which is decreasing with $\lambda$, where $\sigma^2$ characterizes the noise variance, for example, in the homoscedastic case (i.i.d.~additive noise), the marginal variance of the noise; see assumption \textbf{(A6)} for the detailed assumption that we need in our stochastic approximation context.

\item[--] \emph{Bias term}: 
$ O \big(  \lambda^{   \min \{ (\delta-1)/\alpha,2\} } \big)$, which is increasing with $\lambda$. Note that the corresponding $r$ from  assumptions $\textbf{(A3-4)}$ is $r = 
(\delta-1)/2\alpha$, and the bias term becomes proportional to  $\lambda^{   \min \{2r,2\} }$.
\end{itemize}

There are then two regimes:
\begin{itemize}
\item[--] \emph{Optimal predictions}: If $r < 1$, then the optimal value of $\lambda$ (that minimizes the sum of two terms and makes them asymptotically equivalent) is proportional to $n^{- \alpha / ( 2 r \alpha + 1)} = n^{- \alpha / \delta }$ and the excess prediction error $\| \hat{g} - g_\H \|_{\Ld}^2 = O\big(n^{-2\alpha r / ( 2\alpha r + 1)} \big)= O\big(n^{-1+1/\delta} \big)$, and the resulting procedure is then ``optimal'' in terms of estimation of $g_\H$ in $\Ld$ (see Section~\ref{sec:links} for details).  
\item[--] \emph{Saturation}: If $r \geqslant 1$, where the optimal value of $\lambda$ (that minimizes the sum of two terms and makes them equivalent) is proportional to $n^{- \alpha / ( 2 \alpha + 1)} $, and the excess prediction error is less than $ O\big(n^{-2\alpha   / ( 2\alpha   + 1)} \big) $, which is suboptimal. Although assumption \textbf{(A4)} is valid for a larger~$r$, the rate is the same than if $r=1$.
\end{itemize}

In this paper, we consider a stochastic approximation framework with improved running-time complexity and similar theoretical behavior than regularized empirical risk minimization, with the advantage of (a) needing a single pass through the data and (b) simple assumptions.

\subsection{Stochastic approximation}\label{subsec:approxstoch}
Using the reproducing property, we have for any $g \in \H$,  $\varepsilon(g) =  \E ( Y - g(X))^2  =\E ( Y - \langle g, K_X \rangle_\H )^2$, with gradient (defined with respect to the dot-product in $\H$)
$\nabla \varepsilon(g) = - 2 \E \big[ ( Y - \langle g, K_X \rangle_\H ) K_X \big].$

Thus, for each pair of observations $(x_n,y_n)$, we have $\nabla \varepsilon(g) = - 2 \E \big[ ( y_n - \langle g, K_{x_n} \rangle_\H ) K_{x_n} \big]$, and thus, the quantity 
$\big[ -  ( y_n - \langle g, K_{x_n} \rangle_\H ) K_{x_n} \big] = 
\big[ -( y_n - g(x_n) \rangle ) K_{x_n} \big]$ is an \emph{unbiased stochastic (half) gradient}. We thus consider the stochastic gradient recursion, in the Hilbert space $\H$,   started from a function $g_0 \in \H$ (taken to be zero in the following):
$$g_n = g_{n-1} - \gamma_n  \big[  y_n -\langle g_{n-1}, K_{x_n} \rangle_\H\big] K_{x_n} 
=   g_{n-1} - \gamma_n  \big[  y_n - g_{n-1}(x_n) \big] K_{x_n}  ,
$$
where $\gamma_n$ is the \emph{step-size}.

We may also apply the recursion using representants. Indeed, if $g_0=0$, which we now assume, then for any $n \geqslant 1$, 
$$g_n = \sum_{i=1}^n a_i K_{x_i},$$
 with the following recursion on the sequence $(a_n)_{n \geqslant 1}$:
$$
  a_n = -\gamma_n(g_{n-1} (x_n) - y_n) = -\gamma_n\left(\sum_{i=1}^{n-1} a_i K(x_n,x_i) - y_n \right).
$$
   We also output the averaged iterate defined as
   \begin{equation}\label{eq:ouroutput}
   \overline{g}_n= \frac{1}{n+1} \sum_{k=0}^n  {g}_k = \frac{1}{n+1} \sum_{i=1}^n \Big( \sum_{j=1}^i a_j \Big) K_{x_i}.
  \end{equation}

\paragraph{Running-time complexity}

The running time complexity is $O(i)$ for iteration $i$---if we assume that kernel evaluations are $O(1)$, and thus $O(n^2)$ after $n$ steps. This is a serious limitation for practical applications. Several authors have considered expanding $g_n$ on a subset of all $(K_{x_i})$, which allows to bring down the complexity of each iteration and obtain an overall linear complexity is $n$~\cite{dekel2005forgetron,bordes2005fast}, but this comes at the expense of not obtaining the sharp generalization errors that we obtain in this paper. Note that when studying regularized least-squares problem (i.e., adding a penalisation term), one has to update every coefficient $ (a_i) _{1\le i\le n} $ at step $n$, while in our situation, only $a_n$ is computed at step $n$.

 \paragraph{Relationship to previous works} Similar algorithms have been studied before \cite{ros2014regularisation,yin2008online,kiv2004online,yao2006dynamic,zha2004solving}, under various forms. Especially, in \cite{tar2011online,kiv2004online,yao2006dynamic,zha2004solving} a regularization term is added to the loss function (thus considering the following problem: $ \arg\min_{f\in \H} \epsilon(f)+ \lambda ||f||_{K}^{2}$). In \cite{ros2014regularisation,yin2008online}, neither regularization nor averaging procedure are considered, but in the second case, multiple pass through the data are considered. In~\cite{zha2004solving}, a non-regularized averaged procedure equivalent to ours is considered. However, the step-sizes $\gamma_n$ which are proposed, as well as the corresponding analysis, are different. Our step-sizes are larger and our analysis uses more directly the underlying linear algebra to obtain better rates (while the proof of~\cite{zha2004solving} is applicable to all smooth losses).

\paragraph{Step-sizes} We are mainly interested in two different types of step-sizes (a.k.a. \emph{learning rates}): the sequence $(\gamma_i)_{1\le i\le n}$ may be either:
\begin{enumerate}
\item a subsequence of a universal sequence $(\gamma_i)_{i\in \mathbb{N}}$, we refer to this situation as the \emph{``online setting''}. Our bounds then hold for any of the iterates. 

\item  a sequence of the type $\gamma_i= \Gamma(n) $ for  $i \le n $, which will be referred to as the \emph{``finite horizon setting''}: in this situation the  number of samples is assumed to be known and fixed and we chose a constant step-size which may depend on this number. Our bound then hold only for the last iterate.
\end{enumerate}
In practice it is important to have an online procedure, to be able to deal with huge amounts of data (potentially infinite). However, the analysis is easier in the ``finite horizon'' setting. Some \emph{doubling tricks} allow to pass to varying steps \cite{haz2011beyond}, but it is not fully satisfactory in practice as it creates jumps at every $n$ which is a power of two.

  \subsection{Extra regularity assumptions}

  We denote by $\Xi=(Y-g_\H(X)) K_{X}$ the residual, a random element of $\H$. We have $ \E\left[\Xi\right]=0 $ but in general we do not have $ \E\left[\Xi|X\right]=0 $ (unless the model of homoscedastic regression is well specified). We make the following extra assumption:
 
  \begin{itemize}
  \item[\textbf{(A6)}] There exists $ \sigma >0 $ such that $ \E \left[ \Xi \otimes \Xi \right] \lec \sigma^2 \Sigma$, where $\lec$ denotes the order between self-adjoint operators.
\end{itemize}
  In other words, for any $f \in \H$, we have
$ \E \big[ (Y-g_\H(X))^2 f(X)^2\big] \leqslant \sigma^2 \E [f(X)^2]$.
  
  In the well specified homoscedastic case, we have that  $(Y-g_\H(X)) $ is independent of $X$ and with $\sigma^2=\E\[(Y-g_\H(X))^2\right]$, $ \E\left[\Xi|X\right]=  \sigma^2 \Sigma$ is clear: the constant $\sigma^2$ in  the first part of our assumption characterizes the noise amplitude. Moreover when $|Y-g_{\H}(X)|$ is a.s. bounded by $\sigma^{2}$, we have \textbf{(A6)}.

We first present the results in the \emph{finite horizon} setting in Section~\ref{sec:finite} before turning to the \emph{online} setting in Section~\ref{sec:online}.

\subsection{Main results (finite horizon)}
\label{sec:finite}

We can first get some guarantee on the consistency of our estimator, for any small enough constant step-size:
\begin{Th}
\label{theo:cons}
Assume \textbf{(A1-6)}, then for any constant choice $\gamma_n=\gamma_0 < \frac{1}{2R^2}$, the prediction error of $\bar{g}_n$ converges to the one of $g_\H$, that is:
\begin{equation}
 \E \left[ \epsilon\left(\bar{g}_n\right)-\epsilon(g_\H) \right] =\E \| \bar{g}_n -  g_\H \|_{\Ld}^2 \xrightarrow{n\rightarrow\infty} 0.
\end{equation}
\end{Th}
The expectation is considered  with respect to the distribution of the sample $(x_i, y_i)_{1\le i \le n}$, as in all the following theorems (note that $\| \bar{g}_n -  g_\H \|_{\Ld}^2$ is itself a different expectation with respect to  the law $\rho_X$).

Theorem~\ref{theo:cons} means that for the simplest choice of the learning rate as a constant, our estimator tends to the perform as well as the best estimator in the class $\H$. Note that in general, the convergence in $\H$ is meaningless if $r<1/2$. The following results will state some assertions on the speed of such a convergence; our main result, in terms of generality is the following:



\begin{Th}[Complete bound, $\gamma$ constant, finite horizon]\label{prop.dinf.rand} Assume \textbf{(A1-6)} and $\gamma_i=\gamma = \Gamma(n)$, for $1\le i\le n$. If $\gamma R^2 \leqslant 1/4$:
\begin{eqnarray*}
 \E \| \bar{g}_n -  g_\H \|_{\Ld}^{2}  \le  \frac{4 \sigma^2}{ n }\(1+  (s^2 \gamma n )^{\frac{1}{\alpha}}\)  + 4 (1+q_{n,\gamma,s,r})\ \frac{ \| \Td^{-r} g_\H\| _{\Ld}^2 }{\gamma^{2r} n^{  2 \min\{r,1\}} }.  
\end{eqnarray*}
Where  $q_{n,\gamma,s,r}:= (R^{2\alpha} \gamma^{1+\alpha} n s^2)^{\frac{2r-1}{\alpha}}$ if $r\geq \frac{1}{2}$ and  $q_{n,\gamma,s,r}:= 0$ otherwise is a residual quantity.
  \end{Th}

We can make the following observations:
\begin{itemize}

\item[--] \textbf{Proof}: Theorem~\ref{theo:cons} is directly derived from Theorem~\ref{prop.dinf.rand}, which is proved in Appendix~\refawc{subsec:completeproofFH}:
 we derive for our algorithm a new error decomposition and bound the different sources of error via algebraic calculations. More precisely, following the proof in Euclidean space~\cite{bac2013nonstrongly}, we first analyze (in 
 Appendix~\refawc{subsec:semisto_analysis})
  a closely related recursion (we replace $\x$ by its expectation $\Sigma$, and we thus refer to it as a semi-stochastic version of our algorithm): $$
g_n = g_{n-1} - \gamma_n  ( y_n K_{x_n} - \Sigma g_{n-1})  .
$$
 It (a) leads to an easy  computation of the main bias/variance terms of our result, (b) will be used to derive our main result by bounding the drifts between our algorithm and its semi-stochastic version. A more detailed sketch of the proof is given in Appendix~\ref{app_sketch}.

 \item[--] \textbf{Bias/variance interpretation}: The two main terms have a simple interpretation. The first one is a  variance term, which shows the effect of the noise $\sigma^2$ on the error. It is bigger when $\sigma$ gets bigger, and moreover it also gets bigger when $\gamma$ is growing (bigger steps mean more variance). As for the second term, it is a bias term, which accounts for the distance of the initial choice (the null function in general) to the objective function. As a consequence, it is smaller when we make bigger steps.

\item[--] \textbf{Assumption (A4)}: Our assumption \textbf{(A4)} for $r>1$ is stronger than for $r=1$ but we do not improve the bound. Indeed the bias term (see comments below) cannot decrease faster than $O(n^{-2})$: this phenomenon in known as saturation \cite{eng1996saturation}.
To improve our results with $r >1$ it may be interesting to consider another type of averaging. In the following, $r<1$ shall be considered as the main and most interesting case.

\item[--] \textbf{Relationship to regularized empirical risk minimization}: Our bound ends up being very similar to bounds for regularized empirical risk minimization, with the identification $\lambda = \frac{1}{\gamma n}$. It is thus no surprise that once we optimize for the value of $\gamma$, we recover the same rates of convergence. Note that in order to obtain convergence, we require that the step-size $\gamma$ is bounded, which corresponds to an equivalent $\lambda$ which has to be lower-bounded by $1/n$. 

 \item[--] \textbf{Finite horizon}: Once again, this theorem holds in the finite horizon setting. That is we first choose the number of samples we are going to use, then the learning rate as a constant. It allows us to chose $\gamma$ as a function of $n$, in order to balance the main terms in the error bound. The trade-off must be understood as follows: a bigger $\gamma$ increases the effect of the noise, but a smaller one makes it harder to forget the initial condition. 
 
 \end{itemize}

We may now deduce the following corollaries, with specific  optimized values of $\gamma$: 
\begin{Cor}[Optimal constant $\gamma$]\label{Cor_fh}
 Assume \textbf{(A1-6)} and a constant step-size  $\gamma_i= \gamma= \Gamma(n)$, for $1\le i \le n$: 
 \begin{enumerate}
   \item If $ \frac{\alpha -1}{2\alpha}<r  $  and $\Gamma(n)=\gamma_0 \ n^{\frac{-2\alpha \min\{ r, 1\} -1 +\alpha}{2\alpha \min\{r,1\}+1}}$, $\gamma_0 R^2  \leqslant 1/4$, we have:
\begin{equation} \label{eq:rgrand}
  \E \( \| \bar{g}_n -  g_\H \|_{\Ld}^{2}\) \le   A \  {n^{-\frac{2\alpha \min\{r,1\} }{2\alpha \min\{r,1\}+1}}}. 
\end{equation}
with $A =  4 \(1+{(\gamma_0 s^2)^{\frac{1}{\alpha}}}\) \sigma^2  + \frac{4 (1+o(1))}{\gamma_0^{ 2r}} ||L_K^{-r} g_\H||_{\Ld}^2   $.

\item If $ 0<r< \frac{\alpha -1}{2\alpha} $, with $\Gamma(n)= \gamma_0$ is constant, $\gamma_0 R^2  \leqslant 1/4$, we have:
 \begin{equation}\label{eq:rpetit}
 \E \( \| \bar{g}_n -  g_\H \|_{\Ld}^{2}\)\le A \ {n^{-2r}}, 
\end{equation}
with the same constant $A$.
   \end{enumerate}  

\end{Cor}

We can make the following observations:

\begin{itemize}
\item[--] \textbf{Limit conditions:}   Assumption \textbf{(A4)}, gives us some kind of ``position'' of the objective function with respect to our reproducing kernel Hilbert space. If $r\geq 1/2$ then $g_\H \in \H $. That means the regression function truly lies in the space in which we are looking for an approximation. However, it is not necessary neither to get the convergence result, which stands for any $r>0$,  nor to get the optimal rate (see definition in Section~\ref{subsec:optimalrates}), which is also true for $\frac{\alpha -1}{2\alpha}<r<1$ . 
\item[--] \textbf{Evolution with $r$ and $\alpha$:} As it has been noticed above, a bigger $\alpha$ or $r$ would be a stronger assumption. It is thus natural to get a rate which improves with a bigger $\alpha $ or $r$: the function $(\alpha, r) \mapsto \frac{2\alpha r}{2\alpha r+1}$ is increasing in both parameters.

\item[--] The quantity $o(1)$ in Equation~\eqref{eq:rgrand} stands for $(\gamma_0 s^2 n^{-2\alpha^2 r+1})^{\frac{2r-1}{\alpha}}$ if $r\geq 1/2$ (0 otherwise) and is a quantity which decays to 0.

\item[--] \textbf{Different regions:}  in Figure~\ref{fig:regionHF}, we plot in the plan of coordinates $\alpha, \delta$ (with $\delta = 2 \alpha r + 1$) our limit conditions concerning our assumptions, that is,
$r=1 \Leftrightarrow \delta = 2 \alpha + 1$ and $ \frac{\alpha -1}{2\alpha} = r \Leftrightarrow \alpha = \delta$.  The region between the two green lines is the region for which the optimal rate of estimation is reached. The magenta dashed lines stands for $r=1/2$, which has appeared to be meaningless in our context. 

The region $\alpha \geqslant \delta \Leftrightarrow \frac{\alpha-1}{2\alpha} >  r $ corresponds to a situation where regularized empirical risk minimization would still be optimal, but with a regularization parameter $\lambda$ that decays faster than $1/n$, and thus, our corresponding step-size $\gamma = 1/(n \lambda)$ would not be bounded as a function of $n$. We thus saturate our step-size to a constant and the generalization error is dominated by the bias term.

The region $\alpha \leqslant (\delta-1)/2 \Leftrightarrow  r>1 $ corresponds to a situation where regularized empirical risk minimization  reaches a saturating behaviour. In our stochastic approximation context, the variance term dominates.

\end{itemize}

\subsection{Online setting}
\label{sec:online}
We now consider  the second case when the sequence of step-sizes does not depend on the number of samples we want to use (online setting).

The computation are more tedious in such a situation so that we will only state asymptotic theorems in order to understand the similarities and differences between the finite horizon setting and the online setting, especially in terms of limit conditions.

\begin{Th}[Complete bound, $(\gamma_n)_n$ online]\label{prop.dinf.rand.onl} Assume \textbf{(A1-6)}, assume for any $i$, $\gamma_i= \frac{\gamma_0}{i^\zeta}$, $\gamma_0 R^2  \leqslant 1/2$ : 
\begin{itemize}
\item[--] If $0< r (1-\zeta) < 1$, if $0< \zeta< \frac{1}{2}$ then 
\begin{equation}
\E \| \bar{g}_n -  g_\H \|_{\Ld}^{2} \le  O\left( \frac{\sigma^2(s^2\gamma_n)^{\frac{1}{\alpha}}}  {n^{1-\frac{1}{\alpha}}} \right) + O \left( \frac{||L_K^{-r} g_\H||_{\Ld}^2}{(n\gamma_n)^{2r}} \right) .
\end{equation}

\item[--] If $0< r (1-\zeta)  < 1$, $ \frac{1}{2}< \zeta$
\begin{equation}
\E \| \bar{g}_n -  g_\H \|_{\Ld}^{2}  \le  O\left( \frac{\sigma^2(s^2\gamma_n)^{\frac{1}{\alpha}}}  {n^{1-\frac{1}{\alpha}} }\  \frac{1}{n\gamma_n^{2}} \right)  + O \left( \frac{||L_K^{-r} g_\H||_{\Ld}^2}{(n\gamma_n)^{2r}} \right) .
\end{equation}
\end{itemize}
The constant in the $O(\cdot)$ notations  only depend on    $\gamma_0$ and $\alpha$.
\end{Th}

Theorem~\ref{prop.dinf.rand.onl} is proved in 
Appendix~\refawc{subsec:completeproofONL}. 
In the first case, the main bias and variance terms are the same as in the finite horizon setting, and so is the optimal choice of $\zeta$. However in the second case, the variance term behaviour changes: it does not decrease any more when $\zeta$ increases beyond~$1/2$.  Indeed, in such a case our constant averaging procedure puts to much weight on the first iterates, thus we do not improve the variance bound by making the learning rate decrease faster.  Other type of averaging, as proposed for example in \cite{lac2012simpler}, could help to improve the bound.

Moreover, this extra condition thus changes a bit the regions where we get the optimal rate (see Figure~\ref{fig:regionONL}), and we have the following corollary:

\begin{Cor}[Optimal decreasing $ \gamma_n$]
 Assume \textbf{(A1-6)} (in this corollary, $O(\cdot)$ stands for a constant depending on $\alpha, ||L_K^{-r} g_\H||_{\Ld}, s, \sigma^2, \gamma_0 $ and universal constants): 
 \begin{enumerate}
   \item If $ \frac{\alpha -1}{2\alpha}<r<\frac{2 \alpha -1}{2\alpha} $, with $\gamma_n= \gamma_0 n^{\frac{-2\alpha r -1 +\alpha}{2\alpha r+1}}$ for any $n\geq 1$ we get the rate: 
\begin{equation}
\E  \| \bar{g}_n -  g_\H \|_{\Ld} ^2  = O\left( {n^{-\frac{2\alpha r}{2\alpha r+1}}}  \right). 
\end{equation}

\item If $ \frac{2 \alpha -1}{2\alpha} <r$, with $\gamma_n= \gamma_0 n^{-1/2}$ for any $n\geq 1$, we get the rate: 
\begin{equation}
\E  \| \bar{g}_n -  g_\H \|_{\Ld} ^2  = O\left( {n^{-\frac{2\alpha -1}{2\alpha }}}  \right). 
\end{equation}

\item If $ 0< r<\frac{ \alpha -1}{2\alpha}$, with $\gamma_n= \gamma_0$ for any $n\geq 1$, we get the rate given in~\eqref{eq:rpetit}. Indeed the choice of a constant  learning rate naturally results in an online procedure.
\end{enumerate}
\end{Cor}

This corollary is directly derived from Theorem~\ref{prop.dinf.rand.onl}, balancing the two main terms. 
The only difference with the finite horizon setting is the shrinkage of the optimality region as the condition $r<1$ is replaced by $r< \frac{2 \alpha -1}{2\alpha} < 1$ (see Figure~\ref{fig:regionONL}). In the next section, we relate our results to existing work.

\begin{figure}[h]
 \centering
        \begin{subfigure}[b]{0.5\textwidth}
                \includegraphics[scale=.4]{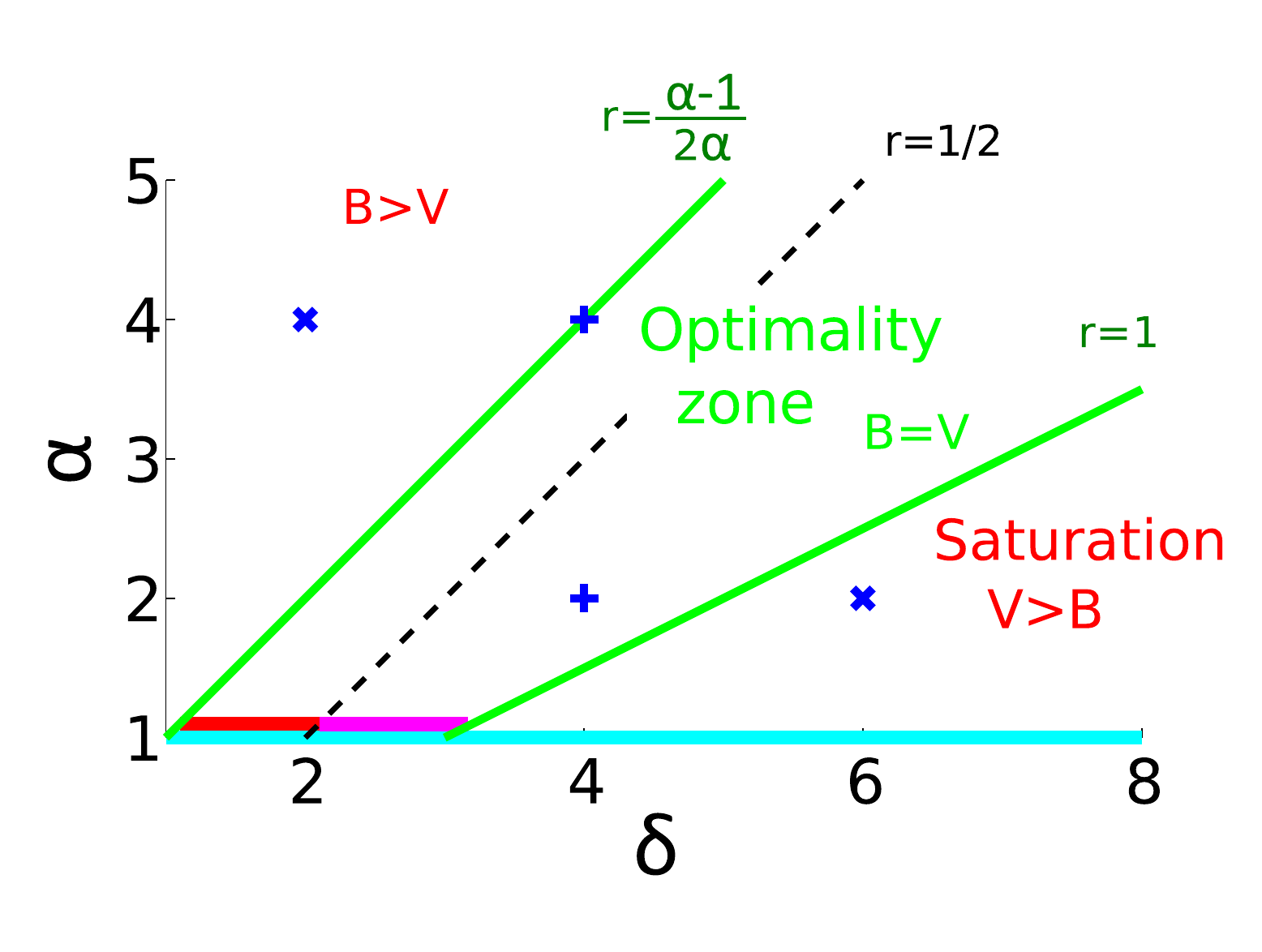}
                \caption{ Finite Horizon }
                \label{fig:regionHF}
        \end{subfigure}%
        ~ 
        \begin{subfigure}[b]{0.5\textwidth}
                \includegraphics[scale=.4]{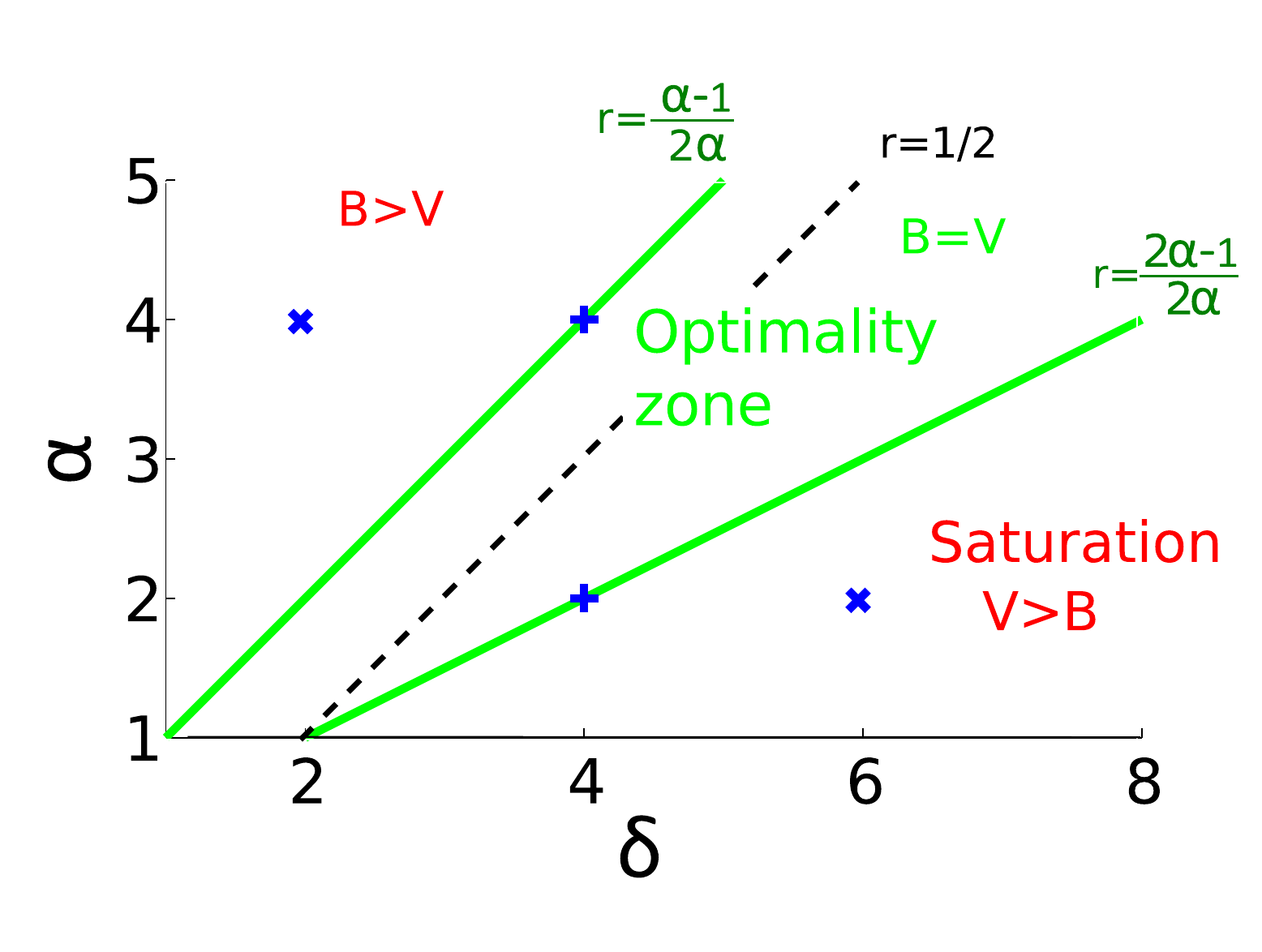}
                \caption{Online}
                \label{fig:regionONL}
        \end{subfigure}
          \caption{Behaviour of convergence rates: (left) finite horizon and (right) online setting. We describe in the $(\alpha, \delta)$ plan (with $\delta=2\alpha r+1$) the different optimality regions : between the two green lines, we achieve the optimal rate. On the left plot the red (respectively magenta and cyan) lines are the regions for which Zhang~\cite{zha2004solving} (respectively Yao \& Tarr\`es~\cite{tar2011online} and Ying \& Pontil~\cite{yin2008online})  proved to achieve the overall optimal rate (which may only be the case if $\alpha=1$).  The four blue points match the coordinates of the four couples $(\alpha,\delta)$ that will be used in our simulations : they are spread over the different optimality regions. }\label{fig:region}
\end{figure}

\section{Links with existing results}
\label{sec:links}

In this section, we relate our results from the previous section to existing results.

\subsection{Euclidean spaces}\label{subsec:euclidien}
Recently, Bach and Moulines showed in \cite{bac2013nonstrongly} that for least squares regression, averaged stochastic gradient descent achieved a rate of $O(1/n)$, in a finite-dimensional Hilbert space (Euclidean space), under the same assumptions as above (except the first one of course), which is replaced by:
\begin{enumerate}
\item[\textbf{(A1-f)}] $\mathcal{H}$ is a $d$-dimensional Euclidean space.  
\end{enumerate}
 
 They showed the following result:
\begin{Prop}[Finite-dimensions \cite{bac2013nonstrongly}]\label{th.css.lms}
Assume \textbf{(A1-f), (A2-6)}. Then for 
$ \gamma=\frac{1}{4R^2} $,
 \begin{equation}
 \E \left[ \epsilon\left(\tb{n}\right)-\epsilon(\t{\H}) \right]\le \frac{4}{n} \left[{\sigma \sqrt{d}} + R \| \t{\H}\|_{\H}  \right]^2.\label{th.cc.lms}
 \end{equation} 
\end{Prop}

We show that we can deduce such a result from Theorem~\ref{prop.dinf.rand} (and even with comparable constants). Indeed under \textbf{(A1-f)} we have: 
\begin{itemize}
\item[--] If $\E\left[||x_n||^2\right] \le R^2 $ then $\Sigma\lec R^2 I$ and \textbf{(A3)} is true for any $\alpha \ge 1$ with  $s^2 = R^2 d^\alpha$. Indeed $\lambda_i \le R^2 $ if $i\le d$ and $\lambda_i=0$ if $i > d+1 $ so that for any $\alpha>1, i \in \mathbb{N}^*$, $\lambda_i \le  R^2 \frac{d^\alpha}{i^\alpha}$.
\item[--] As we are in a finite-dimensional space \textbf{(A4)} is true for $r=1/2$ as  $||\Td^{-1/2} g_\H||^2_{\L}=||g_\H||^2_\H$.
\end{itemize}

Under such remarks, the following corollary may be deduced from Theorem~\ref{prop.dinf.rand}:

\begin{Cor}\label{corfd}
Assume \textbf{(A1-f), (A2-6)}, then for any $\alpha>1$, with $\gamma R^2 \le 1/4$:
\begin{equation*} 
 \E \| \bar{g}_n -  g_\H \|_{\Ld}^{2}  \le  \frac{  4  \sigma^2}{ n}\(1+         (R^2 \gamma d^\alpha  n )^{\frac{1}{\alpha}} \) + 4 \frac{ \| g_\H\| _{\H}^2 }{{n\gamma}}.  
\end{equation*}
So that, when $\alpha\rightarrow\infty$,
\begin{equation*} 
\E \left[ \epsilon\left(\tb{n}\right)-\epsilon(\t{\H}) \right]  \le \frac{4 }{n }\(    \sigma\sqrt{d}      +  R \| g_\H\| _{\H}  \frac{ 1 }{\sqrt{ \gamma R^2}}\)^2.  
\end{equation*}
\end{Cor}
This bound is easily comparable to \eqref{th.cc.lms} and shows that our more general analysis has not lost too much. Moreover our learning rate  is proportional to $n^{\frac{-1}{2\alpha+1}}$ with $r=1/2$, so tends to behave like a constant when $\alpha\rightarrow \infty$, which recovers the constant step set-up from \cite{bac2013nonstrongly}.

\vspace{1em}

Moreover, N. Flammarion  proved (Personnal communication, 05/2014), using the same tpye of techniques, that their bound could be extended to:
\begin{equation}\label{extension.css}
\E \left[ \epsilon\left(\tb{n}\right)-\epsilon(\t{\H}) \right]\le 8\frac{\sigma^2 {d}}{n}+  4R^4\frac{\| \Sigma^{-1/2 } \t{\H} \|^2 }{{(\gamma R^2)^2 n^2}} , 
\end{equation}
 a result that may be deduced of the following more general corollaries of our Theorem~\ref{prop.dinf.rand}:

 \begin{Cor}\label{cor:actrace}
Assume \textbf{(A1-f), (A2-6)},  and, for some $q \geqslant -1/2$,   $||\Sigma^{-q } \t{\H}||^2_\H =||\Sigma^{-(q+1/2) } \t{\H}||^2_{\Ld}  < \infty $, then:
$$\E \left[ \epsilon\left(\tb{n}\right)-\epsilon(\t{\H}) \right] \le 16 \frac{ \sigma^2 \tr (\Sigma^{1/\alpha} ) (\gamma n)^{1/\alpha} }{n}  +8   R^{4(q+1/2)}   \frac{ ||\Sigma^{-q } \t{\H}||^2_\H}{(n \gamma R^2)^{2(q+1/2)}}.$$
\end{Cor}
 
 Such a result is derived from Theorem~\ref{prop.dinf.rand} and with the stronger assumption $\tr(\Sigma^{1/\alpha})<\infty$ clearly satisfied in finite dimension, and  with $r=q+1/2$. Note that the result above is true for all values of $\alpha \geqslant 1$ and all $q\geqslant -1/2	$ (for the ones with infinite $||\Sigma^{-(q+1/2) } \t{\H}||^2_{\Ld} $, the statement is trivial). This   shows that we may take the infimum over all possible $\alpha \leqslant 1$ and $q \geqslant 0$, showing adaptivity of the estimator to the spectral decay of $\Sigma$ and the smoothness of the optimal prediction function $g_\H$.
 
 Thus with $\alpha \rightarrow\infty$, we obtain :  
 \begin{Cor} 
Assume \textbf{(A1-f), (A2-6)},  and, for some $q \geqslant -1/2$,   $||\Sigma^{-q } \t{\H}||^2_\H =||\Sigma^{-(q+1/2) } \t{\H}||^2_{\Ld}  < \infty $, then:
$$\E \left[ \epsilon\left(\tb{n}\right)-\epsilon(\t{\H}) \right] \le 16 \frac{ \sigma^2 {d} }{n}  +8   R^{4(q+1/2)}   \frac{ ||\Sigma^{-q } \t{*}||^2_\H}{(n \gamma R^2)^{2(q+1/2)}}.$$
\end{Cor}

 \begin{itemize}
 \item[--] This final result bridges the gap between Proposition~\ref{th.css.lms} ($q=0$), and its extension \eqref{extension.css} ($q=1/2$). The constants 16 and 8 come from the upper bounds $(a+b)^2 \le a^2+b^2$ and $1+1/\sqrt{d}\le 2$ and are thus non optimal.
 \item[--] Moreover, we can also derive from Corollary~\ref{cor:actrace}, with $\alpha=1$, $q=0$, and $\gamma \varpropto n^{-1/2}$, we recover the rate $O(n^{-1/2})$ (where the constant does not depend on the dimension $d$ of the Euclidean space). Such a rate was described, e.g., in \cite{nem2009robust}. 
 \end{itemize}

Note that linking our work to the finite-dimensional setting is made using the fact that our assumption \textbf{(A3)} is true for any $\alpha >1$. 

\subsection{Optimal rates of estimation}
\label{subsec:optimalrates}

In some situations, our stochastic  approximation framework leads to ``optimal'' rates of prediction in the following sense.
In \cite[Theorem 2]{cap2007optimal} a minimax lower bound was given: let $ \mathcal{P} (\alpha,r)\ \ (\alpha>1, r\in[1/2,1])$ be the set of all probability measures $ \rho $
on $ \mathcal{X}\times\mathcal{Y} $, such that:
\begin{itemize}
\item[--]  $|y| \le M_\rho$ almost surely,
\item[--] $\Td^{-r} g_\rho \in \Ld$,
\item[--] the eigenvalues $ (\mu_j) _{j\in \mathbb{N}}$ arranged in a non increasing order, are subject to the decay $ \mu_j=O(j^{-\alpha}) $. 
\end{itemize}
Then the following minimax lower rate stands: 
\begin{equation*}
\liminf_{n\rightarrow \infty} \ \inf_{ g_n} \sup_{\rho \in \mathcal{P}(b,r)} \mathbb{P} \left\lbrace \epsilon(g_n)-\epsilon(g_\rho) >  C n^{-2r\alpha/(2r\alpha+1)} \right\rbrace  =1,
\end{equation*}
for some constant $ C> 0$ where the infimum in the middle is taken over all algorithms as a map $  ((x_i,y_i)_{1\le i\le n}) \mapsto g_n  \in \H  $.

When making assumptions \textbf{(a3-4)}, the assumptions regarding the prediction problem (i.e., the optimal function $g_\rho$) are summarized in the decay of the components of $g_\rho$ in an orthonormal basis, characterized by the constant~$\delta$. Here, the minimax rate of estimation (see, e.g.,~\cite{johnstone1994minimax}) is  $O(n^{-1+1/\delta})$ which is the same  as $O \big( n^{-2r\alpha/(2r\alpha+1)} \big)$ with the identification $\delta = 2 \alpha r + 1$.

That means the rate we get  is optimal for $ \frac{\alpha -1}{2\alpha}<r<1 $ in the finite horizon setting, and for $ \frac{\alpha -1}{2\alpha}<r<\frac{2\alpha-1}{2\alpha}$ in the online setting. This is the region between the two green lines on Figure~\ref{fig:region}.

\subsection{Regularized stochastic approximation}
It is interesting to link our results to what has been done in \cite{yao2006dynamic} and \cite{tar2011online} in the case of regularized least-mean-squares, so that the recursion is written:
\begin{eqnarray*}
g_n&=& g_{n-1}- \gamma_n \( (g_{n-1} (x_n) -y_n) K_{x_n}  +\lambda_n g_{n-1}\)
\end{eqnarray*}
with $(g_{n-1} (x_n) -y_n) K_{x_n}  +\lambda_n g_{n-1}$ an unbiased gradient of  $\frac{1}{2} \E_{\rho}\left[(g(x)- y)^2\right] + \frac{\lambda_n}{2} ||g||^2 $. 
In \cite{tar2011online} the following result is proved (\textit{Remark 2.8} following \textit{Theorem~C}):
\begin{Th}[Regularized, non averaged stochastic gradient\cite{tar2011online}]
Assume that $\Td^{-r} g_\rho \in\Ld$ for some $ r\in [1/2,1]$. Assume the kernel is bounded and $\mathcal{Y}$ compact. Then with probability at least $ 1-\kappa$, for all $ t\in \mathbb{N} $, 
\begin{equation*}
\epsilon(g_n)-\epsilon(g_\rho) \le O_{\kappa}\(n^{-2r/(2r+1)}\). 
\end{equation*}
Where $ O_{\kappa}$ stands for a constant which depends on $\kappa$.

\end{Th}

No assumption is made on the covariance operator beyond being trace class, but only on  $ \| \Td^{-r} g_\rho \|_{\Ld}$ (thus no assumption \textbf{(A3)}).
A few remarks may be made:
\begin{enumerate}
\item They get almost-sure convergence, when we only get convergence in expectation. We could perhaps derive a.s. convergence by considering moment bounds in order to be able to derive convergence in high probability and to use Borel-Cantelli lemma.
\item They only assume $\frac{1}{2}\le r\le 1$, which means that they assume the regression function to lie in the RKHS.
\end{enumerate}

\subsection{Unregularized  stochastic approximation}
In \cite{yin2008online}, Ying and Pontil studied the same unregularized  problem as we consider, under assumption \textbf{(A4)}. They obtain the same rates as above ($n^{-2r/(2r+1)} \log(n)$) in both online case (with $0\le r\le \frac{1}{2}$) and finite horizon setting ($0<r$).

They led as an open problem to improve bounds with some additional information on some decay of the eigenvalues of $\Td$, a question which is answered here.

Moreover, Zhang \cite{zha2004solving} also studies stochastic gradient descent algorithms in an unregularized setting, also with averaging. As described in \cite{yin2008online}, his result is stated in the linear kernel setting but may be extended to kernels satisfying $\sup_{x\in\X} K(x,x) \le R^2$. Ying and Pontil derive from Theorem 5.2 in \cite{zha2004solving}  the following proposition: 
\begin{Prop}[Short step-sizes~\cite{zha2004solving}]
Assume we consider the algorithm defined in Section~\ref{subsec:approxstoch} and output $\overline{g}_n$ defined by equation~\eqref{eq:ouroutput}.  Assume the kernel $K$ satisfies $\sup_{x\in\X} K(x,x) \le R^2$. Finally assume $g_\rho$ satisfies assumption \textbf{(A4)} with $0<r<1/2$. Then in the finite horizon setting, with $\Gamma(n)=\frac{1}{4R^2} n^{-\frac{2r}{2r+1}}$, we have:
$$\E \left[ \epsilon\left(\bar{g}_n\right)-\epsilon(g_\H) \right] = O\left(n^{-\frac{2r}{2r+1}}\right).$$
\end{Prop}

Moreover, note that we may derive their result from Corollary~\ref{Cor_fh}. Indeed, using  $\Gamma(n)=\gamma_0 n^{\frac{-2r}{	2r+1}}$, we get a bias term which is of   order $n^{\frac{-2r}{2r+1}}$ and a variance term of order $n^{-1 + \frac{1}{2r\alpha+\alpha}}$ which is smaller. Our analysis thus recovers their convergence rate with their step-size.
Note that this step-size     is significantly smaller than ours, and that the resulting bound is worse (but their result holds in more general settings than least-squares). See more details in Section~\ref{subsec:summingup}.

\subsection{Summary of results}\label{subsec:summingup}

All three algorithms are variants of the following: 
\begin{eqnarray*}
g_0 &=&0\\
\forall n \geq 1, \quad g_n &=& (1-\lambda_n)g_{n-1} - \gamma_n  ( y_n - g_{n-1}(x_n) ) K_{x_n} .
\end{eqnarray*}

But they are studied under different settings, concerning regularization, averaging, assumptions: we sum up in Table~\ref{tab:descriptionalgo} the settings of each of these studies. For each of them, we consider the finite horizon settings, where results are generally better.

\begin{table}[H]
\makebox[\textwidth][c]{ 
\begin{tabular}{|c|c|c|c|c|c|c|}%
\hline 
Algorithm        &  Ass.         & Ass.	    	& $\gamma_n$ & $\lambda_n$ & Rate & Conditions   \\
type             & \textbf{(A3)} & \textbf{(A4)}&            &             &      &              \\
\hline
                 &               &              &            &             &      &              \\[-.3cm]
This paper    &  yes          & yes          & 1          & 0           & $n^{- 2r}$ & $ r  <  \frac{\alpha-1}{2\alpha}$\\
This paper    &  yes          & yes          & $n^{- \frac{2\alpha r +1 - \alpha }{2\alpha r+1}}$  & 0 & $n^{\frac{-2\alpha r}{2\alpha r+1}}$  & $ \frac{\alpha-1}{2\alpha} < r < 1$\\
This paper    &   yes & yes & $n^{- \frac{\alpha+1}{2\alpha+1}}$ & 0 & $n^{\frac{-2\alpha  }{2\alpha +1}}$ & $r>1$ \\
\hline
& & & & & & \\[-.3cm]
Zhang \cite{zha2004solving} &  no & yes & $n^{\frac{-2r}{2r+1}}$  & 0 & $n^{\frac{-2r}{2r+1}}$  &  $0 \le r \le \frac{1}{2}$ \\
 Tarrès \& Yao \cite{tar2011online}  &  no & yes &  $n^{\frac{-2r}{2r+1}}$ &  $n^{\frac{-1}{2r+1	}}$ & $n^{\frac{-2r}{2r+1}}$  &$\frac{1}{2} \le r \le 1$  \\
 Ying \& Pontil \cite{yin2008online} &  no & yes & $n^{\frac{-2r}{2r+1}}$  & 0 & $n^{\frac{-2r}{2r+1}}$ &  $r>0$  \\
 \hline
\end{tabular}
}
\vspace{0.5em}
\caption{Summary of assumptions and results (step-sizes, rates and conditions) for our three regions of convergence and related approaches. We focus on finite-horizon results.} \label{tab:descriptionalgo}
\end{table}

 We can make the following observations:
\begin{itemize}
 \item[--] \textbf{Dependence of the convergence rate on $\alpha$}: For learning with any kernel with $\alpha>1$ we strictly improve the asymptotic rate compared to related methods that only assume summability of eigenvalues: indeed, the function $x\mapsto x/(x+1)$ is increasing on $\mathbb{R}^+$. If we consider a given optimal prediction function and a given kernel with which we are going to learn the function, considering the decrease in eigenvalues allows to adapt the step-size and obtain an improved learning rate. Namely, we improved the previous rate $\frac{-2r}{2r+1}$ up to $\frac{-2\alpha r}{2\alpha r +1}$.
 \item[--] \textbf{Worst-case result in $r$}: in the setting of assumptions \textbf{(a3,4)}, given~$\delta$, the optimal rate of convergence is known to be $O(n^{-1+1/\delta})$, where $\delta = 2 \alpha r + 1$. We thus get the optimal rate, as soon as  $\alpha< \delta < 2 \alpha+1$, while the 
 other algorithms   get the suboptimal rate $n^{\frac{\delta-1}{\delta+\alpha-1}}$ under various conditions.
 Note that this sub-optimal rate becomes close to the optimal rate when $\alpha$ is close to one, that is, in the \emph{worst-case} situation. Thus, in the worst-case ($\alpha$ arbitrarily close to one), all methods behave similarly, but for any particular instance where $\alpha >1$, our rates are better.
 
  \item[--] \textbf{Choice of kernel}: in the setting of assumptions \textbf{(a3,4)}, given $\delta$, in order to get the optimal rate, we may choose the kernel (i.e., $\alpha$) such that  $\alpha< \delta < 2 \alpha+1$ (that is neither too big, nor too small), while other methods 
need to choose a kernel for which $\alpha $ is as close to one as possible, which may not be possible in practice.

\item[--] \textbf{Improved bounds}:  Ying and Pontil~\cite{yin2008online} only give asymptotic bounds, while we have exact constants for the finite horizon case. Moreover there are some logarithmic terms in~\cite{yin2008online}  which disappear in our analysis.

\item[--] \textbf{Saturation}: our method does saturate for $r>1$, while the non-averaged framework of~\cite{yin2008online} does not (but does not depend on the value of $\alpha$). We conjecture that a proper non-uniform averaging scheme (that puts more weight on the latest iterates), we should get the best of both worlds.
\end{itemize}

\section{Experiments on artificial data} 
\label{sec:experiments}

Following~\cite{yin2008online}, we consider synthetic examples with smoothing splines on the circle, where our assumptions \textbf{(A3-4)} are easily satisfied.

\subsection{Splines on the circle}
The simplest example to match our assumptions may be found in \cite{wah1990splines}. We consider $ Y=g_\rho(X)+\epsilon $, with $ X\sim \mathcal{U} [\,0;1] $ is a uniform random variable in $[0,1]$, and $ g_\rho $ in a particular RKHS  (which is actually a Sobolev space).

Let $\H$ be the collection of all zero-mean periodic functions  on $ [0;1] $ of the form $$ f:t \mapsto \sqrt{2} \sum_{i=1}^\infty a_i(f) \cos(2 \pi i t) +\sqrt{2} \sum_{i=1}^\infty b_i(f) \sin(2 \pi i t) ,$$
with $$ \| f\|_\H^2 = \sum_{i=1}^\infty (a_i(f)^2 +b_i(f)^2)(2\pi i)^{2m}< \infty.$$
This means that the $m$-th derivative of $f$,  $ f^{(m)} $   is in $\mathcal{L}^2( [0\,;1])$.
We consider the inner product: $$\langle f, g \rangle_{\H}= \sum_{i=1}^\infty (2 \pi i)^{2m} \(a_i(f) a_i(g) +b_i(f) b_i(g)\).$$
It is known that $\H$ is an RKHS and that  the reproducing kernel $ R_m(s,t) $ for  $ \H $  is \begin{eqnarray*} 
R_m(s,t) &=&\sum_{i=1}^\infty \frac{2}{(2 \pi i) ^{2m}} [ \cos (2 \pi i s)\cos (2 \pi i t)+\sin (2 \pi i s)\sin (2 \pi i t) ]\\
&=& \sum_{i=1}^\infty \frac{2}{(2 \pi i) ^{2m}}  \cos (2 \pi i (s-t)).
\end{eqnarray*}
 
Moreover  the study of Bernoulli polynomials gives a close formula for~$R(s,t)$, that is:
 $$ R_m(s,t)= \frac{(-1)^{m-1}}{(2m)!} B_{2m}\(\lbrace s-t\rbrace\), $$
with $ B_m $ denoting the m-th Bernoulli polynomial and $ \lbrace s-t\rbrace $ the fractional part of $ s-t $~\cite{wah1990splines}.

We can derive the following proposition for the covariance operator which means that our assumption \textbf{(A3)} is satisfied for our algorithm in  $\H$ when $X\sim \mathcal{U}[0;\,1]$, with $\alpha=2m$, and $s=2(1/2\pi)^m$.

\begin{Prop}[Covariance operator for smoothing splines]
If $X\sim \mathcal{U}[0;\,1]$, then in $\H$:
\begin{enumerate}
 \item the eigenvalues of $\Sigma$ are all of multiplicity 2 and are $  \lambda_i = (2\pi i)^{-2m} $,
 \item the eigenfunctions are $ \phi_i^c:t\mapsto  \sqrt{2} \cos (2 \pi i t)$ and  $ \phi_i^s:t\mapsto \sqrt{2} \sin (2 \pi i t)$.
 \end{enumerate} 
\end{Prop}
\begin{proof}
 For $\phi_i^c$ we have (a similar argument holds for $\phi_i^s$):
\begin{eqnarray*}
\Td (\phi_i^c)(s)&=&\int_{0}^1 R(s,t) \sqrt{2} \cos (2 \pi i t) dt \\
& = &  \(\int_0^1 \frac{2}{(2i\pi)^{2m}} \sqrt{2} \cos (2 \pi i t)^2 dt \)\cos (2 \pi i s) = \lambda_i \sqrt{2} \cos (2 \pi i s) \\
&= & \lambda_i \phi_i^c(s). 
\end{eqnarray*}

It is well known that $ (\phi_i^c, \phi_i^s)_{i\geq 0} $ is an orthonormal system  (the Fourier basis) of the functions in ${L}^2 ([\,0;1] )$ with zero mean, and it is easy to check that $ ((2 i \pi )^{-m} \phi_i^c, (2 i \pi )^{-m} \phi_i^s)_{i\geq 1}$ is an orthonormal basis of our RKHS $\H$ (this may also be seen as a consequence of the fact that $ \Td^{1/2} $ is an isometry).
\end{proof}

 Finally, considering $ g_\rho(x)=B_{\delta/2}(x)$ with $\delta= 2\alpha r+1 \in 2\N $, our assumption \textbf{(A4)} holds.  Indeed it implies \textbf{(a3-4)}, with $\alpha>1, \delta=2\alpha r+1$, since for any $k \in \mathbb{N}$,
 $ \displaystyle
 B_k(x) = -2 k! \sum_{i=1}^\infty \frac{ \cos \big(2i \pi x - \frac{k \pi}{2} \big)}{ (2i\pi)^k}
 $ (see, e.g.,~\cite{abramowitz1964handbook}).

We may notice a few points: 
\begin{enumerate}
\item Here the eigenvectors do not depend on the kernel choice, only the re-normalisation constant depends on the choice of the kernel. Especially the eigenbasis of $\Td$ in $\Ld$ does not depend on $m$. That can be linked with the previous remarks made in Section~\ref{sec:links}.
\item Assumption \textbf{(A3)} defines here the size of the RKHS: the smaller $\alpha = 2m$ is, the bigger the space is, the harder it is to learn a function. 
\end{enumerate}

In the next section, we illustrate on such a toy model our main results and compare our learning algorithm to Ying and Pontil's \cite{yin2008online}, Tarrès and Yao's \cite{tar2011online} and Zhang's \cite{zha2004solving} algorithms.

\subsection{Experimental set-up}

We use  $ g_\rho(x)=  B_{\delta/2}(x)$ with $\delta= 2\alpha r+1$, as proposed above, 
with $B_1(x) =  x-\frac{1}{2}$, $B_2(x) = x^2-x+\frac{1}{6}$ and $B_3(x) = x^3-\frac{3}{2}x^2+\frac{1}{2}x$.

We give in Figure~\ref{fig:diff_noy_et_fonct} the functions used for simulations in a few cases that span our three regions. We also remind the choice of $\gamma$  proposed for the 4 algorithms. We always use the finite horizon setting.

\begin{table}[H]
\begin{center}
\begin{tabular}{|c c c |c c|c|c| }\hline
\rule[-1ex]{0pt}{4ex}$r$ & $\alpha$&  $\delta$  & $K$ & $g_\rho$ & {$\frac{\log(\gamma)}{ \log(n)}$ (this paper)} &  $\frac{\log(\gamma)}{ \log(n)}$ (previous)\\
\hline
\rule[-1ex]{0pt}{4ex}  0.75 & 2 & 4   & $R_1$ & $ B_2$ & $-1/2 = -0.5$  & $-3/5 = -0.6$ \\
 \hline
 \rule[-1ex]{0pt}{4ex} 0.375 & 4 & 4   & $R_2$ & $B_2$ & 0 & $-3/7 \simeq -0.43$ \\
 \hline
 \rule[-1ex]{0pt}{4ex} 1.25 & 2 &  6  & $R_1$ & $B_3$ & $-3/7 \simeq -0.43$ &$-5/7\simeq -0.71$  \\
\hline
 \rule[-1ex]{0pt}{4ex}  0.125 & 4 & 2   & $R_2$ & $B_1$ & $0$ & $-1/5 = -0.2$ \\
 \hline
\end{tabular}

\vspace*{.25cm}

\caption{Different choices of the parameters $\alpha, r$ and the corresponding convergence rates and step-sizes. The $(\alpha, \delta)$ coordinates of the four choices of couple ``(kernel, objective function)'' are mapped on Figure~\ref{fig:region}.  They are spread over the different optimality regions.}
\label{fig:diff_noy_et_fonct}
\end{center}
\end{table}

\subsection{Optimal learning rate for our algorithm}

In this section, we empirically search for the best choice of a finite horizon learning rate, in order to check if it matches our prediction. For a certain number of values for $n$, distributed exponentially between $1$ and $10^{3.5}$, we look for the best choice $\Gamma_{\rm best}(n)$ of a constant learning rate for our algorithm up to horizon $n$.  In order to do that, for a large number of constants $C_1, \cdots, C_p$, we estimate the expectation of error $\E [\epsilon(\overline{g}_n (\gamma=C_i)) -\epsilon(g_\rho)]$ by averaging over 30 independent sample of size $n$, then report the constant giving minimal error as a function of $n$ in Figure~\ref{fig:gammaopt}. We consider   here the situation $\alpha=2, r=0.75$. We plot results in a logarithmic scale, and  evaluate the asymptotic decrease of $\Gamma_{\rm best}(n)$ by fitting an affine approximation to the second half of the curve. We get a slope of $-0.51$, which matches our choice of $-0.5$  from Corollary~\ref{Cor_fh}. Although, our theoretical results are only upper-bounds, we conjecture that our proof technique also leads to lower-bounds in situations where assumptions $(\textbf{a3-4})$ hold (like in this experiment).

\begin{figure}[H]
\includegraphics[width=8cm]{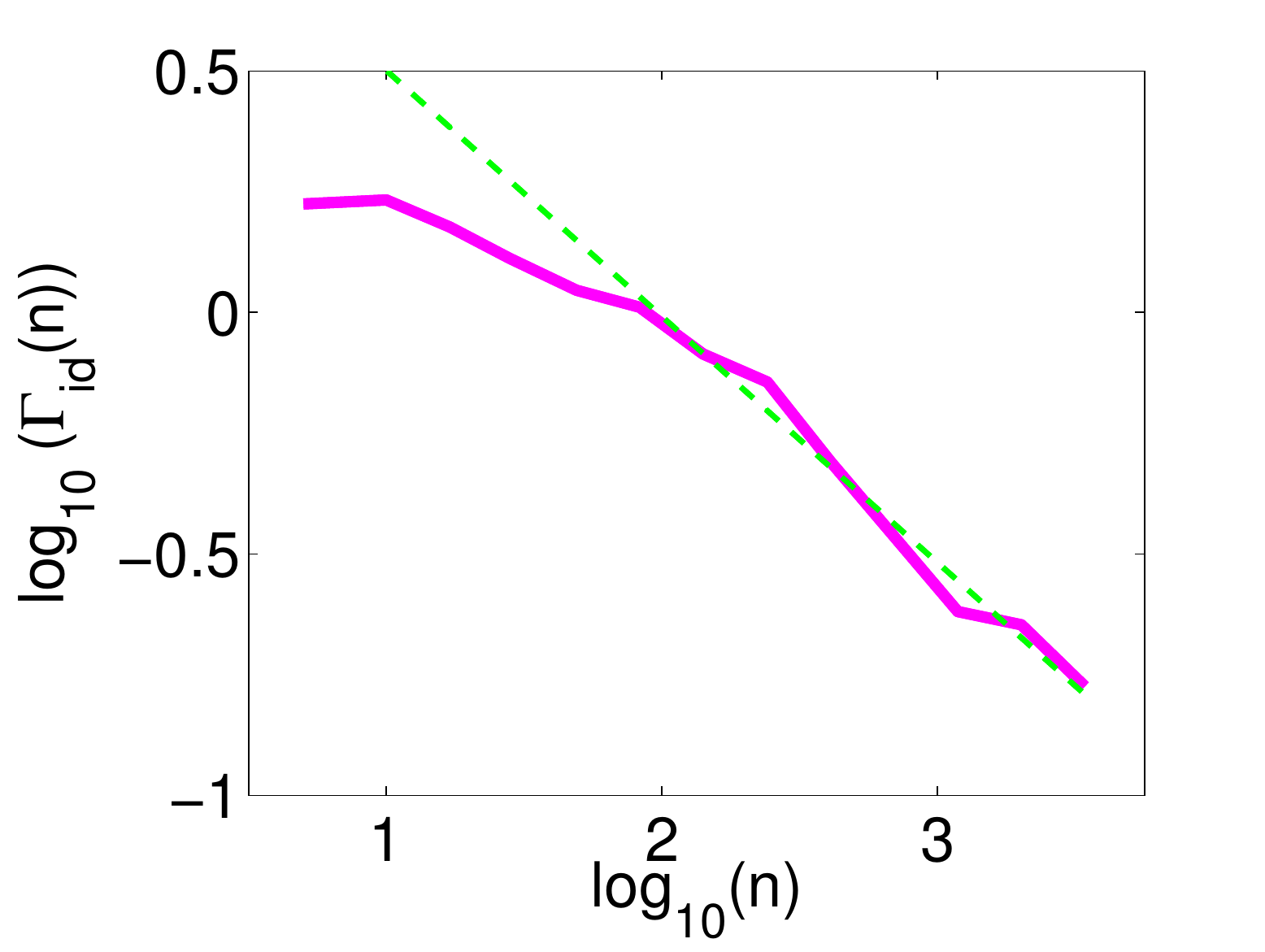}
\caption{Optimal learning rate $\Gamma_{\rm best}(n)$ for our algorithm in the finite horizon setting (plain magenta). The dashed green curve is a first order affine approximation of the second half of the magenta curve.}\label{fig:gammaopt}
\end{figure}

\subsection{Comparison to competing algorithms}
In this section, we compare the convergence rates of the four algorithms described in Section~\ref{subsec:summingup}. We consider the different choices of $(r,\alpha)$ as described in Table~\ref{fig:diff_noy_et_fonct} in order to go all over the different optimality situations. The main properties of each algorithm are described in Table~\ref{tab:descriptionalgo}. However we may note: 
\begin{itemize}
\item[--] For our algorithm, $\Gamma(n)$ is chosen accordingly with Corollary~\ref{Cor_fh}, with $\gamma_0=\frac{1}{R^2}$.  
\item[--] For Ying and Pontil's algorithm, accordingly to Theorem~6 in \cite{yin2008online}, we consider $\Gamma(n)= \gamma_0 n^{-\frac{2r}{	2r+1}}$. We choose $\gamma_0=\frac{1}{R^2}$ which behaves better than the proposed $\frac{r}{64(1+R^4)(2r+1)}$.
\item[--] For Tarr\`es and Yao's algorithm, we refer to Theorem~C in \cite{tar2011online}, and consider $\Gamma(n)=a \({n_0+n}\)^{-\frac{2r}{2r+1}}$ and $\Lambda(n)=\frac{1}{a} \({n_0+n}\)^{-\frac{1}{2r+1}}$. The theorem is stated for all $a\geq4$: we choose $a=4$. 
\item[--] For Zhangl's algorithm, we refer to Part~2.2 in \cite{yin2008online}, and choose $\Gamma(n)=\gamma_0 n^{-\frac{2r}{2r+1}}$ with $\gamma_0=\frac{1}{ R^2}$ which behaves better than the proposed choice $\frac{1}{4(1+R^2)}$.
\end{itemize}

\begin{figure}
\begin{subfigure}[b]{0.35\textwidth}
 \hspace*{-.9cm}\includegraphics[width=6.2cm]{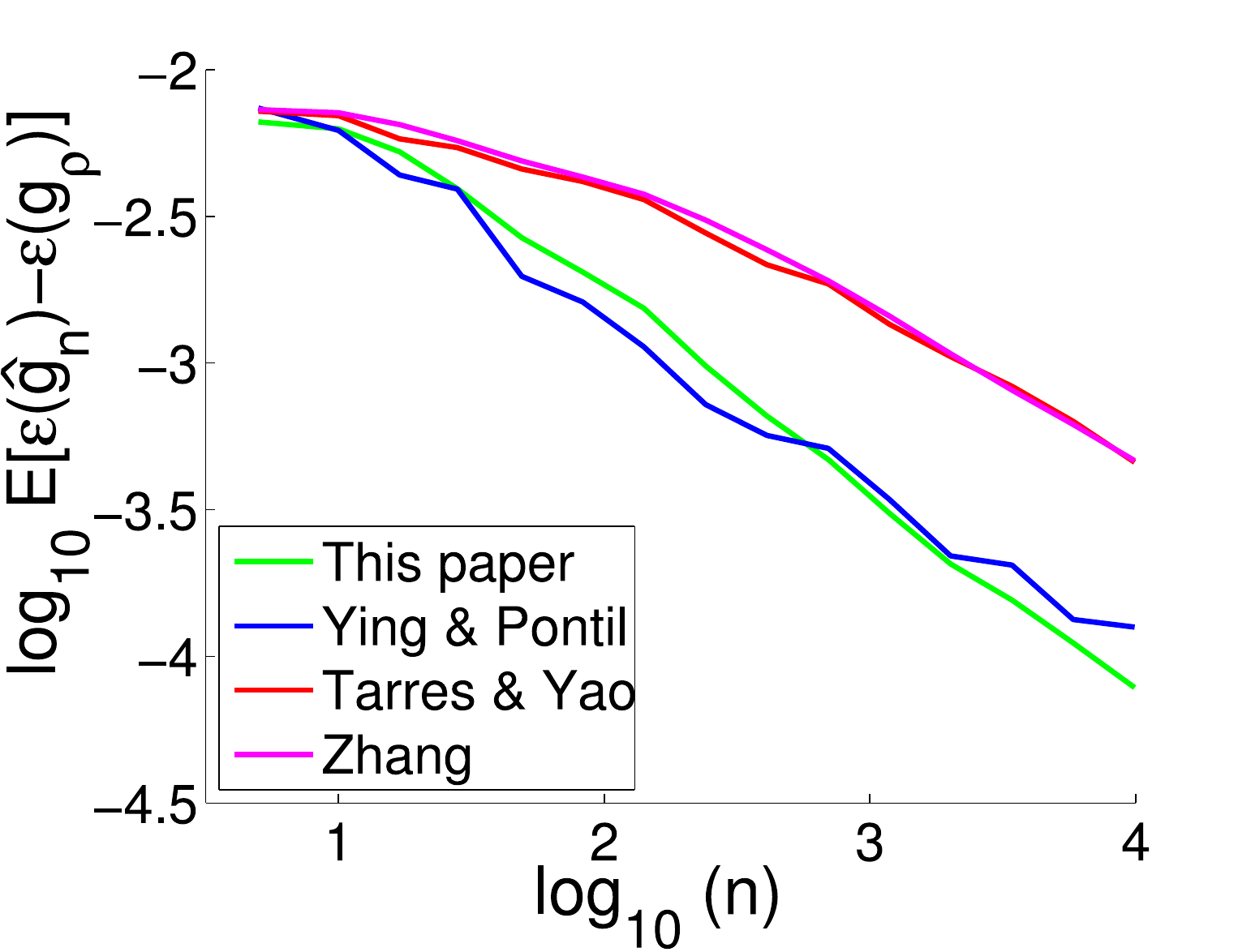}
  \caption{{$r=0.75, \alpha=2$}}
  \end{subfigure}~~~~~~~~~~~~~~\begin{subfigure}[b]{0.35\textwidth}
 \hspace*{-.2cm}  \includegraphics[width=6.2cm]{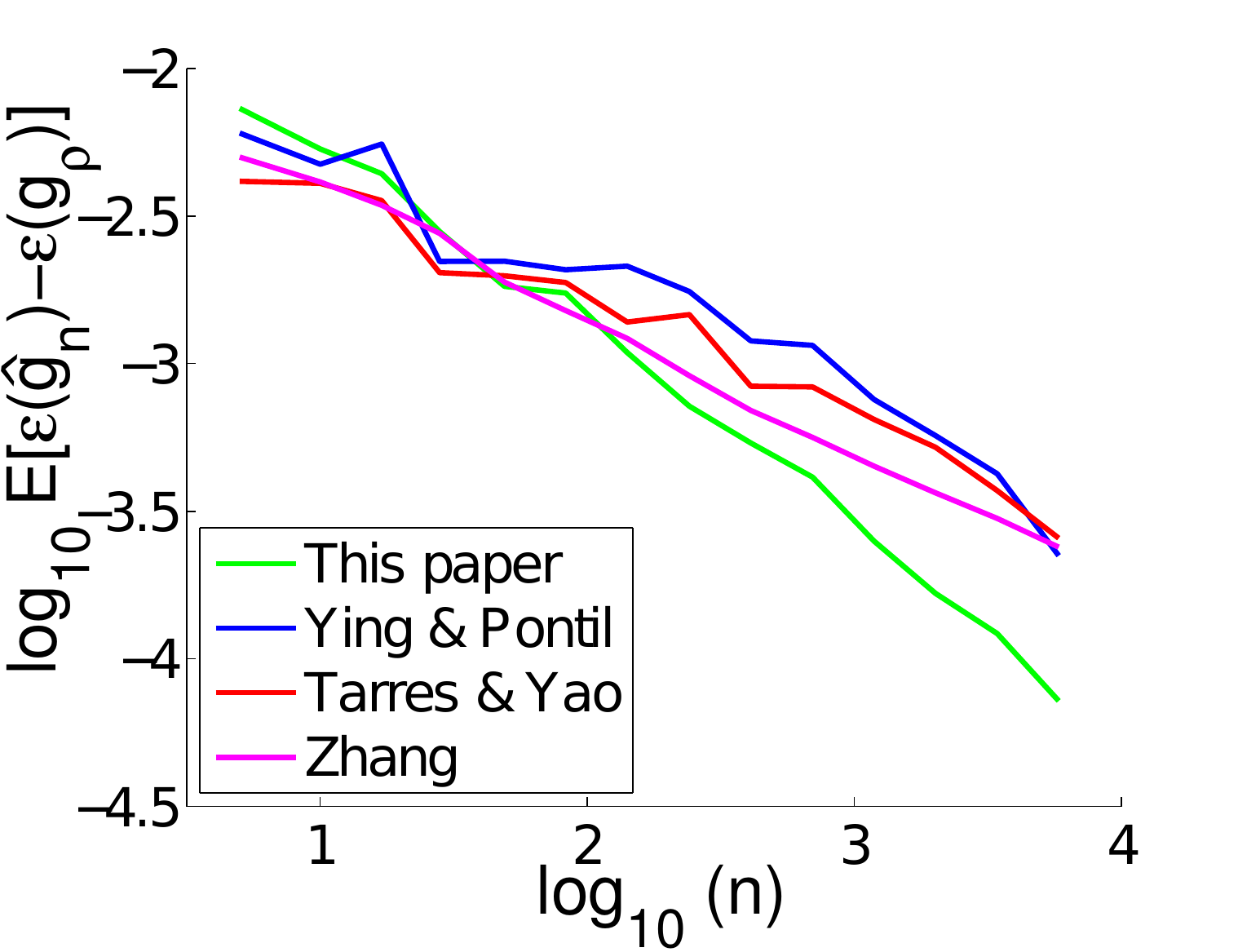}
  \caption{{$r=0.375, \alpha=4$}}
  \end{subfigure}  \\
  \begin{subfigure}[b]{0.35\textwidth}
    \hspace*{-.9cm} \includegraphics[width=6.2cm]{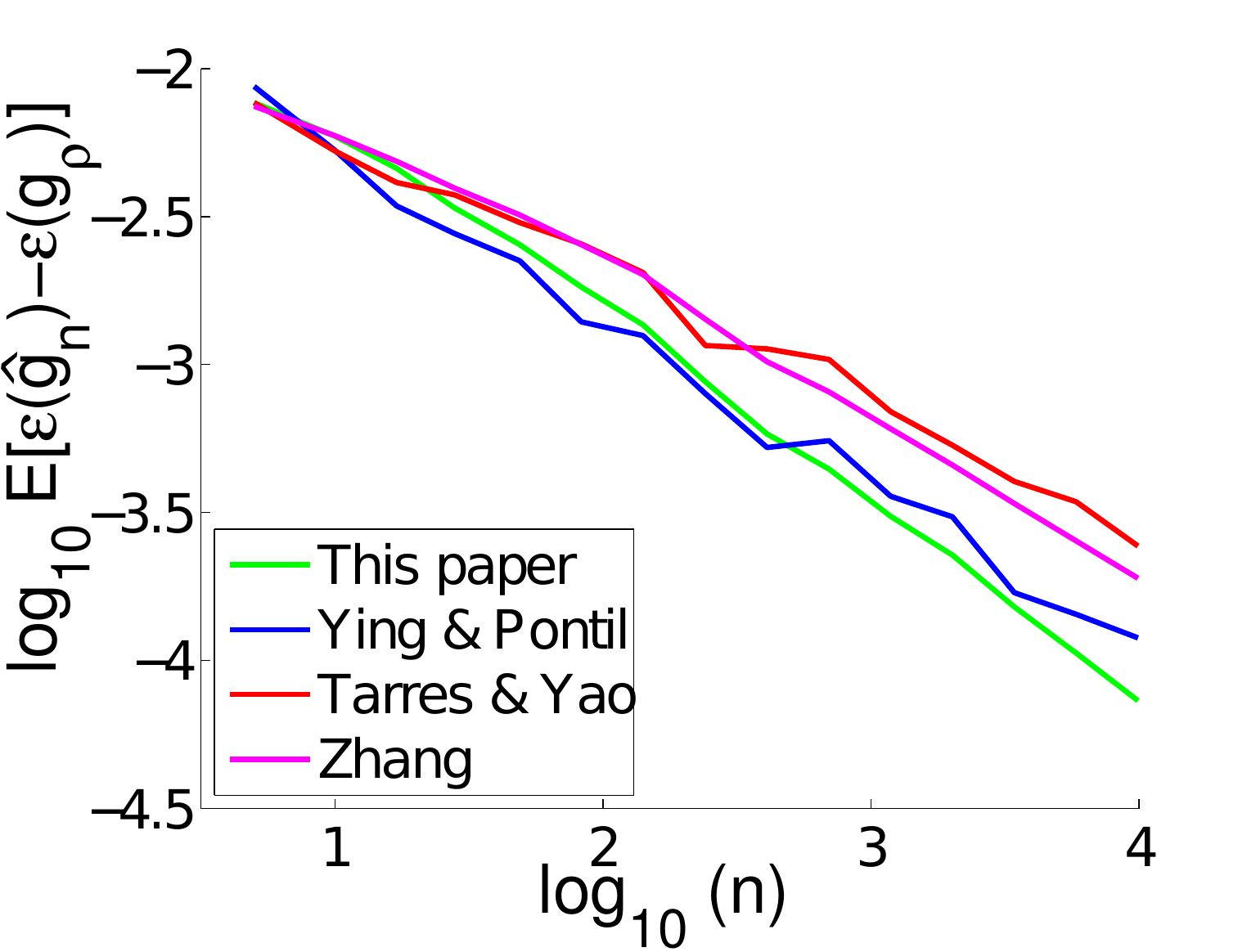}
  \caption{{$r=1.25, \alpha=2$}}
  \end{subfigure}~~~~~~~~~~~~~~
  \begin{subfigure}[b]{0.35\textwidth}
  \hspace*{-.2cm} \includegraphics[width=6.2cm]{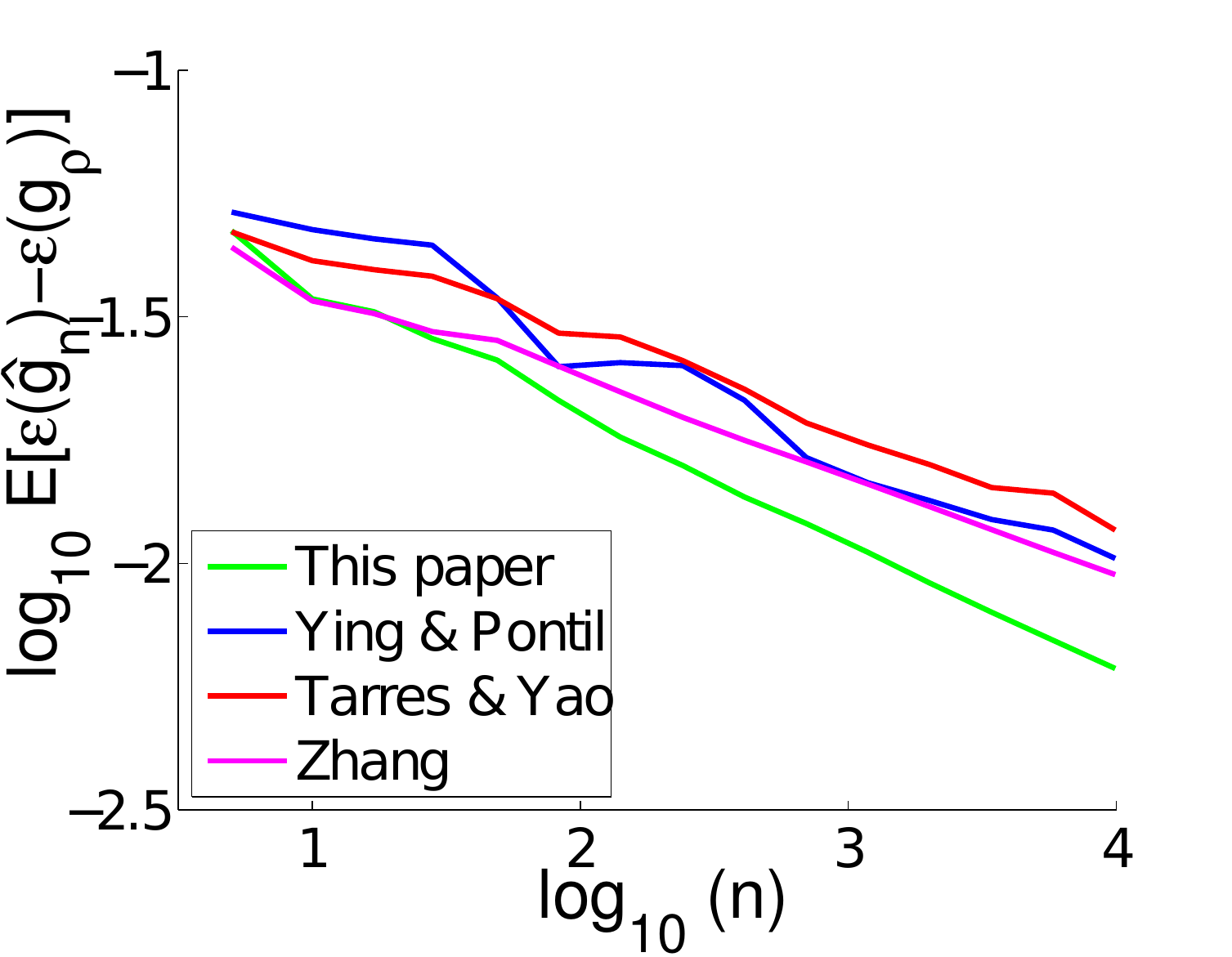}
  \caption{$r=0.125, \alpha=4$}
  \end{subfigure}
\caption{Comparison between  algorithms. We have chosen parameters in each algorithm accordingly with description in Section~\ref{subsec:summingup}, especially for the choices of $\gamma_0$.  The y-axis is $\log_{10} \( \E [\epsilon(\hat{g}_n)-\epsilon(g_\rho)] \)$, where the final output $\hat{g}_n$ may be either $\overline{g}_n$ (This paper, Zhang) or $g_n$(Ying \& Pontil, Yao \& Tarres). This expectation is computed by averaging over 15 independent samples.}
\end{figure}

Finally, we sum up the rates that were both predicted and derived for the four algorithms in the four cases for $(\alpha,\delta)$ in Table~\ref{tab:rates}. It appears that (a) we approximatively match the predicted rates in most cases (they would if $n$ was larger), (b) our rates improve on existing work.

\begin{table}
\makebox[\textwidth][c]{
\begin{tabular}{|c|c|c|c|c|}
\hline
						&  {$r=0.75 $}& {$r=0.375 $}& {$r=1.25 $}&{$r=0.125 $} \\
						&  {$  \alpha=2$}& {$  \alpha=4$}& {$  \alpha=2$}&{$  \alpha=4$} \\						
						 \hline
Predicted rate (our algo.) & \textbf{-0.75 }& \textbf{-0.75} &\textbf{ -0.8} & \textbf{-0.25}\\
Effective rate (our algo.) &  \textbf{-0.7}& \textbf{-0.71 }&\textbf{-0.69} &\textbf{-0.29 }\\ \hline
\hline
Predicted rate (YP) & -0.6 & -0.43 & -0.71 &-0.2 \\
Effective rate (YP) & -0.53& -0.5& -0.63& -0.22\\ \hline
\hline
Predicted rate (TY) & -0.6 &  &   & \\
Effective rate (TY) & -0.48 & -0.39  & -0.43&  -0.2\\ \hline
\hline
Predicted rate (Z) &  & -0.43  &   &-0.2  \\
Effective rate (Z) & -0.53 &  -0.43 & -0.41 &-0.21 \\ \hline
\end{tabular}
}
\vspace{0.5em}
\caption{Predicted and effective rates (asymptotic slope of the $\log$-$\log$ plot) for the four different situations. We leave empty cases when the set-up does not come with existing guarantees: most algorithms seem to exhibit the expected behaviour even in such cases.}\label{tab:rates}
\end{table}

\section{Conclusion}
\label{sec:conclusion}

In this paper, we have provided an analysis of averaged unregularized stochastic gradient methods for kernel-based least-squares regression. Our novel analysis allowed us to consider larger step-sizes, which in turn lead to optimal estimation rates for many settings of eigenvalue decay of the covariance operators and smoothness of the optimal prediction function. Moreover, we have worked on a more general setting than previous work, that includes most interesting cases of positive definite kernels.

Our work can be extended in a number of interesting ways: First, (a) we have considered results in expectation; following the higher-order moment bounds from~\cite{bac2013nonstrongly} in the Euclidean case, we could consider higher-order moments, which in turn could lead to high-probability results or almost-sure convergence. Moreover, (b) while we obtain optimal convergence rates for a particular regime of kernels/objective functions, using different types of averaging (i.e., non uniform) may lead to optimal rates in other regimes. Besides, (c) following~\cite{bac2013nonstrongly}, we could extend our results for infinite-dimensional least-squares regression to other smooth loss functions, such as for logistic regression, where an online Newton algorithm with the same running-time complexity would also lead to optimal convergence rates. Also, (d) the running-time complexity of our stochastic approximation procedures is still quadratic in the number of samples $n$, which is unsatisfactory when $n$ is large; by considering reduced set-methods~\cite{dekel2005forgetron,bordes2005fast,bac2012sharp}, we hope to able to obtain a complexity of $O(d_n n)$, where $d_n$ is such that the convergence rate is $O(d_n/n)$, which would extend the Euclidean space result, where $d_n$ is constant equal to the dimension. Finally, (e) in order to obtain the optimal rates when the bias term dominates our generalization bounds, it would be interesting to combine our spectral analysis with recent accelerated versions of stochastic gradient descent which have been analyzed in the finite-dimensional setting~\cite{flammarion2015averaging}.

\section*{Acknowledgements}

This work was partially supported by the European Research Council (SIERRA Project). We thank Nicolas Flammarion for helpful discussions.

\newpage
\begin{appendices}
\section{Minimal assumptions}
 \label{App:rkhsnoproof} 
\subsection{Definitions}
We first define the set  of square $\rho_X$-integrable functions~$\L$:
\begin{eqnarray*}
\L = \left\lbrace f~: \mathcal{X} \rightarrow \mathbb{R} \ \Big/ \int_\mathcal{X} f^2(t) d\rho_{X}(t) < \infty\right\rbrace;
\end{eqnarray*}
we will always make the assumptions that this space is separable (this is the case in most interesting situations. See~\cite{tho2000elementary} for more details.)
$\Ld$ is its quotient under the equivalence relation given by $$
f \equiv g \Leftrightarrow \int_\mathcal{X} (f(t)-g(t))^2 d\rho_{X}(t)=0,
$$which makes it a separable Hilbert space (see, e.g.,~\cite{kolmogorov1999elements}). 

We denote $p$ the canonical projection from $\L$ into $\Ld$ such that $p: f \mapsto \tilde{f}$, with $\tilde{f}=\lbrace g \in \L, \text{ s.t. } f\equiv g\rbrace.$ 

Under assumptions \textbf{A1, A2} or \textbf{A1', A2'}, any function in $\H$ in in $\L$. Moreover, under \textbf{A1, A2} the spaces $\H$ and $p(\H)$ may be identified,  where $p(\H)$ is the image of $\H$ via the mapping $p \circ i ~: \H \xrightarrow{i} \L \xrightarrow{p} \Ld$, where $i$ is the trivial injection from $\H$ into $\L$.

\subsection{Isomorphism}
As it has been explained in the main text, the minimization problem will appear to be an approximation problem in~$\L$, for which we will build estimates in $\H$. However, to derive theoretical results, it is easier to consider it as an approximation problem in the Hilbert space $\Ld$, building estimates in $p(\H)$.

 We thus need to define a notion of the best estimation in $p(\H)$. We first define the closure  $\overline{{F}}$ (with respect to $\| \cdot\| _{\Ld}$) of any set $F \subset \Ld$ as  the set of limits of sequences in ${F}$. The space $\overline{p(\H)}$ is a closed and convex subset in $\Ld$.  We can thus define $g_\H=\arg\min_{f\in \,\overline{p(\H)}} \epsilon(g)$, as the orthogonal projection  of $g_{\rho}$ on $\overline{p(\H)}$, using the existence of the projection on any closed convex set in a Hilbert space. See Proposition~\refa{prop:def_approximation_function} in Appendix~\refawc{App:rkhsnoproof} for details.

\begin{Prop}[Definition of best approximation function]  \label{prop:def_approximation_function} Assume \textbf{(A1-2)}. The minimum of $ \epsilon(f)$ in $\overline{p(\H)}$ is attained at a certain $ g_\H $ (which is unique and well defined   in $\Ld$). 
\end{Prop}
Where $\overline{{p(\H)}}= \Big\{  f\in \Ld \ / \  \exists (f_n) \subset {{p(\H)}}, \| f_n-f\| _{\Ld} \rightarrow 0  \Big\}$ is the set of functions $f$ for which we can hope for consistency, i.e., having a sequence $(f_n)_n$ of estimators in $\H$ such that $\epsilon(f_n) \rightarrow \epsilon(f)$.

The properties of our estimator, especially its rate of convergence will strongly depend on some properties of both the kernel, the objective function and the distributions, which may be seen through the properties of the covariance operator which is defined in the main text. We have defined the covariance operator, $\Sigma: \H \rightarrow \H$. 
In the following, we extend such an operator as an endomorphism $\T$ from $\Ld$  to $\L$ and by projection as an endomorphism  $\Td = p \circ \T$ from $\Ld$ to $\Ld$.
Note that $\T$ is well defined as $ \int_\mathcal{X} g(t)\  K_t  \ d\rho_\mathcal{X}(t) $ does not depend on the function $g$ chosen in the class of equivalence of $g$.

\begin{Def}[Extended covariance operator]
Assume \textbf{(A1-2)}. We define the operator $\T$ as follows (this expectation is formally defined as a Bochner expectation in $\H$.):
\begin{eqnarray*}
\T   \:\ \  \Ld  & \rightarrow&  \L  \\
{g} &\mapsto & \int_\mathcal{X} g(t)\  K_t  \ d\rho_\mathcal{X}(t),
\end{eqnarray*}
so that for any $ z \in \mathcal{X}  $, $ \T(g)(z)= \displaystyle \int_\mathcal{X} g(x)\  K(x,z)  \ d\rho_\mathcal{X}(t) = \E [ g(X) K(X,z)]. $
\end{Def}
  
A first important remark is that $\Sigma f=0$ implies $\langle f, \Sigma f\rangle= \|f\|_{\Ld}^{2}=0$, that is $p(\Ker(\Sigma))=\lbrace 0\rbrace$. However, $\Sigma$ may not be injective (unless $ \|f\|_{\Ld}^{2}\Rightarrow f=0$, which is true when $f$ is continuous and $\rho_X$ has full support). $\Sigma$ and $\T$ may independently be injective or not. 
   
The operator $\Td$ (which is an endomorphism of the separable Hilbert space $\Ld$) can  be reduced in some Hilbertian eigenbasis of $\Ld$.
The linear operator $\T$ happens to  have an image included in $\H$, and the eigenbasis of $\Td$ in $\Ld$ may also be seen as eigenbasis of $\Sigma$ in $\H$ (See proof in Appendix~\refa{subsec:app_covar}, Proposition~\refawc{prop:app_decSigma}):
\begin{Prop}[Decomposition of $\Sigma$]\label{prop:decsigmamain}
Assume  \textbf{(A1-2)}. The image of $\T$ is included in $\H$: $\text{Im}(\T) \subset \H$, that is, for any $f \in \Ld$, $\T f \in \H$. Moreover, for any $i \in I$, $\phi_i^H = \frac{1}{\mu_i} \T \phi_i \in \H \subset \L$ is a representant for the equivalence class $\phi_i$, that is $p(\phi_i^H) = \phi_i$. Moreover $\mu_i^{1/2} \phi_i^H$ is an orthonormal eigen-system of the orthogonal supplement $\S$ of the null space $\text{Ker}(\Sigma)$.
That is:
\begin{itemize}
\item[--] $\forall i \in I , \   \Sigma \phi_i^H= \mu_i \phi_i^H $.
\item[--] $\H = \Ker(\Sigma) \overset{\perp}{\oplus} \S$.
\end{itemize}
\end{Prop}

Such decompositions allow to define $\T^r:\Ld \rightarrow\H$ for $r\geq 1/2$. Indeed , completeness allows to define infinite sums which satisfy a Cauchy criterion. See proof in Appendix~\refa{subsec:app_covar}, Proposition~\refawc{propTdr}.
Note the different condition concerning $r$ in the definitions. For $r\geq 1/2$, $\Td^{r}=p\circ \T^{r}$. We need $r \geqslant 1/2$, because $(\mu_i^{1/2} \phi^H)$ is an orthonormal system of $\S$.

 \begin{Def}[Powers of $\T$]\label{def:Tr}
 We define, for any $ r \geq 1/2 $, $\T^r:  \Ld      \rightarrow  \H$, for any $ h \in  \Ker(\Td)$ and $(a_i)_{i \in I}$ such that $\sum_{i \in I } a_i^2 < \infty$, through:
 $$
 \T^r\( h + \sum_{i \in I} a_i \phi_i \right) = \sum_{i \in I} a_i \mu_i^r \phi_i^H.
 $$
\end{Def}

{We  have two decompositions of $\Ld= \Ker(\Td) \overset{\perp}{\oplus} S$ and $\H = \Ker(\Sigma) \overset{\perp}{\oplus} \S$. The two orthogonal supplements $S$ and $\S$ happen to be related through the mapping $\T^{1/2}$, as stated in Proposition~\ref{isomor}: $\T^{1/2}$ is an isomorphism from $S$ into $\S$. It also has he following consequences, which generalizes Corollary~\ref{cor_struct_main}:}
\begin{Cor} \label{cor_struct}
\begin{itemize}
\item[--]  $\Td^{1/2}(S)=p(\H)$, that is any element of $p(\H)$ may be expressed as $\Td^{1/2} g$ for some $g \in \Ld$.
\item[--] For any $ r\geq 1/2,\  \Td^{r}(S)\subset \H$, because  $\Td^r(S) \subset \Td^{1/2}(S)$, that is, with large powers $r$, the image of $\Td^r$ is in the projection of the Hilbert space.
\item[--] $\forall r >0, \ \overline{\Td^{r} (\Ld)} = S= \overline{\Td^{1/2} (\Ld)}= \overline{\H}$, because (a) $ \Td^{1/2}(\Ld) ={p(\H)}$ and (b) for any $r>0$, $\overline{\Td^r(\Ld)} = S$. In other words, elements of $ \overline{p(\H)}$ (on which our minimization problem attains its minimum), may seen as limits (in $\Ld$) of elements of $\Td^{r}( \Ld)$, for any $r>0$.
\item[--]  \emph{$p(\H)$ is dense in $\Ld $ if and only if $T$ is injective.}
\end{itemize}
\end{Cor}

\subsection{Mercer theorem generalized}

Finally, although we will not use it in the rest of the paper, we can state a version of Mercer's theorem, which does not make any more assumptions that are required for defining RKHSs.

\begin{Prop}[Kernel decomposition] \label{prop:kerneldec}
Assume  \textbf{(A1-2)}. We have for all $x,y \in \X$,
$$
K(x,y) = \sum_{i \in I} \mu_i \phi_i^H(x) \phi_i^H(y) + g(x,y),
$$
 and we have for all $x\in \X$,  $\int_{ \X} g(x,y)^2 d\rho_X(y) = 0$. Moreover, the convergence of the series is absolute.
\end{Prop}
 We thus obtain a version of Mercer's theorem (see Appendix~\refawc{subsec:app-mercerker}) 
  without any topological assumptions. Moreover, note that (a) $\S$ is also an RKHS, with kernel $(x,y) \mapsto \sum_{i \in I} \mu_i \phi_i^H(x) \phi_i^H(y)$ and (b) that given the decomposition above,  the optimization problem in $\S$ and $\mathcal{H}$ have equivalent solutions.  Moreover, considering the algorithm below, the estimators we consider will almost surely build equivalent functions (see Appendix~\refawc{subsec:app_smallerrkhs}).  
   Thus, we could assume without loss of generality that the kernel $K$ is exactly equal to its expansion $\sum_{i \in I} \mu_i \phi_i^H(x) \phi_i^H(y)$.

%
%
 \subsection{Complementary (A6) assumption}
 Under minimal assumptions, we also have to make a complementary moment assumption :
\begin{itemize}
  \item  [\textbf{(A6')}] There exists $ R>0 $ and $ \sigma >0 $ such that $ \E \left[ \Xi \otimes \Xi \right] \lec \sigma^2 \Sigma$, and $ \E( K(X,X) K_{X}\otimes K_{X})\lec R^2 \Sigma $ where $\lec$ denotes the order between self-adjoint operators.
  \end{itemize}  
  In other words, for any $f \in \H$, we have: $ \E \big[K(X,X) f(X)^2 \big] \leqslant R^2 \E [f(X)^2]$. Such an assumption is implied by \textbf{(A2)}, that is  if $K(X,X)$ is almost surely bounded by $R^{2}$: this constant can then  be understood as the radius of the set of our data points. However, our analysis holds in these more general set-ups where only fourth order moment of $\| K_x\|_\H = K(x,x)^{1/2}$ is finite.

\section{Sketch of the proofs}
\label{app_sketch}
Our main theorems are Theorem~\ref{prop.dinf.rand} and Theorem~\ref{prop.dinf.rand.onl}, respectively in the finite horizon and in the online setting. Corollaries can be easily derived by optimizing over $\gamma$ the upper bound given in the theorem.

The complete proof is given in Appendix~\refawc{A_proofs}. The proof is nearly the same for finite horizon and online setting. It relies on a refined analysis of strongly related recursions in the RKHS and on a comparison between iterates of the recursions (controlling the deviations).
 
 \ \\
 We first present the sketch of the proof for the \emph{finite-horizon setting} : \\
 We want to analyze the error of our sequence of estimators $(g_n)$ such that $g_0=0$ and 
\begin{eqnarray*}
g_n &=&  g_{n-1} - \gamma_n  \big[  y_n -\langle g_{n-1}, K_{x_n} \rangle_\H\big] K_{x_n} \\
g_n &=&  (I-\gamma \x)g_{n-1} +\gamma y_n K_{x_n} \\
g_n - g_\H &=&  (I-\gamma \widetilde{\x})(g_{n-1}- {g_\H})  +\gamma  \Xi_n.
\end{eqnarray*}
 Where we have denoted $\Xi_n = (y_n - g_\H(x_n)) K_{x_n}$ the residual, which has 0 mean, and $ \widetilde{\x}: \Ld \rightarrow \H $ an a.s. defined extension of $\x: \H \rightarrow \H $, such that $\widetilde{\x} (f)=f(x_n)  K_{x_n}$, that will be denoted for simplicity $\x$ in this section.
 
\ \\
Finally, we are studying a sequence $ (\n{n})_n =(g_n-g_\H)_n$  defined by:
\begin{eqnarray} 
 \n{0}&=& {g_\H} , \nonumber \\
 \eta_n&=&(I-\gamma_n {\x})\eta_{n-1} +\gamma_n \Xi_n.\nonumber 
\end{eqnarray}
We first consider  splitting this recursion in two simpler recursions $ \n{n}^{init}$ and $\n{n}^{noise}$ such that $ \n{n}=\n{n}^{init}+\n{n}^{noise}$: 
 \begin{itemize}
   \item  $(\n{n}^{init})_n $ defined by~: $$ \n{0}^{init} = {g_\H} \mbox{ and } \n{n}^{init}=(I-\gamma \x)\n{n-1}^{init} .$$ 
 $ \n{n}^{init} $ is the part of $ (\n{n})_n $ which is due to the \textbf{initial conditions} ( it is equivalent to assuming $\Xi_n\equiv 0$).
\item Respectively, let $(\n{n}^{noise})_n $ be defined by~: $$ \n{0}^{noise} =0 \mbox{ and } \n{n}^{noise}=(I-\gamma \x)\n{n-1}^{noise}   +\gamma \Xi_n.$$
 $ \n{n}^{noise} $ is the part of $ (\n{n})_n $ which is due to the\textbf{ noise}.
 \end{itemize}
We will bound $\|\eta_n\|$ by $\|\n{n}^{init}\|+\|\n{n}^{noise}\| $ using Minkowski's inequality. \emph{That is how the bias-variance trade-off originally appears.}

\ \\
Next,  we notice that $\E [\x] =\T$, and thus define ``semi-stochastic'' versions of the previous recursions by replacing $\x$ by its expectation: 

\ \\ 
\textbf{For the initial conditions: }
$(\eta_n^{0, init})_{n\in \N}$ so that~: $$\eta_0^{0, init}=    {g_\H}, \quad \eta_n^{0, init}= (I-\gamma \T) \eta_{n-1}^{0, init} .$$
which is a deterministic sequence. 

An algebraic calculation gives an estimate of the norm of $\eta_n^{0, init}$, and we can also bound the residual term $\eta_n^{ init}-\eta_n^{0, init}$, then conclude by Minkowski.

\ \\
\textbf{For the variance term:}
We follow the exact same idea, but have to define a sequence of ``semi-stochastic recursion'', to be able to bound the residual term.

\ \\
This decomposition is summed up in Table~\ref{tab:FHapp0}.

\begin{table}[H]
\makebox[\textwidth][c]{
\begin{tabular}{|ccccc|}
   \hline
  &    \multicolumn{3}{c}{Complete recursion $\eta_n$ } &      \\ 
   
     &    $\swarrow$ &   & $\searrow$ &      \\ 
   
    \multicolumn{2}{|c}{variance term $\eta^{noise}_n$} &  | &     \multicolumn{2}{c|}{ bias term $\eta^{init}_n$}  \\ 
   
    \multicolumn{2}{|c}{$\downarrow$  } &  | &    \multicolumn{2}{c|}{ $\downarrow$ }   \\ 
   
    \multicolumn{2}{|c}{multiple recursion } &  | &   \multicolumn{2}{c|}{ semi stochastic variant}   \\ 
   
   $\swarrow$ &    $\searrow$ & |  & $\swarrow$ &   $\searrow$ \\ 
   
   main terms $\eta^r_n$, $r\geq 1$ &    residual term $\eta^{noise}_n - \sum \eta^r_n$ &  | & main term $\eta^0_n $ &    residual term $\eta^{init}_n -  \eta^0_n$ \\ 
   
  satisfying semi-sto recursions &    satisf. stochastic recursion & |  & satisf. semi-sto recursion &    satisf. stochastic recursion \\ 
   
   $\downarrow$ Lemma \refa{lem.ssto.rec} &    $\downarrow$ Lemma \refa{lem.stoch.rec} &  | & $\downarrow$  &    $\downarrow$ Lemma \refa{lem.stoch.rec} \\ 
   
   $\le C$ Variance term &    $\rightarrow_{r\rightarrow \infty} 0$ & |  & $\le $ Bias term &    residual negligible term \\ 
  
  &&&& \\
    \multicolumn{2}{|c}{ \ $\qquad\qquad$ Lemma \refa{var_gam_const} $ \searrow$} &   & \multicolumn{2}{c|}{$\swarrow$ Lemma \refa{bias_gam_const}}     \\ 
 
     &    \multicolumn{3}{c}{Theorem \refa{prop.dinf.rand}} &      \\ 
   \hline
   
\end{tabular} 

} \vspace{0.5em}\caption{Error decomposition in the finite horizon setting. All the referances refer to Lemmas given in Appendix~\refawc{A_proofs}.}\label{tab:FHapp0}
\end{table}

For the \emph{online setting}, we follow comparable ideas and end in a similar decomposition.

\newpage

\setcounter{section}{0}
\renewcommand{\thesection}{\Roman{section}}

In Appendix~\ref{A_rkhs}, we provide proofs of the propositions from Section~\ref{sec:rkhs} that provide the Hilbert space set-up for kernel-based learning, while in Appendix~\ref{A_proofs}, we prove convergence rates for the least-mean-squares algorithm.

\section{Reproducing kernel Hilbert spaces}\label{A_rkhs}
In this appendix, we provide proofs of the results from Section~\ref{sec:rkhs} that provide the RHKS space set-up for kernel-based learning. See~\cite{Aro1950theory,Sma2001mathematical,yao2006dynamic} for further properties of RKHSs.

We consider a reproducing kernel Hilbert space $\H$ with kernel $K $ on space $\X$ as defined in Section~\ref{subsec:rkhs}. Unless explicitly mentioned, we do not make any topological assumption on $\X$.

As detailed in Section~\ref{subsec:randomvar} we consider a set $\X$ and $\mathcal{Y}\subset \R$ and a distribution  $\rho$ on $\X\times \mathcal{Y}$. We denote $\rho_X$ the marginal law on the space $\X$. In the following, we use the notation $(X,Y)$ for a random variable following the law $\rho$. We define spaces $\Ld, \L$ and  the canonical projection  $p$. In the following we further assume that $\Ld$ is separable, an assumption satisfied in most cases.

We remind our assumptions: 
 \begin{enumerate}
\item[\textbf{(A1)}] $\H$ is a separable RKHS associated with kernel $K$ on a space $\mathcal{X}$.
\item[\textbf{(A2)}] $ \E \left[K(X,X)\right] $ and $ \E [Y^2] $ are finite. 
 \end{enumerate}

Assumption \textbf{(A2)} ensures that every function in $\mathcal{H}$ is square-integrable, that is,
if $\E [K(X,X)] < \infty  $, then  $ \H \subset \L$. Indeed, we have:  
\begin{Prop}\label{prop:inclusionRKHSL2}
Assume \textbf{(A1)}.
\begin{enumerate}
\item If $\E [K(X,X)] < \infty  $, then  $ \H \subset \L$.
\item If $\sup_{x\in \mathcal{X}} K(x,x) < \infty$, then  any function in $ \H$ is bounded.
\end{enumerate}
 \end{Prop}
 
\begin{proof} 
Under such condition, by Cauchy-Schwartz inequality, any function $f\in \H$ is either bounded or integrable: 
 \begin{eqnarray*}
|f(x)|^{2} &\le &  \|f\|_K^{2} K(x,x) \le \|f\|_K^{2} \sup_{x\in \mathcal{X}} K(x,x),\\
\int_\X |f(x)|^{2} d\rho_X(x) &\le &  \|f\|_K^{2}  \int_\X K(x,x) d\rho_x(x).
 \end{eqnarray*} 
\end{proof} 
 
 The assumption  $\E [K(X,X)] < \infty  $ seems to be the weakest assumption to make, in order to have at least $\H \subset \L$. However they may exist functions  $f\in \H \setminus \lbrace 0\rbrace$ such that $\|f\|_{\L} =0$. However under stronger assumptions (see Section~\ref{sec:strongerassumptions}) we may identify $\H$ and $p(\H)$.

\subsection{Properties of the minimization problem} \label{subsec:app_minpb}
We are interested in minimizing the following quantity, which is the \textit{prediction error} of a function $f$, which may be rewritten as follows with dot-products in $\Ld$:
\begin{eqnarray}
\epsilon(f)&=& \E \left[\left(f(X)-Y\right)^2\right]\nonumber\\
&=& \|f\|^2_{\Ld} - \int_{\mathcal{X}\times \mathcal{Y}} f(x) y d\rho(x,y) +c \nonumber\\
&=&  \|f\|^2_{\Ld} - \int_{\mathcal{X}} f(x) \({\int_\mathcal{Y} y d\rho_{Y|X=x}(y)}\) d\rho_{|X}(x) +c \nonumber\\
&=& \|f\|^2_{\Ld} - \left\langle f , \int_\mathcal{Y} y d\rho_{Y|X=\cdot}(y) \right\rangle_{\Ld} +c \label{Formstamp}\\
&=&  \|f\|^2_{\Ld} - \left\langle f , \E\[Y|X=\cdot\] \right\rangle_{\Ld} +c  \nonumber
\end{eqnarray}
Notice that the problem may be re-written, if $f$ is in $\H$, with dot-products in $\H$: 
\begin{eqnarray}
\epsilon(f)&=& \E[f(X)^2] -2 \langle f, \E [ Y K_X] \rangle_K + \E[Y^2]\nonumber\\
 &=& \langle f, \Sigma f \rangle_K - 2\langle f, \mu \rangle_K + c \nonumber.
\end{eqnarray}
\textbf{Interpretation:} Under the form \eqref{Formstamp}, it appears to be a minimisation problem in a Hilbert space of the sum of a continuous coercive function and a linear one. Using Lax-Milgramm and Stampachia theorems \cite{bre1983analyse} we can conclude with the following proposition, which implies Prop.~\ref{prop:def_approximation_function} in Section~\ref{sec:rkhs}: 

\begin{Prop}[$g_\rho, g_{\H}$] \label{prop:defgH} 
Assume \textbf{(A1-2)}. We have the following points:
\begin{enumerate}
\item There exists a unique minimizer over the space $\Ld$. This minimizer is the regression function $g_\rho:x \mapsto \int_\mathcal{Y} y d\rho_{Y|X=x}(y)$ (Lax-Milgramm).
\item For any non empty closed convex set, there exists a unique minimizer (Stampachia). As a consequence, there exists a unique minimizer: $$g_\H=\arg\min_{f\in \overline{p(\H)}} \E \left[ (f(X)-Y)^2 \right]$$
over $\overline{p(\H)}$. $g_\H$ is the orthogonal projection over $g_\rho$ over $\overline{p(\H)}$, thus satisfies the following equality: 
for any $ \epsilon \in \overline{H}$:
\begin{equation}
\E \[ (g_\H(X)-Y) \epsilon(X) \]=0 \label{orthogonalitedegh}
\end{equation}
\end{enumerate}
\end{Prop}

\subsection{Covariance Operator} \label{subsec:app_covar}
We defined operators $\Sigma, \T,\Td$ in Section~\ref{subsec:covoper}. We here state the main properties of these operators, then prove the two main decompositions stated in Propositions~\ref{propTd} and~\ref{prop:decsigmamain}.

\begin{Prop}[Properties of $\Sigma$]\label{propSigma}
Assume \textbf{(A1-2)}.
\begin{enumerate}
\item\label{point1propsig} $\Sigma$ is well defined (that is for any $f\in \H$, $z\mapsto \E f(X) K(X,z) $ is in $\H$). 
\item $\Sigma$ is a continuous operator.
\item $\Ker(\Sigma)= \lbrace f \in \H \text{ s.t. } \|f\|_{\Ld} =0 \rbrace$. Actually  \label{fsigmaf}
for any $f \in \H$, $\langle f, \Sigma f \rangle_K=\|f\|_{\Ld}$.
\item $\Sigma$  is a self-adjoint operator.
\end{enumerate}
\end{Prop}
\begin{proof}
\begin{enumerate}
\item for any $x\in \X$, $f(x) K_x $ is in $\H$. To show that the integral $\int_{x \in \X} f(x) K_x$ is converging, it is sufficient to show the is is absolutely converging in $\H$, as absolute convergence implies convergence in  any Banach space\footnote{A Banch space is a linear normed space which is complete for the distance derived from the norm.} (thus any Hilbert space). Moreover:\begin{eqnarray*}
\int_{x \in \X} \|f(x) K_x\|_K &\le & \int_{x \in \X} |f(x)| \langle K_x, K_x \rangle ^{1/2}_K\\
&\le & \int_{x \in \X} |f(x)| K(x,x)^{1/2} d\rho_X(x)\\
&\le & \(\int_{x \in \X} f(x)^2 d\rho_X(x)\)^{1/2}  \(\int_{x \in \X} K(x,x) d\rho_X(x) \)^{1/2}\\
&<& \infty, 
\end{eqnarray*}
under assumption $\E[K(X,X)]<\infty$ (\textbf{(A2)}). 
\item For any $f \in \H$, we have  \begin{eqnarray*}
\|\Sigma f\|_K&= &  \langle \Sigma f, \Sigma f \rangle_K = \int_{x \in \X} (\Sigma f)(x) f(x) d\rho_X(x) \\
&=& \int_{x \in \X} \(\int_{y\in \X} f(y) K(x,y) d\rho_X(y)\) f(x) d\rho_X(x) \\
&=& \int_{x,y \in \X^2} \langle f, K_x\rangle_K \langle f, K_y\rangle_K  \langle K_y, K_x\rangle_K d\rho_X(x) d\rho_X(y)\\
&\le & \int_{x,y \in \X^2} \|f\|_K \|K_x\|_K \|f\|_K \|K_y\|_K \|K_x\|_K \|K_y\|_K d\rho_X(x) d\rho_X(y) \\
& & \hspace{7cm} \text{ by Cauchy Schwartz,} \\
&\le & \|f\|_K ^{2 } \(\int_{x \in \X^2}  \|K_x\|_K^{2}   d\rho_X(x)\)^{2}\\
&\le & \|f\|_K ^{2 } \(\int_{x \in \X^2}  K(x,x)  d\rho_X(x)\)^{2},
\end{eqnarray*}
which proves the continuity under assumption \textbf{(A2)}.
\item $\Sigma f =0 \Rightarrow \langle f, \Sigma f \rangle = 0 \Rightarrow \E[f^2(X)] = 0$. Reciprocally, if $\|f\|_{\Ld}=0$, it is clear that $\|\Sigma f\|_{\Ld}=0$, then $\|\Sigma f\|_K= \E\[ f(X) (\Sigma f)(X) \right] =0 $, thus $f\in \Ker(\Td )$.
\item It is clear that $\langle \Sigma f, g \rangle =\langle  f,\Sigma g \rangle $.
\end{enumerate}
\end{proof}

\begin{Prop}[Properties of $\T$] \label{propT}
Assume \textbf{(A1-2)}. 
$\T $ satisfies the following properties:
\begin{enumerate}
\item $\T $ is a well defined, continuous operator. 
\item For any $f\in \H $, $\T (\tilde{f})=\Sigma f$, $\| \T f\|_K^2= \int_{x,y \in \X^2} f(y) f(x) K(x,y) d\rho_X(y)  d\rho_X(x) $.
\item The image  of $\Td$ is a subspace of $\H$.
\end{enumerate}
\end{Prop}
\begin{proof}
 It is clear that $\T $ is well defined, as for any class $\tilde{f}$,  $\int_\mathcal{X} f(t)\  K_t  \ d\rho_X(t)$ does not depend on the representer $f$, and is converging in $\H$ (which is the third point), just as in the previous proof. The second point results from the definitions. Finally for continuity, we have: \begin{eqnarray*}
\|\T f\|_K^2&= &  \langle \T f, \T f \rangle_K\\
&=& \int_{x \in \X} \int_{y\in \X} f(y) f(x) K(x,y) d\rho_X(y)  d\rho_X(x) \\
&\le& \( \int_{x \in \X^2} | f(x) K(x,x)^{1/2}|   d\rho_X(x)\)^2\\
&\le& \( \int_{x \in \X} f(x)^2 d\rho_X(x) \)\(\int_{x \in \X} K(x,x)   d\rho_X(x)\) \le C \|f\|_{\Ld}^2.
 \end{eqnarray*}
\end{proof}
We now state here a simple lemma that will be useful later:

\begin{Lem}\label{prophypo} 
Assume \textbf{(A1)}. 
\begin{enumerate}
\item $\E\[k(X,X)\right]<\infty \Rightarrow \int_{x,y\in \X} k(x,y)^{2} d\rho_X(x) d\rho_X(y)< \infty $. 
\item $\E\[|k(x,y)|\right]<\infty \Rightarrow \int_{x,y\in \X} k(x,y)^{2} d\rho_X(x) d\rho_X(y)< \infty$.
\end{enumerate}
\end{Lem}

\begin{Prop}[Properties of $\Td$]\label{propertiesTd-app}
Assume \textbf{(A1-2)}. 
$ \Td $ satisfies the following properties:
\begin{enumerate}
\item $ \Td $ is a well defined, continuous operator. 
\item The image  of $\Td$ is a subspace of $p (\H)$.
\item $\Td $ is a self-adjoint semi definite positive operator in the Hilbert space $\Ld$.
\end{enumerate}
\end{Prop}
\begin{proof}
 $\Td =p\circ \T$ is clearly well defined,  using the arguments given above. Moreover:
\begin{eqnarray*}
\|\Td  f\|_{\Ld}^{2}&=& \int_{x\in \X} \(\int_{t\in X} K(x,t) f(t) d\rho_{X}(t)\)^2 d\rho_{X}(x)\\
&\le&  \(\int_{x\in \X} \int_{t\in X} K(x,t)^2 d\rho_{X}(t) d\rho_{X}(x)\)  \(\int_{t\in \X} f^2(t)d\rho_{X}(t)  \)\text{ by C.S.}\\
&\le& C \|f\|_{\L} \quad\text{by Lemma~\ref{prophypo}, }
\end{eqnarray*}
which is continuity\footnote{We could also use the continuity of $p~: \H \rightarrow \Ld$.}.
Then by Proposition~\ref{propT}, $\Image(Td)\subset p(\Image(\T)) \subset p(\H)$. Finally, for any $f,g\in \L$,  
\begin{eqnarray*}
\langle f ,\Td  g \rangle_{\L}&=&\int_\mathcal{X} f(x)\ \Td g(x)d\rho_X(x) \\
&=& \int_\mathcal{X} f(x) \(\int_\mathcal{X} g(t) K(x,t) d\rho_X(t)\) d \rho_X(x)\\
&=&\int_{\mathcal{X}\times \mathcal{X}} f(x)  g(t) K(x,t) d\rho_X(t) d \rho_X(x) =\langle \Td  f , g \rangle_{\L}.
\end{eqnarray*}
and $\langle f ,\Td  f \rangle_{\L}\ge 0$ as a generalisation of the positive definite property of~$K$.
\end{proof}

In order to show the existence of an eigenbasis for $\Td$, we now show that $\Td$ is trace-class.
\begin{Prop}[Compactness of the operator]
We have the following properties:
\begin{enumerate}
\item Under \textbf{(A2)}, $ \Td $ is a trace class operator\footnote{Mimicking the definition for matrices, a bounded linear operator $A$ over a separable Hilbert space $H$ is said to be in the trace class if for some (and hence all) orthonormal bases $(e_k)_k$ of $H$ the sum of positive terms ${\tr}|A|:=\sum_{k} \langle (A^*A)^{1/2} \, e_k, e_k \rangle $ is finite.}. As a consequence, it is also a Hilbert-Schmidt operator\footnote{A Hilbert-Schmidt operator is a bounded operator $A$ on a Hilbert space $H$ with finite Hilbert–Schmidt norm: $\|A\|^2_{\text{HS}}={\tr} |(A^{{}^*}A)|:= \sum_{i \in I} \|Ae_i\|^2$.}.
\item If $K\in L^2(\rho_X \times \rho_X)$ then $\Td  $ is a Hilbert-Schmidt operator.
\item Any Hilbert-Schmidt operator is a compact operator.
\end{enumerate}
\end{Prop}
\begin{proof}
Proofs of such facts may be found in \cite{bre1983analyse,pau2009topologie}. Formally, with $ \ (\phi_i)_i$ an Hilbertian basis in $\Ld$:
\begin{eqnarray*}
\E\[K(X,X)\]&=& \E\[\langle K_x, K_x\rangle_{K}\]\\
&=&  \E\[\sum_{i=1}^{\infty} \langle K_x, \phi_i\rangle_{K}^2\] \quad \text{by Parseval equality,}\\
&=&  \sum_{i=1}^{\infty} \E\[ \langle K_x, \phi_i\rangle_{K}^2\] \\
&=&  \sum_{i=1}^{\infty}  \langle \Td  \phi_i, \phi_i\rangle_{K} = \tr(\Td ).
\end{eqnarray*}
\end{proof}
\begin{Cor}
We have thus proved that under \textbf{(A1)} and \textbf{(A2)}, the operator $\Td$ may be reduced in some Hilbertian eigenbasis: the fact that $\Td $  is self-adjoint and compact implies the existence of an orthonormal eigensystem (which is an Hilbertian basis of ${L}^2_ {\rho_X}$).
\end{Cor}
This is a consequence of a very classical result, see for example \cite{bre1983analyse}.

\begin{Def}
The null space $\text{Ker}(\Td):=\left\lbrace f\in \Ld \text{ s.t. } \Td f=0 \right\rbrace $ may not be $\lbrace 0\rbrace$. We denote by $S$ an orthogonal  supplementary of $\text{Ker}(\Td)$.
\end{Def}

Proposition~\ref{propTd} is directly derived from a slightly more complete Proposition~\ref{propTd-app} below: 
\begin{Prop}  [Eigen-decomposition of $\Td$] \label{propTd-app}
 Under  \textbf{(A1)}  and \textbf{(A2)}, $\Td$ is a  bounded  self adjoint semi-definite positive operator on $\Ld$, which is trace-class. There exists\footnote{$S$ is stable by $ \Td $ and  $\Td ~: S \rightarrow S$ is a self adjoint compact positive operator.} a Hilbertian eigenbasis $(\phi_i)_{i \in I}$ of the orthogonal supplement $S$ of the null space ${\rm Ker}(\Td)$, with summable eigenvalues $(\mu_i)_{i \in I}$. That is:
\begin{itemize}
\item $\forall i \in I , \   \Td \phi_i= \mu_i \phi_i $, $(\mu_i)_i$ strictly positive non increasing (or finite) sequence such that $\sum_{i \in I } \mu_i < \infty$.
\vspace{0.2em}
\item $
\Ld=  \Ker( \Td ) \overset{\perp}{\oplus} S.
$
\end{itemize}

We have\footnote{We denote by $\text{ span}(A)$ the smallest linear space which contains $A$, which is in such a case the set of all finite linear combinations of $(\phi_i)_{i\in I}$.}: $ \displaystyle S= \overline{\text{span} \lbrace\phi_i\rbrace} = \left\lbrace \sum_{i=1}^{\infty} a_i \phi_i \text{ s.t. }  \sum_{i=1}^{\infty} a_i^2 <\infty \right\rbrace.$
Moreover: \begin{equation}
S= \overline{p(\H)}.\label{directsumLd}
\end{equation}
\end{Prop}

\begin{proof}
For any $i\in I$, $\phi_i=\frac{1}{\mu_i} L_K \phi_i  \in p(\H)$. Thus $\text{ span}\left\lbrace {\phi_i} \right\rbrace \subset p(\H)$, thus $S= \overline{\text{ span}\left\lbrace {\phi_i} \right\rbrace } \subset \overline{p(\H)}$.
Moreover, using the following Lemma, $p(\H)\subset \Ker(\Td )^{\perp} = S$, which concludes the proof, by taking the closures.
\end{proof}

\begin{Lem}
We have the following points: 
\begin{itemize}
\item  if $\Td ^{1/2} f = 0$ in $\Ld$, then $ \Td  f=0$ in $\H$.
\item  $p(\H)\subset \Ker(\Td )^{\perp}$.
\end{itemize}
\end{Lem}

\begin{proof}
 We first notice that if $\Td ^{1/2} f = 0$ in $\Ld$, then $ \T  f=0$ in $\H $: indeed\footnote{In other words, we the operator defined below $\Td^{1/2}$
\begin{eqnarray*}
\Td ^{1/2} f&=_{\Ld}& 0\\
\ \T  f &=_{\H}&\Sigma^{1/2} (\ \T ^{1/2} f) \\
 \|\ \T  f\|^2_{K}&=&\|\Sigma^{1/2} (\ \T ^{1/2} f)\|_{K}^{2} = \|(\ \T ^{1/2} f)\|_{\Ld}^{2} = 0\\
\ ^H\!                                                                                                                                                                                                                                                                                                                                                                                                                                                                                                                                                                                                                                                                                           \Td  f&=_\H&0.
\end{eqnarray*}  }
 \begin{eqnarray*}
 \|\Td f\|^2_\H &=& \bigg\langle \int_\X f(x) K_x d\rho_X(x),\int_\X f(y) K_y d\rho_X(y)   \bigg\rangle_{K} \\
 &=& \int_{\X^2} f(x) f(y) K(x,y)  d\rho_X(x)d\rho_X(y)\\
 &=&\langle f , \Td  f  \rangle_{\Ld}=0 \text{ if $\Td f=0$ in $\Ld$.} 
 \end{eqnarray*}

Moreover $\H$ is the completed space of $\text{ span } \lbrace K_x, x\in \X \rbrace$, with respect to $\|\cdot\|_K$ and for all
 $ x \in \X$, for all $\psi_k \in \Ker(\Td )$:   
\begin{eqnarray*}
\langle p(K_x) , \psi_k \rangle_{\Ld}&=& \int_\X K_x(y) \psi_K(y) d \rho_X(y)= (\Td  \psi_k) (x),\\
\text{ however,  }\quad \Td  \psi_k &= _{\Ld}& 0 \quad \Rightarrow  \Td  \psi_k =_{\H} 0 \quad \forall x\in \X \Rightarrow  \Td  \psi_k (x)= 0.
\end{eqnarray*}

As a consequence, $\text{ span} \left\lbrace p(K_x) , x \in \X \right\rbrace \subset \Ker(\Td )^{\perp}$. We just have to show that $\overline{\text{ span} \left\lbrace p(K_x) , x \in \X \right\rbrace} = p(\H)$, as $\Ker(\Td )^{\perp}$ is a closed space. It is true as for any $\tilde{f}\in p(H), f\in \H$ there exists $f_n \subset \text{ span} \left\lbrace K_x , x \in \X \right\rbrace$ such that $f_n \overset{\H}{\rightarrow} f$, thus $p(f_n) \rightarrow \tilde{f}$ in $\Ld$\footnote{$\|f_n-f\|_{\Ld}= \|\Sigma^{1/2}(f_n-f)\|_{K} \rightarrow 0$ as $\Sigma$ continuous.}. Finally we have proved that $p(\H) \subset \Ker(\Td )^{\perp}$.
\end{proof}

Similarly, Proposition~\ref{prop:decsigmamain} is derived from Proposition~\ref{prop:app_decSigma} below:
\begin{Prop}[Decomposition of $\Sigma$]\label{prop:app_decSigma}
Under  \textbf{(A1)}  and \textbf{(A2)}, $\Image(\T) \subset \H$, that is, for any $f \in \Ld$, $\T f \in \H$. Moreover, for any $i \in I$, $\phi_i^H = \frac{1}{\mu_i} \T \phi_i \in H$ is a representant for the equivalence class $\phi_i$. Moreover $\(\mu_i^{1/2} \phi_i^H\)_{i\in I}$ is an orthonormal eigein-system of $\S$ 
That is:
\begin{itemize}
\item $\forall i \in I , \   \Sigma \phi_i^H= \mu_i \phi_i^H $.
\item $\(\mu_i^{1/2} \phi_i^H\)_{i\in I}$ is an orthonormal family in  $\S$. 
\end{itemize}
We thus have: $$\S=\left\lbrace \sum_{i\in I} a_i \phi_i^H \text{ s.t. } \sum_{i\in I} \frac{a_i^2}{\mu_i} < \infty \right\rbrace. $$
Moreover $\S$ is the orthogonal supplement of the null space $\text{Ker}(\Sigma)$:$$\H = \Ker(\Sigma) \overset{\perp}{\oplus} \S.$$
\end{Prop}

\begin{proof}
The family $\phi^H_i= \frac{1}{\mu_i} \Td  \phi_i $ satisfies: 
\begin{itemize}
\item $\widetilde{\phi^H_i} = \phi_i$ (in $\Ld$),
\item $\phi^H_i \in \S$,
\item $ \Td  \phi_i^H= \mu_i \phi_i$ in $\Ld$,
\item $ \T \phi_i^H =\Sigma \phi_i^H = \mu_i \phi_i^H$ in $\H$.
\end{itemize}     

All the points are clear: indeed for example $\Sigma \phi_i^H= \Td  \phi_i = \mu_i \phi_i^H$.  Moreover, we have that:
\begin{eqnarray*}
\|\phi_i\|^2_{\Ld}=\|\phi^H_i\|^2_{\Ld} &=& \langle \phi_i^H, \Sigma \phi_i\rangle_K \text{ by Proposition~\ref{fsigmaf} }\\
&=& \mu_i \|\phi_i^H\|_K^2 \\
&=&  \|\sqrt{\mu_i}\phi_i^H\|_K^2
\end{eqnarray*}
That means that $(\sqrt{\mu_i}\phi_i^H)_i$ is an orthonormal family in $\H$.

Moreover, $\S$ is defined as the completion for $\|\cdot\|_K$ of this  orthonormal family, which gives  $\S=\left\lbrace \sum_{i\in I} a_i \phi_i^H \text{ s.t. } \sum_{i\in I} \frac{a_i^2}{\mu_i} < \infty \right\rbrace. $

To show that $\H = \Ker(\Sigma) \overset{\perp}{\oplus} \S,$ we use the following sequence of arguments: \begin{itemize}
\item First, as $\Sigma$ is a continuous operator, $\Ker(\Sigma)$ is a closed space in $\H$, thus $\H=\Ker(\Sigma) \overset{\perp}{\oplus} \(\Ker(\Sigma)\)^{\perp}$.
\item $\Ker(\Sigma) \subset (\T^{1/2}(S))^{\perp}$: indeed for all $f\in \Ker(\Sigma)$, $\langle f , \phi_i^\H \rangle = \frac{1}{\mu_i}\langle f , \Sigma \phi_i^\H \rangle = \frac{1}{\mu_i}\Sigma \langle f , \phi_i^\H \rangle =0$, and as a consequence for any $f\in \Ker(\Sigma), g \in \T^{1/2}(S)$, there exists $(g_n)\subset \text{ span}(\phi_i^H) \text{ s.t. } g_n\overset{\H}{\rightarrow} g$, thus $0=\langle g_n , f \rangle_{\H}\rightarrow \langle f,g  \rangle$ and finally $f \in (\T^{1/2}(S))^{\perp}$. Equivalently $\T^{1/2}(S) \subset \(\Ker(\Sigma)\)^{\perp}$.
\item $(\T^{1/2}(S))^{\perp} \subset \Ker(\Sigma)$. For any $i$, $\phi_i^H\in \T^{1/2}(S)$. If $f\in (\T^{1/2}(S))^{\perp}$, then $\langle p(f) , \phi_i  \rangle_{\Ld}= \langle f ,  \T \phi_i  \rangle_{\H}= 0. $ As a consequence $p(f)\in p(\H) \cap \Ker(\Td)=\left\lbrace 0 \right\rbrace $, thus $f\in \Ker(\Sigma)$. That is $(\T^{1/2}(S))^{\perp} \subset \Ker(\Sigma)$. Equivalently $\Ker(\Sigma)^{\perp} \subset(\T^{1/2}(S)) $.
\item Combining these points:  $\H = \Ker(\Sigma) \overset{\perp}{\oplus} \S.$
\end{itemize}
\end{proof}

We have two decompositions of $\L = \Ker(\Td) \overset{\perp}{\oplus} S$ and $\H= \Ker(\Sigma) \overset{\perp}{\oplus} \S$. They happen to be related through the mapping $\T^{1/2}$, which we now define.

\subsection{Properties of $\Td^{r}$, $r>0$}

We defined operators $\Td^r$, $r>0$ and $\T^r$, $r\geq 1/2$ in Section~\ref{subsec:covoper} in Definitions~\ref{def:tdr},\ref{def:Tr}.

\begin{Prop}[Properties of $\Td^{r}$, $\T^{r}$]\label{propTdr}\ 
\begin{itemize}
\item $\Td^{r}$ is well defined for any $r>0$.
\item $\T^{r}$ is well defined for any $r\geq \frac{1}{2}$. 
\item $\T^{1/2}: S \rightarrow \S$ is an isometry.  
\item \label{isomorphism} Moreover $\Image(\Td^{1/2})=p(\H)$. That means $\Td^{1/2}: S \rightarrow p(\H)$ is an isomorphism.
\end{itemize} 
\end{Prop}

\begin{proof}
\emph{$\Td^{r}$ is well defined for any $r>0$.} 

$S= \left\lbrace \sum_{i=1}^{\infty} a_i \phi_i \text{ s.t. }  \sum_{i=1}^{\infty} a_i^2 <\infty \right\rbrace$. For any sequence $(a_i)_{i\in I}$ such that $ \sum_{i=1}^{\infty} a_i^2<\infty$, $\Td^{r} (\sum a_i \phi_i) = \sum_i  \mu_i^r a_i \phi_i$ is a converging sum in the Hilbert space $\Ld$ (as $(\mu_i)_{i\in I}$ is bounded thus $ \sum_i \mu_i^r a_i \phi_i$ satisfies Cauchy is criterion: $\| \sum_{i=n}^{p} \mu_i^r a_i \phi_i  \|^2 \le \mu_0^r \(\sum_{i=n}^{p}  a_i^2 \)^{1/2}$). And Cauchy is criterion implies convergence in Hilbert spaces.

\vspace{0.5em}
 \emph{$\T^{r}$ is well defined for any $r\geq \frac{1}{2}$.  } \\
We have shown that $(\sqrt{\mu_i}\phi_i^H)_i$ is an orthonormal family in $\H$. As a consequence (using the fact that $(\mu_i)$ is a bounded sequence), for any sequence $(a_i)_i$ such that $\sum a_i^2 < \infty$,  $ \sum_i \mu_i^r a_i \phi^H_i $ satisfies Cauchy is criterion thus is converging in $\H$ as $\|\sum_{i\in I'} \mu_i^r a_i \phi^H_i\|_K= \sum_{i\in I'} \mu_i^{r-1/2} a_i^2 \le  \mu_0^{r-1/2} \sum_{i\in I'} a_i^2 <\infty $. (We need $r\geq 1/2$ of course).

\vspace{0.5em}
\emph{ $\T^{1/2}: S \rightarrow \S$ is an isometry.}\\ 
Definition has been proved. Surjectivity in  $\S$ is by definition, as $\T ^{1/2} ( S)=\left\lbrace \sum_{i\in I} a_i \phi_i^H \text{ s.t. } \sum_{i\in I} \frac{a_i^2}{\mu_i} < \infty \right\rbrace. $  Moreover, the operator is clearly injective as for any $f\in S$, $\Td  f \neq 0 $ in $\Ld$ thus $\Td  f \neq 0$ in $\H$. Moreover for any $f=\sum_{i=1}^{\infty} a_i \phi_i \in S$, $\|\Td  f\|_K^2= \|\sum_{i=1}^{\infty} a_i \sqrt{\mu_i} \phi_i\|_K^2= \sum_{i=1}^{\infty} a_i^2 =\|f\|_{\L}^2$, which is the isometrical property. 

It must be noticed that we cannot prove surjectivity in  $\H $\footnote{It is actually easy to build a counter example, f.e. with a measure of ``small'' support (let is say $[-1,1]$), a Hilbert space of functions on $\X=[-5;5]$, and a kernel like $\min(0,1- |x-y|)$: $\Image\(\!\ \T ^{1/2}\)\subset \lbrace f\in \H \text{ s. t. } \supp(f)\subset [-2;2] \rbrace \varsubsetneq \H$. }, that is without our ``strong assumptions''. However we will show that operator $\Td ^{1/2}$ is surjective in $p(\H)$. 

\vspace{0.5em}
\emph{ $\Image(\Td^{1/2})=p(\H)$. That means $\Td^{1/2}: S \rightarrow p(\H)$ is an isomorphism.}\\
 $\Image(\Td^{1/2})=p(\Image(\T^{1/2}) )= p(\S)$. Moreover $p(\H)=p(\Ker(\Sigma) \oplus \S)= p(\S).$ Consequently  $\Image(\Td^{1/2})=p(\H)$. Moreover $\Td^{1/2}: S \rightarrow \Ld$ is also injective, which give the isomorphical character.
 
 Note that it is clear that $\Td^{1/2}(S) \subset p(\H) $ and that for any $x\in \X$, $p(K_x) \in \Td^{1/2}(S) $ indeed  $p(K_x)=\sum_{i=1	}^{\infty } \langle K_x , \phi_i  \rangle_{\Ld} \phi_i=\sum_{i=1	}^{\infty } \mu_i \phi_i^H (x) \phi_i$, with $\sum_{i=1	}^{\infty } \frac{\(\mu_i \phi_i^H (x)\)^2}{\mu_i}=\sum_{i=1	}^{\infty } \mu_i \phi_i^H (x)^2< \infty, $ as $K(x,x)=\sum_{i=1}^\infty \mu_i \phi_i^H (x)^2$
\end{proof}

Finally, it has appeared that $S$ and $\S$ may be identified via the isometry $\T^{1/2}$. We conclude by a proposition which sums up the properties of the spaces $\T^r(\Ld)$.

\begin{Prop}
The spaces $\Td ^{r}(\Ld), r>0$  satisfy: 
\begin{eqnarray*}
\forall r \geq r' >0 , \ \ \Td ^{r}\(\Ld\)&\subset& \Td ^{r'}\(\Ld\)\\
\forall r >0 ,\ \ \overline{\Td ^{r}\(\Ld\)}&=&S \\
\Td ^{1/2}\(\Ld\)&=&p(\H)\\
\forall r \geq\frac{1}{2}, \ \ \Td ^{r}\(\Ld\)&\subset&p(\H)
\end{eqnarray*} 
\end{Prop}

\subsection{Kernel decomposition} \label{subsec:app_smallerrkhs}

We prove here Proposition~\ref{prop:kerneldec}.

\begin{proof}

Considering our  decomposition of $\H=\S \overset{\perp}{\oplus} \ker(\Sigma)$, an the fact the $(\sqrt{\mu_i} \phi_i^\H)$ is a Hilbertian eigenbasis of $\S$, we have for any $x\in \X$, \begin{eqnarray*}
K_x&=&\sum_{i=1}^{\infty} \langle \sqrt{\mu_i} \phi_i^\H , K_x \rangle_{\H} \sqrt{\mu_i} \phi_i^\H + g_x\\
&=& \sum_{i=1}^{\infty} \mu_i \phi_i^\H(x) \phi_i^\H + g_x
\end{eqnarray*}

And as it has been noticed above this sum is converging in $\S$ (as in $\H$) because $\sum_{i=1}^{\infty} \frac{(\mu_i \phi_i^\H(x) )^2}{\mu_i}=\sum_{i=1}^{\infty} \mu_i (\phi_i^\H(x))^2=K(x,x) <\infty$. However, the convergence may not be absolute in $\H$.
Our function $g_x$ is in $\Ker(\Sigma)$, which means $\int_{y\in \X} g_x(y)^2 d\rho_X(y) =0$.

And as a consequence, we have for all $x,y \in \X$,
$$
K(x,y) = \sum_{i \in I} \mu_i \phi_i^H(x) \phi_i^H(y) + g(x,y),
$$
With $g(x,y):= g_x(y)$. Changing roles of $x,y$, it appears that $g(x,y)=g(y,x)$. And we have for all $x\in \X$,  $\int_{\X} g(x,y)^2 d\rho_X(y) = 0$. Moreover, the convergence of the series is absolute

We now prove the following points\begin{itemize}
\item[(a)]  $(\S, \|\cdot\|_{\H})$ is also an RKHS, with kernel $K^{\S}:(x,y) \mapsto \sum_{i \in I} \mu_i \phi_i^H(x) \phi_i^H(y)$
\item[(b)] given the decomposition above, almost surely the optimization problem in $\mathcal{S}$ and $\mathcal{H}$ have equivalent solutions.
\end{itemize}

\textbf{(a)} $(\S, \|\cdot\|_{\H})$ is a Hilbert space as a closed subspace of a Hilbert space.  Then for any $x\in \X$ : $K^{\S}_{x} := (y \mapsto K^{\S}(x,y)) =  \sum_{i=1}^{\infty} \mu_i \phi_i^\H(x) \phi_i^\H   \in \S$. Finally, for any $f\in \S$ 
$$\langle f , K^{\S}_{x} \rangle_{\H} = \langle f , K^{\S}_{x} +g_x \rangle_{\H}= \langle f , K_{x} \rangle_{\H}= f(x),$$ 
 because $g_x\in \Ker(\Sigma) = \S^{\perp}\ni f$. Thus stands the reproducing property.
 
 \textbf{(b)} We have that $p(\S)=p(\H)$ and our best approximating function is a minimizer over this set. Moreover if $K^{\S}_{x} $ was used instead of $K_{x}$ in our algorithm, both estimators  are almost surely almost surely equal (i.e., almost surely in the same equivalence class). Indeed, at any step $n$, if we denote $g^{\S}_n$ the sequence built in $\S$ with $K^{\S}$, if we have $g_n^{\S}\overset{a.s.}{=}g_n$, then almost surely $g_n^{\S}(x_n)= g_n(x_n)$ and  moreover $K_{x_n}\overset{a.s.}{=} K^{S}_{x_n}$. Thus almost surely, $g_{n+1}\overset{a.s.}{=} g_{n+1}^{\S}$.

\end{proof}

\subsection{Alternative assumptions}\label{subsec:strongass}
\label{sec:strongerassumptions}
As it has been noticed in the paper, we have tried to minimize assumptions made on $\X $ and $K$. In this section, we review some of the consequences of such assumptions.

\subsubsection{Alternative assumptions}

The following have been considered previously:
\begin{enumerate}
\item Under the assumption that \textit{$\rho$ is a Borel probability measure} (with respect with some topology on $\R^{d}$) and $\X$ is a closed space, we may assume that $\supp(\rho)=\X$, where $\supp(\rho)$ is the smallest close space of measure one.
\item The assumption that \textit{$ K$ is a Mercer kernel} ($\X$ compact, $K$ continuous) has generally been made before \cite{tar2011online,sma2007learning,Sma2001mathematical,yin2008online}, but does not seem to be necessary here.
\item \textbf{(A2)} was replaced by the stronger assumption $\sup_{x\in \X} K(x,x)<\infty$  \
\cite{tar2011online,yin2008online,ros2014regularisation} and $|Y|$ bounded \cite{tar2011online,ros2014regularisation} .
 \end{enumerate}

\subsubsection{Identification $\H$ and $p(\H)$}

Working with mild assumptions has made it necessary to work with sub spaces of $\Ld$, thus projecting $\H$ in $p(\H)$.  With  stronger assumptions given above, the space $\H$ may be identified with $p(\H)$.

Our problems are linked with the fact that a function $f$ in $\H$ may satisfy both $\|f\|_\H\neq 0 $ and $\|f\|_{\Ld}=0$.
 \begin{itemize}
\item the ``support'' of $\rho$ may not be $\X$. 
\item even if the support is $\X$, a function may be $\rho$-a.s. 0 but not null in $\H$. 
\end{itemize}

Both these ``problems'' are solved considering the further assumptions above. We have the following Proposition: 

\begin{Prop}
If we consider a Mercer kernel $K$ (or even any continuous kernel), on a space $\X$ compact and a measure $\rho_{X}$ on $\X$ such that $\supp(\rho)=\X$ then the map: \begin{eqnarray*}
p: \H &\rightarrow& p(\H)\\
f &\mapsto& \tilde{f}
\end{eqnarray*}
is injective, thus bijective.
\end{Prop}

\subsubsection{Mercer kernel properties}\label{subsec:app-mercerker}

We review here some of the properties of Mercer kernels, especially Mercer's theorem which may be compared to Proposition~\ref{prop:kerneldec}.

\begin{Prop}[Mercer theorem]
Let $\X$ be a compact domain or a manifold, $\rho$ a Borel measure on $\X$, and $K: \X \times \X \rightarrow \R$ a Mercer Kernel. Let $\lambda_k$ be the $k$-th eigenvalue of $\Td $ and $\Phi_k$ the corresponding eigenvectors. For all $x,t \in \X$, $K(x,t)= \sum_{k=1}^{\infty} \lambda_k \Phi_k(x)\Phi_k(t)$ where the convergence is absolute (for each $x,t\in \X^2$) and uniform on $\X \times \X$.
\end{Prop}

The proof of this theorem is given in \cite{hoc1973integral}.

\begin{Prop}[Mercer Kernel properties]
In a Mercer kernel, we have that:
\begin{enumerate}
\item $C_K:=\sup_{x,t\in \X^2}(K(x,t)) < \infty$.
\item $\forall f \in \H$, $f$ is $C^0$.
\item The sum $\sum \lambda_k$ is convergent and 
$\sum_{k=1}^{\infty}\lambda_k = \int _{X} K(x,x) \le \rho(\X) C_K$.
\item The inclusion $I_K~: \H \rightarrow C(\X)$ is bounded with $|||I_K|||\le C_K^{1/2}$. 
\item The map \begin{eqnarray*}
\Phi~: \X &\rightarrow& \ell^{2}\\
x &\mapsto& (\sqrt{\lambda_k} \Phi_k(x))_{k\in \N}
\end{eqnarray*} is well defined, continuous, and satisfies $K(x, t )=\langle \Phi_k(x), \Phi_k(t)\rangle$.
\item The space $\H $ is independent of the measure considered on $\X$.
\end{enumerate}
\end{Prop}

We can characterize $\H$ via the eigenvalues-eigenvectors:
$$\H=\left\lbrace f\in \Ld| f=\sum_{k=1}^{\infty} a_k \Phi_k \text{ with } \sum_{k=1}^{\infty} \(\frac{a_k}{\sqrt{\lambda_k}}\)^2 <\infty \right\rbrace.$$

Which is equivalent to saying that $\Td ^{1/2}$ is an isomorphism between $\Ld$ and $\H$. Where we have only considered $\lambda_k >0$. It has no importance to consider the  linear subspace $S$ of $\Ld$ spanned by the eigenvectors with non zero eigenvalues. However it changes the space $\overline{\H}$ which is in any case $S$, and is of some importance regarding the estimation problem.

\section{Proofs}\label{A_proofs}
To get our results, we are going to derive from our recursion a new error decomposition and bound the different sources of error via algebraic calculations. We first make a few remarks on short notations that we will use in this part and difficulties that arise from the Hilbert space setting in Section~\ref{subsec:preliminary_remarks}, then provide intuition via the analysis of a closely related recursion in Section~\ref{subsec:semisto_analysis}. We give in Sections~\ref{subsec:completeproofFH},~\ref{subsec:completeproofONL} the complete proof of our bound respectively in the finite horizon case (Theorem~\ref{prop.dinf.rand}) and the online case (Theorem~\ref{prop.dinf.rand.onl}). We finally provide technical calculations of the main bias and variance terms in Section~\ref{subsec:technicalcal}.

\subsection{Preliminary remarks}
\label{subsec:preliminary_remarks}

We remind that we consider a sequence of functions $(g_n)_{n\in \N}$ satisfying the system defined in Section~\ref{sec:lms}.
\begin{eqnarray}
g_0 &=& 0 \text{ (the null function), }  \nonumber \\
 g_n&=& \sum_{i=1}^{n} a_i K_{x_i} . \nonumber
\end{eqnarray}

With a sequence $(a_n)_{n\geq 1}$ such that for all $n$ greater than 1~:\begin{equation}
a_n:= -\gamma_n(g_{n-1} (x_n) - y_n) = -\gamma_n\left(\sum_{i=1}^{n-1} a_i K(x_n,x_i) - y_i\right).
\end{equation}   We output \begin{equation}
\overline{g}_n= \frac{1}{n+1} \sum_{k=0}^n \overline{g}_k.
\end{equation}

We  consider a representer $g_\H \in \L$ of $g_\H$ defined by Proposition~\ref{prop:def_approximation_function}. We accept to confuse notations as far as our calculations are made on $\Ld$-norms, thus does not depend on our choice of the representer.

We aim to estimate~: $$\epsilon(\overline{g}_n)-\epsilon(g_\H)= \|\overline{g}_n-g_\H\|_{\Ld}^2.$$

\subsubsection{Notations}
In order to simplify reading, we will use some shorter notations~:

\begin{itemize}
 \item For the covariance operator, we will only use $\Sigma$ instead of $\Sigma, \Td, \T$,
 \end{itemize}

\vspace{0.5em}
\begin{center}
\begin{tabular}{|c|c|}
\hline
Space~:& $ \mathcal{H}$\\
Observations~:& $(x_n,y_n)_{n\in \N} \text{ i.i.d. } \sim \rho$\\
Best approximation function~:  & $g_\H$ \\
Learning rate~: & $(\gamma_i)_i$\\
\hline
\end{tabular}
\end{center}

All the functions may be split up the orthonormal eigenbasis of the operator $\T$. We can thus see any function as an infinite-dimensional vector, and operators as matrices. This is of course some (mild) abuse of notations if we are not in finite dimensions. For example, our operator $\Sigma$ may be seen as $\text{ Diag}(\mu_i)_{1\le i}$. Carrying on the analogy with the finite dimensional setting, a self adjoint operator, may be seen as a symmetric matrix.

We will have to deal with several ``matrix products'' (which are actually operator compositions). We denote~:
\begin{eqnarray*}
M(k, n, \gamma) &=& \prod_{i=k}^{n} (I- \gamma K_{x_i} \otimes K_{x_i}) =   (I- \gamma K_{x_k} \otimes K_{x_k}) \cdots (I- \gamma K_{x_n} \otimes K_{x_n})\\
M(k,n, (\gamma_i)_i) &=& \prod_{i=k}^{n} (I- \gamma_i K_{x_i} \otimes K_{x_i}) \\
D(k,n, (\gamma_i)_i) &=& \prod_{i=k}^{n} (I- \gamma_i \Sigma)
\end{eqnarray*}

Remarks~:
\begin{itemize}
\item As our operators may not commute, we use a somehow unusual convention by defining the products for any $k, n$, even with $k>n$ , with $M(k, n, \gamma) = (I- \gamma K_{x_k} \otimes K_{x_k}) (I- \gamma K_{x_{k-1}} \otimes K_{x_{k-1}})\cdots (I- \gamma K_{x_n} \otimes K_{x_n})$.
\item We may denote $D(k,n, \gamma) = \prod_{i=k}^{n} (I- \gamma \Sigma) $ even if its clearly $(I- \gamma \Sigma)^{n-k+1} $ just in order to make the comparison between equations easier.
\end{itemize}

\subsubsection{On norms}
In the following, we will use constantly the following observation~:

\begin{Lem}\label{lem.fondamental}
Assume \textbf{A2-4} , let $\n{n}=\t{n}-{g_\H}$, $\nb{n}=\tb{n}-{g_\H}$~:
\begin{eqnarray*}
\epsilon(\t{n})-\epsilon({g_\H})&=&\langle \n{n}, \Sigma \n{n} \rangle  =\E\left[ \langle x, \t{n}-{g_\H} \rangle ^2 \right] \left(:=  \| \t{n}-{g_\H}\|^2_{\Ld}\right) ,\\
 \epsilon\left(\tb{n}\right)-\epsilon({g_\H}) &=&\langle \nb{n}, \Sigma \nb{n} \rangle.
\end{eqnarray*}
\end{Lem}

\subsubsection{On symmetric matrices}
One has to be careful when using auto adjoint operators, especially when using the order $A\lec B$ which means that $B-A$ is non-negative.

Some problems may arise when some self adjoint $A, B $  do not commute, because then $A B $ is not even in auto adjoint. 
It is also hopeless to compose such relations~: for example $A\lec B $ does not imply $A^2\lec B^2 $ (while the opposite is true).

However, it is true that if $A \lec B$, then for any $C$ in $S_n(\mathbb{R})$, we have $C^t \, A C \lec C^t \, B C$. We will often use this final point. Indeed for any $x $, $x^t \, (C^t \, B C- C^t \, A C)x = (Cx)^t  ( B- A)  (Cx)  \geq 0$.

\subsubsection{Notation}
In the proof, we may use, for any $x\in \H$:
\begin{eqnarray*}
 \widetilde{K_x\otimes K_x}~: \Ld &\rightarrow& \H\\
 f &\mapsto& f(x) \  K_x. 
 \end{eqnarray*} 
 
We only consider functions $\mathcal{L}^2_{\rho_X}$, which are well defined at any point. The regression function is only almost surely defined but we will consider a version of the function in $\mathcal{L}^2_{\rho_X}$.

The following properties clearly hold~:
\begin{itemize}
\item $\widetilde{K_x\otimes K_x}_{| \H} ={K_x\otimes K_x}$
\item $\E\(\widetilde{K_x\otimes K_x}\)= \T $
\item $\E\({K_x\otimes K_x}\)= \Sigma$ as it has been noticed above.
\end{itemize}

For some $x\in \mathcal{X}$, we may denote $x\otimes x~:= K_x \otimes K_x$. Moreover, abusing notations, we may forget the $\sim$ in many cases.

\subsection{Semi-stochastic recursion - intuition}
\label{subsec:semisto_analysis}

We remind that~: $$ g_n=(I-\gamma \x)g_{n-1} +\gamma y_n K_{x_n} ,$$ with $g_0 =0$. We have denoted $\Xi_n = (y_n - g_\H(x_n)) K_{x_n}$.
Thus  $ y_n K_{x_n} = {g_\H} (x_n) K_{x_n} + \Xi_n \stackrel{\text{def}}{=} \widetilde{\x} {g_\H} + \Xi_n $, and our recursion may be rewritten~:
\begin{equation}
 g_n- {g_\H}=(I-\gamma \widetilde{\x})(g_{n-1}- {g_\H})  +\gamma  \Xi_n ,
\end{equation}

\textbf{ Finally, we are studying a sequence $ (\n{n})_n$  defined by~:}
\begin{eqnarray} 
 \n{0}&=& {g_\H} , \nonumber \\
 \eta_n&=&(I-\gamma_n \widetilde{\x})\eta_{n-1} +\gamma_n \Xi_n.\label{def.eta.n}
\end{eqnarray}

\textbf{Behaviour}~: It appears that to understand how this will behave, we may compare it to the following recursion,  which may be described as a ``semi-stochastic'' version of \eqref{def.eta.n}~: we keep the randomness due to the noise $\Xi_n$ but forget the randomness due to sampling by replacing $\widetilde{\x}$ by its expectation $\Sigma$ ($\Td$, more precisely)~:
\begin{eqnarray} 
 \n{0}^{ssto}&=& {g_\H} , \nonumber\\
 \eta_n^{ssto}&=&(I-\gamma_n \Sigma)\eta^{ssto}_{n-1} +\gamma_n \Xi_n.\label{def.eta.n.semi-sto}
\end{eqnarray}

\textbf{Complete proof}~: This comparison will give an interesting insight and the main terms of bias and variance will appear if we study \eqref{def.eta.n.semi-sto}. However this is not the true recursion~: to get Theorem \ref{prop.dinf.rand}, we will have to do a bit of further work~: we will first separate the error due to the noise from the error due to the initial condition, then link the true recursions to their ``semi-stochastic'' counterparts to make the variance and bias terms appear. That will be done in Section \ref{subsec:completeproofFH}. 

\vspace{1em}
\textbf{Semi-stochastic recursion~:}
In order to get such intuition, in both the finite horizon and on-line case, we will begin by studying the semi-stochastic equation \eqref{def.eta.n.semi-sto}.

 First, we have, by induction: 
\begin{eqnarray*}
\forall j\geq 1 \qquad \eta^{ssto}_j &=&(I-\gamma_j \Sigma) \eta^{ssto}_{j-1} + \gamma_j \Xi_j.\\
\eta^{ssto}_j &=& \left[\prod_{i=1}^j (I-\gamma_i \Sigma) \right]\eta^{ssto}_{0} + \sum_{k=1}^j \left[ \prod_{i=k+1}^j (I-\gamma_i \Sigma)\right] \gamma_k \Xi_k\\
\eta^{ssto}_j &=& D(1, j, (\gamma_i)_i )\eta^{ssto}_{0} + \sum_{k=1}^j D(k+1, j, (\gamma_i)_i ) \gamma_k \Xi_k\\
\overline{\eta}^{ssto}_n &=& \frac{1}{n} \sum_{j=1}^n D(1, j, (\gamma_i)_i ) \eta^{ssto}_{0} + \frac{1}{n} \sum_{j=1}^n  \sum_{k=1}^j D(1, j, (\gamma_i)_i ) \gamma_k \Xi_k.
\end{eqnarray*} 

Then~:
\begin{eqnarray}
\E \| \overline{\eta}^{ssto}_n\|^2_{\Ld} &=&\frac{1}{n^2} \E \| \sum_{j=1}^n D(1, j, (\gamma_i)_i ) {g_\H} +  \sum_{j=1}^n  \sum_{k=1}^j D(k+1, j, (\gamma_i)_i ) \gamma_k \Xi_k\|_{\Ld} \nonumber\\
&=& \frac{1}{n^2} \underbrace{\E \| \sum_{j=1}^n D(1, j, (\gamma_i)_i ) {g_\H} \|_{\Ld}}_{\Bias(n)} \nonumber\\
&+& \underbrace{2 \frac{1}{n^2} \E \langle \sum_{j=1}^n D(1, j, (\gamma_i)_i ) {g_\H},  \sum_{j=1}^n  \sum_{k=1}^j D(k+1, j, (\gamma_i)_i ) \gamma_k \Xi_k\rangle_{\Ld} }_{=0   \text{ by \eqref{orthogonalitedegh} }, }  \nonumber\\
&+& \underbrace{\frac{1}{n^2} \E \|  \sum_{j=1}^n  \sum_{k=1}^j D(k+1, j, (\gamma_i)_i ) \gamma_k \Xi_k\|_{\Ld} }_{\Var(n)} \label{B+Vssto}
\end{eqnarray}

In the following, all calculations may be driven either with $\|\Sigma^{1/2} \cdot\|_K$ or in $\|\cdot\|_{\Ld}$ using the isometrical character  of $\Sigma^{1/2}$. In order to simplify comparison with existing work and especially~\citep{bac2013nonstrongly}, we will mainly use the former as all calculations are only algebraic sums, we may  sometimes use the notation $\langle x, \Sigma x\rangle_{H} $ instead of $\| \Sigma^{1/2} x\|_\H^{2} $. It is an abuse if $x\notin \H$, but however does not induce any confusion or mistake. In the following, if not explicitely specified, $\|\cdot \|$ will denote $\|\cdot \|_K$. 

In the following we will thus denote : 
\begin{eqnarray*}
\Bias\Big(n, (\gamma_i)_i, \Sigma, {g_\H} \Big)&=&\frac{1}{n^2} \E \bigg\| \Sigma^{1/2} \sum_{j=1}^n \left[\prod_{i=1}^j (I-\gamma_i \Sigma) \right]{g_\H} \bigg\|_K^{^2} \\
\Var\Big(n, (\gamma_i)_i, \Sigma, (\Xi_i)_i\Big)&=&\frac{1}{n^2} \E \bigg\| \Sigma^{1/2}  \sum_{j=1}^n  \sum_{k=1}^j \left[ \prod_{i=k+1}^j (I-\gamma_i \Sigma)\right] \gamma_k \Xi_k \bigg\|_K^{^2}.
\end{eqnarray*}

In  section \ref{some_quantities}  we will prove the following Lemmas which upper bound these bias and variance terms under different assumptions~:

\begin{enumerate}
\item $\Bias\Big(n, \gamma, \Sigma, g_\H \Big)$ if we assume \textbf{A3,4}, $\gamma$ constant,
\item  $\Var\Big(n, \gamma, \Sigma, (\Xi_i)_i\Big)$ if we assume \textbf{A3,6}, $\gamma$ constant,
\item $\Bias\Big(n, (\gamma_i)_i, \Sigma, {g_\H} \Big)$ if we assume \textbf{A3,4} and $\gamma_i=\frac{1}{n^\zeta}, \ \ 0 \le \zeta \le 1$,
\item $\Var\Big(n, (\gamma_i)_i, \Sigma, (\Xi_i)_i\Big)$ if we assume \textbf{A3,6} and $\gamma_i=\frac{1}{n^\zeta}, \ \ 0 \le \zeta \le 1$.
\end{enumerate}

 The two terms show respectively the impact~:
 \begin{enumerate}
  \item of the initial setting and the hardness to forget the initial condition,
  \item the noise.
  \end{enumerate}  
   Thus the first one tends to decrease when $\gamma$ is increasing, whereas the second one increases when $\gamma$ increases. We understand we may have to choose our step $\gamma$ in order to optimize the trade-off between these two factors.
 
 \vspace*{0.2em}
 In the finite-dimensional case, it results from such a decomposition that if $C=\sigma^2 \Sigma$ then $\E \left[\big\langle \overline{\alpha}_{n-1}, \Sigma\overline{\alpha}_{n-1} \big\rangle\right] \le \frac{1}{n\gamma} \|\alpha\|_0^2 + \frac{\sigma^2 d}{n}$, as this upper bound is vacuous when $d$ is either large or infinite, we can derive comparable bounds in the infinite-dimensional setting under our assumptions  \textbf{A3,4,6}.
   
\begin{Lem}[Bias, \textbf{A3,4}, $\gamma$ const.] \label{bias_gam_const}
Assume \textbf{A3-4} and let $\alpha$ (resp. $r$) be the constant in \textbf{A3} (resp. \textbf{A4})~:

 \textbf{ If $ r\le 1 $}~: $$ \Bias\Big(n, \gamma, \Sigma, {g_\H} \Big) \le  \|\Sigma^{-r} {g_\H}\|_{\Ld}^2  \left( \frac{1}{(n\gamma)^{2r}} \right) \stackrel{\text{not}}{=}  \bias( n, \gamma, r).$$ 
 
  \textbf{ If $ r \geq 1 $}~: $$ \Bias\Big(n, \gamma, \Sigma, {g_\H} \Big) \le  \|\Sigma^{-r} {g_\H}\|_{\Ld}^2  \left( \frac{1}{n^{2}\gamma^r} \right) \stackrel{\text{not}}{=}  \bias( n, \gamma, r).$$ 
\end{Lem}

\begin{Lem}[Var, \textbf{A3,4}, $\gamma$ const] \label{var_gam_const}
 Assume \textbf{A3,6},  let $\alpha,s$ be the constants in \textbf{A3}, and $\sigma$ the constant in \textbf{A6} (so that $\E \left[\Xi_n \otimes \Xi_n\right] \lec \sigma^2 \Sigma$).

  \begin{equation*}
  \Var\Big(n, \gamma, \Sigma, (\Xi_i)_i\Big)\le C(\alpha) \ s^{2/\alpha}\  \sigma^2  \frac{\gamma^{\frac{1}{\alpha}}}  {n^{1-\frac{1}{\alpha}}} + \frac{\sigma^2}{n}  \stackrel{\text{not}}{=}  \var( n, \gamma, \sigma^2, r, \alpha),
  \end{equation*}
with $C(\alpha)=\frac{2 \alpha^2 }{(\alpha+1)(2\alpha-1)}$.
 \end{Lem} 

\begin{Lem}[Bias, \textbf{A3,4}, $(\gamma)_i$ ] \label{bias_gam_var}
Assume \textbf{A3-4} and let $\alpha$ (resp. $r$) be the constant in \textbf{A3} (resp. \textbf{A4}). Assume we consider a sequence $\gamma_i=\frac{\gamma_0}{i^\zeta}$ with $0<\zeta<1$ then~:
\begin{enumerate}
\item if $r(1-\zeta)<1$:
\begin{eqnarray*}
\Bias\Big(n, (\gamma_i)_i, \Sigma, {g_\H} \Big)
&=& O\( \|\Sigma^{-r}{g_\H}\|_{\Ld}^2 n^{-2r(1-\zeta)}\)\\
&=& O\( \|\Sigma^{-r} {g_\H}\|_{\Ld}^2 \frac{1}{(n\gamma_n)^{2r}}\),
\end{eqnarray*}
\item if $r(1-\zeta)>1$:
\begin{eqnarray*}
\Bias\Big(n, (\gamma_i)_i, \Sigma, {g_\H} \Big)&
=&  O\(\frac{1}{n^2}\).
\end{eqnarray*}
\end{enumerate}

 \end{Lem}

\begin{Lem}[Var, \textbf{A3,4}, $(\gamma)_i$ ] \label{var_gam_var} 
Assume \textbf{A3,6},  let $\alpha,s$ be the constants in \textbf{A3}, and $\sigma$ the constant in \textbf{A6} .
If we consider a sequence $\gamma_i=\frac{\gamma_0}{i^\zeta}$ with $0<\zeta<1$ then~:
\begin{enumerate}
\item if $ 0<\zeta<\frac{1}{2} $ then
\begin{equation*} \Var\Big(n, (\gamma_i)_i, \Sigma, (\Xi_i)_i\Big) = O\( n^{-1+ \frac{1-\zeta}{\alpha}}\) =O\( \frac{\sigma^2 (s^2\gamma_n)^{\frac{1}{\alpha}}}{n^{1-\frac{1}{\alpha}}}\),
\end{equation*}
\item and if $ \zeta>\frac{1}{2} $ then 
\begin{equation*}\Var\Big(n, (\gamma_i)_i, \Sigma, (\Xi_i)_i\Big) =  O\( n^{-1+ \frac{1-\zeta}{\alpha}+ 2\zeta-1 } \).
\end{equation*}
\end{enumerate}
 \end{Lem}

Those Lemmas are proved in section \ref{some_quantities}.

Considering decomposition  \eqref{B+Vssto} and our Lemmas above, we can state a first Proposition.

\begin{Prop}[Semi-stochastic recursion]\label{prop.semisto} Assume  \textbf{A1-6}. Let's consider the semi-stochastic recursion (that is the sequence~: $\eta_n=(I-\gamma_n \Sigma)\eta _{n-1} + \gamma_n \Xi _n$) instead of our recursion initially defined. In the finite horizon setting, thus with $\gamma_i = \gamma$ for $i\le n$, we have :

\begin{equation*}
\frac{1}{2} \E \left[ \epsilon\left(\overline{g}_n\right)-\epsilon(g_\rho) \right] \le  C(\alpha) \ s^{\frac{2}{\alpha}}\   \sigma^2  \frac{\gamma^{\frac{1}{\alpha}}}  {n^{1-\frac{1}{\alpha}}} + \frac{\sigma^2}{n}+ \|\Sigma^{-r} g_\rho\|_{\Ld}^2  \left( \frac{1}{n^{2\min\{r,1\}}\gamma^{2r}} \right) .
 \end{equation*}

\end{Prop} 

Theorem \ref{prop.dinf.rand} must be compared to  Proposition~\ref{prop.semisto}~:  Theorem \ref{prop.dinf.rand}  is just an extension but with the true stochastic recursion instead of the semi-stochastic one.

We finish this first part by a very simple Lemma which states that what we have done above is true for any semi stochastic recursion under few assumptions. Indeed, to get the complete bound, we will always come back to semi-stochastic type recursions, either without noise, or with a null initial condition.
 
 \begin{Lem}\label{lem.ssto.rec}
 Let's assume:
\begin{enumerate}
\item $\alpha_n=(I-\gamma \Sigma) \alpha_{n-1} + \gamma \Xi^\alpha_n$, with $\gamma \Sigma\lec I$.
\item $(\Xi^\alpha_n) \in \H $ is $\mathcal{F}_n$ measurable for a sequence of increasing $\sigma$-fields $ \left( \mathcal{F}_n\right) $.
\item  $ \E\left[\Xi^\alpha_n|\mathcal{F}_{n-1}\right]=0$,  $ \E\left[\|\Xi^\alpha_n\|^2|\mathcal{F}_{n-1}\right]$ is finite and $ \E\left[\Xi^\alpha_n\otimes \Xi_n^{\alpha}\right]\lec \sigma^2_\alpha \Sigma$.
\end{enumerate}

 Then~:
 \begin{equation}
 \E \left[\big\langle \overline{\alpha}_{n-1}, \Sigma\overline{\alpha}_{n-1} \big\rangle\right] = \Bias\Big(n, \gamma, \Sigma, \alpha_0 \Big)+  \Var\Big(n, \gamma, \Sigma, (\Xi^\alpha_i)_i\Big).
\end{equation}
And we may apply Lemmas \ref{bias_gam_const} and \ref{var_gam_const} if we have good assumptions on $ \Sigma, \alpha_0$.
 \end{Lem}

\subsection{Complete proof, Theorem \ref{prop.dinf.rand} (finite horizon) } 
\label{subsec:completeproofFH}
In the following, we will focus on the finite horizon setting, i.e., we assume the step size  is constant, but may depend on the total number of observations $n$ : for all $1 \le i \le n$, $\gamma_i=\gamma=\Gamma(n)$.
The main idea of the proof is to be able to~:
\begin{enumerate}
\item separate the different sources of error (noise \& initial conditions), 
\item then bound the difference between the stochastic recursions and their semi-stochastic versions, a case in which we are able to compute bias and variance as it is done above.
\end{enumerate}

Our main tool will be the Minkowski's inequality, which is the triangular inequality for $\E \left(\|\cdot \|_{\Ld} \right)$. This will allow us to separate the error due to the noise from the error due to the initial conditions. The sketch of the decomposition is given in Table~\ref{tab:FH}.

\vspace{1em}

\begin{table}
\makebox[\textwidth][c]{
\begin{tabular}{|ccccc|}
   \hline
  &    \multicolumn{3}{c}{Complete recursion $\eta_n$ } &      \\ 
   
     &    $\swarrow$ &   & $\searrow$ &      \\ 
   
    \multicolumn{2}{|c}{variance term $\eta^{noise}_n$} &  | &     \multicolumn{2}{c|}{ bias term $\eta^{init}_n$}  \\ 
   
    \multicolumn{2}{|c}{$\downarrow$  } &  | &    \multicolumn{2}{c|}{ $\downarrow$ }   \\ 
   
    \multicolumn{2}{|c}{multiple recursion } &  | &   \multicolumn{2}{c|}{ semi stochastic variant}   \\ 
   
   $\swarrow$ &    $\searrow$ & |  & $\swarrow$ &   $\searrow$ \\ 
   
   main terms $\eta^r_n$, $r\geq 1$ &    residual term $\eta^{noise}_n - \sum \eta^r_n$ &  | & main term $\eta^0_n $ &    residual term $\eta^{init}_n -  \eta^0_n$ \\ 
   
  satisfying semi-sto recursions &    satisf. stochastic recursion & |  & satisf. semi-sto recursion &    satisf. stochastic recursion \\ 
   
   $\downarrow$ Lemma \ref{lem.ssto.rec} &    $\downarrow$ Lemma \ref{lem.stoch.rec} &  | & $\downarrow$  &    $\downarrow$ Lemma \ref{lem.stoch.rec} \\ 
   
   $\le C$ Variance term &    $\rightarrow_{r\rightarrow \infty} 0$ & |  & $\le $ Bias term &    residual negligible term \\ 
  
  &&&& \\
    \multicolumn{2}{|c}{ \ $\qquad\qquad$ Lemma \ref{var_gam_const} $ \searrow$} &   & \multicolumn{2}{c|}{$\swarrow$ Lemma \ref{bias_gam_const}}     \\ 
 
     &    \multicolumn{3}{c}{Theorem \ref{prop.dinf.rand}} &      \\ 
   \hline
   
\end{tabular} 

} \vspace{0.5em}\caption{Error decomposition in the finite horizon setting.}\label{tab:FH}
\end{table}

 We remind that $ (\n{n})_n$ is defined by~:
 $$ \n{0}= {g_\H} , \mbox{ and the recusion } \eta_n=(I-\gamma \x)\eta_{n-1} +\gamma \Xi_n.$$

\subsubsection{A Lemma on stochastic recursions}
Before studying the main decomposition in Section \ref{main.decomp} we must give a classical Lemma on stochastic recursions which will be useful below~:

 \begin{Lem}\label{lem.stoch.rec}
Assume $(x_n, \Xi_n) \in \H \times \H$ are $\mathcal{F}_n$ measurable for a sequence of increasing $\sigma$-fields $ \left( \mathcal{F}_n\right) $. Assume that $ \E\left[\Xi_n|\mathcal{F}_{n-1}\right]=0$,  $ \E\left[\|\Xi_n\|^2|\mathcal{F}_{n-1}\right]$ is finite and $ \E\left[\|K_{x_n}\|^2 \x|\mathcal{F}_{n-1}\right]\lec R^2 \Sigma$, with $ \E \left[K_{x_n} \otimes K_{x_n} | \mathcal{F}_{n-1}\right]=\Sigma$ for all $ n\geq 1 $ , for some $R>0$ and invertible operator $\Sigma$. Consider the recursion $\alpha_n=(I-\gamma K_{x_n} \otimes K_{x_n}) \alpha_{n-1} + \gamma \Xi_n$, with $\gamma R^2\le 1$. Then~: $$(1-\gamma R^2) \,\E\left[\big\langle \overline{\alpha}_{n-1}, \Sigma\overline{\alpha}_{n-1} \big\rangle\right] + \frac{1}{2n\gamma}\E\|\alpha_n\|^2 \le \frac{1}{2n\gamma} \|\alpha_0\|^2 + \frac{\gamma}{n} \sum_{k=1}^n \E \|\Xi_k\|^2.$$
Especially, if $\alpha_0=0$, we have $$ \E\left[\big\langle \overline{\alpha}_{n-1}, \Sigma\overline{\alpha}_{n-1} \big\rangle\right]  \le \frac{1}{(1-\gamma R^2)}  \frac{\gamma}{n} \sum_{k=1}^n \E \|\Xi_k\|^2.$$
 \end{Lem} 
 
Its proof may be found in \cite{bac2013nonstrongly}~: it is a direct consequence of the classical recursion to upper bound $\|\alpha_n\|^2$.

\subsubsection{Main decomposition} \label{main.decomp}
 \noindent
 
 We consider~:
 \begin{enumerate}
   \item  $(\n{n}^{init})_n $ defined by~: $$ \n{0}^{init} = {g_\H} \mbox{ and } \n{n}^{init}=(I-\gamma \x)\n{n-1}^{init} .$$ 
 $ \n{n}^{init} $ is the part of $ (\n{n})_n $ which is due to the \textbf{initial conditions} ( it is equivalent to assuming $\Xi_n\equiv 0$).
 
\item Respectively, let $(\n{n}^{noise})_n $ be defined by~: $$ \n{0}^{noise} =0 \mbox{ and } \n{n}^{noise}=(I-\gamma \x)\n{n-1}^{noise}   +\gamma \Xi_n.$$
 $ \n{n}^{noise} $ is the part of $ (\n{n})_n $ which is due to \textbf{the noise}.
 \end{enumerate}
 
 \vspace{0.5em}
 A straightforward induction shows that for any $n$, $ \n{n}=\n{n}^{init}+\n{n}^{noise}$ and $ \nb{n}=\nb{n}^{init}+\nb{n}^{noise}$. Thus Minkowski's inequality, applied to $\left(\E\left[ \|\cdot\|^2_{\Ld}\right]\right)^{1/2}$, leads to~:
 \begin{equation*}
 \left(\E\left[\| \nb{n}\|^2_{\Ld} \right]\right)^{1/2}\le \left(\E\left[\| \nb{n}^{noise}\|^2_{\Ld}\right]\right)^{1/2}+\left(\E\left[\| \nb{n}^{init}\|^2_{\Ld}\right]\right)^{1/2}
 \end{equation*} 
 \begin{equation} 
\left(\E\left[\langle \nb{n}, \Sigma \nb{n} \rangle \right]\right) ^{1/2} \le \left(\E\left[\langle \nb{n}^{noise}, \Sigma \nb{n}^{noise} \rangle \right]\right) ^{1/2} + \left(\E\left[\langle \nb{n}^{init}, \Sigma \nb{n}^{init} \rangle \right] \right)^{1/2}. \label{init+sto}
 \end{equation}

That means we can always consider separately  the effect of the noise and the effect of the initial conditions. We'll first study $ \n{n}^{noise} $ and then $ \n{n}^{init} $.

\subsubsection{Noise process}
We remind that $(\n{n}^{noise})_n $ is defined by~: 
\begin{equation} \label{sto.noise}
 \n{0}^{noise} =0 \mbox{ and } \n{n}^{noise}=(I-\gamma \x)\n{n-1}^{noise}   +\gamma \Xi_n.
 \end{equation}

We are going to define some other sequences, which are defined by the following ``semi-stochastic'' recursion, in which $\x$ has been replaced be its expectancy $\Sigma$~: first we define  $\left(\n{n}^{noise, 0}\right)_n$ so that
  $$ \n{0}^{noise, 0} =0 \mbox{ and } \n{n}^{noise, 0}=(I-\gamma \Sigma)\n{n-1}^{noise, 0}   +\gamma \Xi_n.$$

Triangular inequality will allow us to upper bound $\left(\E\left[\| \nb{n}^{noise}\|^2_{\Ld}\right]\right)^{1/2}$~:
\begin{equation}
 \left(\E\left[\| \nb{n}^{noise}\|^2_{\Ld}\right]\right)^{1/2}\le   \left(\E\left[\| \nb{n}^{noise,0}\|^2_{\Ld}\right]\right)^{1/2}+\left(\E\left[\| \nb{n}^{noise}- \nb{n}^{noise,0}\|^2_{\Ld}\right]\right)^{1/2}
\end{equation}

So that we're interested in the sequence $\left(\n{n}^{noise}-\n{n}^{noise,0}\right)_n$~: we have 
\begin{eqnarray}
\n{0}^{noise}-\n{0}^{noise,0}&=&0 , \nonumber\\
\n{n}^{noise}-\n{n}^{noise,0}&=&(I-\gamma \x)(\n{n-1}^{noise}-\n{n-1}^{0})  + \gamma (\Sigma- \x) \n{n-1}^0 \nonumber\\
&=&(I-\gamma \x)(\n{n-1}^{noise}-\n{n-1}^{0})  + \gamma \Xi_n^1 \label{sto.noise1}.
\end{eqnarray}
which is the same type of Equation as \eqref{sto.noise}. We have denoted $ \Xi_n^1 = (\Sigma- \x) \n{n-1}^0$.

Thus we may consider the following sequence, satisfying the ``semi-stochastic'' version of recursion \eqref{sto.noise1}, changing  $\x$ into its expectation $\Sigma$~: we define  $\left(\n{n}^{noise, 1}\right)_n$ so that:
\begin{equation}
\n{0}^{noise, 1} =0 \mbox{ and } \n{n}^{noise, 1}=(I-\gamma \Sigma)\n{n-1}^{noise, 1}   +\gamma \Xi_n^1.
\end{equation}

Thanks to the triangular inequality, we're interested in $\left(\n{n}^{noise}-\n{n}^{noise,0}-\n{n}^{noise,1}\right)_n$, which satisfies the \eqref{sto.noise}-type recursion~: 
\begin{eqnarray*}
\n{0}^{noise}-\n{0}^{noise,0}-\n{0}^{noise,1}&=&0,\\
\n{n}^{noise}-\n{n}^{noise,0}-\n{n}^{noise,1}&=&(I-\gamma \x)(\n{n-1}^{noise}-\n{n-1}^{noise,0}-\n{n}^{noise,1})  \\
& & \hspace{11.5em}+ \gamma (\Sigma- \x) \n{n-1}^{noise,1}\\
&=&(I-\gamma \x)(\n{n-1}^{noise}-\n{n-1}^{noise,0}-\n{n}^{noise,1})  + \gamma \Xi^{(2)}_n.
\end{eqnarray*}

With $\Xi^{(2)}_n:=(\Sigma- \x) \n{n-1}^{noise,1}$.

\vspace{0.5em}
And so on...  For any $r\geq 0$  we define a sequence $(\n{n}^{noise,r})_n$ by~:
$$ \n{0}^{noise,r} =0 \mbox{ and } \n{n}^{noise,r}=(I-\gamma \Sigma)\n{n-1}^{noise,r}   +\gamma \Xi_n^r,$$ $$ \quad \mbox{ with } \Xi_n^r = (\Sigma- \x) \n{n-1}^{noise,r-1}.$$

\vspace*{0.5em}

We have, for any $r,n \in \N^2 $~:
\begin{eqnarray}
\n{0}^{noise}-\sum_{i=0}^r \n{0}^{noise,i}&=&0,\nonumber \\
\n{n}^{noise}-\sum_{i=0}^r \n{n}^{noise,i}&=&(I-\gamma \x)\left(\n{n-1}^{noise}-\sum_{i=0}^r \n{n-1}^{noise,i}\right) \nonumber\\
& & \hspace{10em}+ \gamma (\Sigma- \x) \n{n-1}^{noise,r}.\nonumber\\
&=&(I-\gamma \x)\left(\n{n-1}^{noise}-\sum_{i=0}^r \n{n-1}^{noise,i}\right)  + \gamma \Xi_n^{(r+1)}. \label{sro.noise.r}
\end{eqnarray}

So that $\left( \n{n}^{noise,r+1}\right)$ follows the ``semi-stochastic'' version of \eqref{sro.noise.r}...

\paragraph*{Minkowski's inequality}
 Considering this decomposition, we have, for any $r$, using triangular inequality~:
\begin{equation}
 \left(\E\left[\| \nb{n}^{noise}\|^2_{\Ld}\right]\right)^{1/2}\le   \sum_{i=0}^r \left( \E\left[\| \nb{n}^{noise,i}\|^2_{\Ld}\right]\right)^{1/2}+\left(\E\left[\bigg|\bigg| \nb{n}^{noise}- \sum_{i=0}^r \nb{n}^{noise,i}\bigg|\bigg|^2_{\Ld}\right]\right)^{1/2} \label{minkows}
\end{equation}

\paragraph*{Moment Bounds} 
For any $i\geq 0$, we find that we may apply Lemma~\ref{lem.ssto.rec} to the sequence $(\n{n}^{noise,i})$.  Indeed~:
\begin{enumerate}
\item For any  $r\geq 0$, $(\n{n}^{noise,r})$  is defined by~:
$$
	\n{0}^{noise,r} =0 \mbox{ and } \n{n}^{noise,r}=(I-\gamma \Sigma)\n{n-1}^{noise,r}   +\gamma \Xi_n^r, $$ $$\quad \mbox{ with }  \Xi_n^r =
	\left\{ 
	\begin{array}{ll}
	(\Sigma- \x)\n{n-1}^{r-1} \quad \text{if } r \geq 1 . \\
	 \Xi_n \qquad \text{if } r=0 .
	\end{array} 
	\right.$$

\item for any $r\geq 0$, for all $n\geq 0$,  $ \ \Xi_n^r$ is $\mathcal{F}_n:=\sigma\left((x_i, z_i)_{1\le i\le n}\right) $ measurable. (for $r=0$ we use  the definition of $ \Xi_n $ (\textbf{H4}), and by induction,  for any $r\geq 0$ if we have $ \forall n \in \N, \ \ \Xi_n^r$ is $\mathcal{F}_n $ measurable, then for any $n \in \N$, by induction on $n$, $\n{n}^{noise,r}$ is $\mathcal{F}_n$ measurable, thus for any $n\in \N$,  $\Xi_n^{r+1}$ is $\mathcal{F}_n$ measurable.)
\item for any $ r,n\geq 0 $, $\  \E \left[\Xi_n | \mathcal{F}_{n-1}\right]=0 $~: as shown above, $\n{n-1}^{r-1}$ is $\mathcal{F}_{n-1}$  measurable  so $ \E \left[\ \Xi_n\  | \mathcal{F}_{n-1}\right] \ = \ \E \left[\Sigma- \x | \mathcal{F}_{n-1}\right] \n{n-1}^{noise,r-1}= \E \left[\Sigma- \x \right] \n{n-1}^{noise,r-1} =0 $ (as $x_n$ is independent of $\mathcal{F}_{n-1}$ by \textbf{A5} and $\E \left[\Sigma- \x \right]=\E \left[\Sigma- \x \right]$ by \textbf{H4}~).
\item  $ \E\left[\|\Xi_n^r\|^2\right] $ is finite (once again, by \textbf{A2} if $r=0$ and by a double recursion to get the result for any $r,n \geq 0$).
\item We have to find a bound on $ \E \left[\Xi_n^r\otimes\Xi_n^r\right]$. To do that, we are going, once again to use induction on $r$. 
\end{enumerate}

\begin{Lem} \label{moment.bound}
For any $r\geq 0$ we have \begin{eqnarray*}
\E \left[\Xi_n^r \otimes \Xi_n^r \right]&\lec& \gamma^r R^{2r} \sigma^2 \Sigma\\
\E \left[\n{n}^{noise,r }\otimes \n{n}^{noise,r} \right] &\lec& \gamma^{r+1} R^{2r} \sigma^2 I.
\end{eqnarray*} 

\end{Lem}

\begin{proof}[\textbf{Lemma~\ref{moment.bound}}]We make an induction on $ n $.
\ \\

\underline{Initialisation}~: for $r=0$ we have by \textbf{A6} that $\E \left[\Xi_n^0 \otimes \Xi_n^0 \right]\lec \sigma^2 \Sigma
$. Moreover  \begin{eqnarray*}
\E(\n{n}^0\otimes \n{n}^0) &=&  \gamma^2 \sum_{k=1}^{n-1} (I-\gamma \Sigma)^{n-k} \E \left[\Xi_n^0 \otimes \Xi_n^0 \right] (I-\gamma \Sigma)^{n-k}\\
 &\lec& \gamma^2 \sigma^2\sum_{k=1}^{n-1} (I-\gamma \Sigma)^{2(n-k)}\Sigma . 
\end{eqnarray*}  We get 
\begin{equation*}
\forall n \geq 0, \quad \E \left[\n{n}^{0}\otimes \n{n}^{0} \right] \lec \gamma^2 \sigma^2 \sum_{k=1}^{n-1} (I-\gamma \Sigma) ^{2n-2-k} \Sigma \lec \gamma \sigma^2 I.\\
\end{equation*}

\underline{Recursion}~: If we assume that for any $ n \geq 0,  \E \left[\Xi_n^r \otimes \Xi_n^r \right]\lec \gamma^r R^{2r} \sigma^2 \Sigma
$ and $\E \left[\n{n}^{r }\otimes \n{n}^{r} \right] \lec \gamma^{r+1} R^{2r} \sigma^2 I $ then for any $ n \geq 0 $~:
\begin{eqnarray*}
\E \left[\Xi_n^{r+1} \otimes \Xi_n^{r+1} \right]&\lec& \E \left[(\Sigma-\x) \n{n-1}^{r} \otimes \n{n-1}^{r}  (\Sigma-\x)\right]\\
&=& \E \left[(\Sigma-\x) \E\left[\n{n-1}^{r}\otimes \n{n-1}^{r}\right]  (\Sigma-\x)\right] \\
& & \hspace{15em} \quad (\text{as } \n{n-1} \in \mathcal{F}_{n-1})\\
&\lec & \gamma^{r+1} R^{2r} \sigma^2 \E \left[(\Sigma-\x)^2\right] \\
&\lec & \gamma^{r+1} R^{2r+2} \sigma^2  \Sigma. \\
\end{eqnarray*}
Once again we have $ (\n{n}^{r+1 }) = \gamma^2 \sum_{k=1}^{n-1} (I-\gamma \Sigma) ^{n-1-k} \Xi_n^{r+1}$, for any $n$: 
\begin{eqnarray*}
\E \left[\n{n}^{r+1}\otimes \n{n}^{r+1} \right] &\lec& \gamma^2 \E \left[\sum_{k=1}^n (I-\gamma \Sigma) ^{n-1-k} \Xi_n^{r+1} \otimes \Xi_n^{r+1} (I-\gamma \Sigma) ^{n-1-k}\right]\\
&\lec& \gamma^{r+3} R^{2r+2} \sigma^2 \sum_{k=1}^n (I-\gamma \Sigma) ^{2n-2-2k} \Sigma  \\
&\lec& \gamma^{r+2} R^{2r+2} \sigma^2 I .
\end{eqnarray*}
\end{proof}

With the bound on $\E \left[\Xi_n^r \otimes \Xi_n^r \right]$ and as we have said, with Lemma \ref{lem.ssto.rec}:
 \begin{eqnarray}
 \E\left[\| \nb{n}^{noise,i}\|^2_{\Ld}\right]=\E\left[\langle \nb{n}^{i}, \Sigma \nb{n}^{i} \rangle \right] &\le& \var(n, \gamma, \sigma^2\gamma^i R^{2i} , s, \alpha ) \nonumber \\
 &\le& \stackrel{\parallel}{\gamma^i R^{2i}  \var(n, \gamma, \sigma^2 , s, \alpha )} .\label{etai}.\end{eqnarray}

Moreover, using  the Lemma on stochastic recursions  (Lemma~\ref{lem.stoch.rec}) for $ (\nb{n}^{noise}-\sum_{i=0}^r \nb{n}^{i})_n $ (all conditions are satisfied) we have~:
\begin{eqnarray}
(1-\gamma R^2) \ \E\left[\Big\langle \nb{n}^{noise}-\sum_{i=0}^r \nb{n}^{i}, \Sigma \left(\nb{n}^{noise}-\sum_{i=0}^r \nb{n}^{i}  \right)\Big\rangle \right]&\le& \frac{\gamma}{n} \sum_{i=1}^n \E \|\Xi_k^{r+1}\|^2  \nonumber\\
&\le& \gamma \tr \left(\E \left[ \Xi_k^{r+1} \otimes \Xi_k^{r+1}\right]\right)\nonumber\\
&\le&   \gamma^{r+2} R^{2r+2} \sigma^2  \tr(\Sigma) \nonumber \\
\text{that is } \ \ \E\left[\bigg|\bigg| \nb{n}^{noise}- \sum_{i=0}^r \nb{n}^{noise,i}\bigg|\bigg|^2_{\Ld}\right] &\le&   \gamma^{r+2} R^{2r+2} \sigma^2  \tr(\Sigma)\label{diff}.
\end{eqnarray}

\paragraph*{Conclusion} 
Thus using \eqref{minkows}, \eqref{etai} and \eqref{diff}~: 
\begin{eqnarray}
\left(\E\left[\langle \nb{n}^{noise}, \Sigma \nb{n}^{noise} \rangle \right]\right) ^{1/2} &\le& \left(\frac{1}{1-\gamma R^2} \gamma^{r+2} \sigma^2 R^{2r+2} \tr(\Sigma) \right)^{1/2} \nonumber \\
& & \hspace{2em}+ \var(n, \gamma, \sigma^2, s, \alpha )^{1/2} \sum_{i=0}^r   \left(\gamma R^2\right)^{i/2} . \label{previousvar}
\end{eqnarray}

And using the fact that $\gamma R < 1$, when $r \rightarrow \infty $ we get: 
\begin{equation}
\left(\E\left[\langle \nb{n}^{noise}, \Sigma \nb{n}^{noise} \rangle \right]\right) ^{1/2} \le \var(n, \gamma,  \sigma^2,s,\alpha )^{1/2} \frac{1}{1-\sqrt{\gamma R^2}}. \label{noise}
\end{equation}
\vspace*{0.5em}

Which is the main result of this part.

\subsubsection{Initial conditions}

We are now interested in getting such a bound for $\E\left[\langle \nb{n}^{init}, \Sigma \nb{n}^{init} \rangle \right]$. As this part stands for the initial conditions effect we may keep in mind that we would like to get an upper bound comparable to what we found for the Bias term in the proof of Proposition 1.

We remind that~: $$ \n{0}^{init} = {g_\H} \mbox{ and } \n{n}^{init}=(I-\gamma \x)\n{n-1}^{init} .$$ 
  
and define  $(\eta_n^{0})_{n\in \N}$ so that~: $$\eta_0^{0}=    {g_\H}, \quad \eta_n^{0}= (I-\gamma \Sigma) \eta_{n-1}^{0} .$$
 
\paragraph*{Minkowski's again}

As above   \begin{equation}
\left(\E\left[\langle \nb{n}^{init}, \Sigma \nb{n}^{init} \rangle \right]\right) ^{1/2} \le \left(\E\left[\langle \nb{n}^{init}- \nb{n}^{0}, \Sigma \left(\nb{n}^{init}- \nb{n}^{0}  \right)\rangle \right]\right) ^{1/2} + \left(\E\left[\langle \nb{n}^{0}, \Sigma \nb{n}^{0} \rangle \right] \right)^{1/2}.
\label{minkow_init}
\end{equation}
 
\textbf{First} for $\overline{\eta}^0_n$ we have a semi-stochastic recursion, with $\Xi_n \equiv 0$ so that we have $$\E \langle \overline{\eta}^0_n , \Sigma\overline{\eta}^0_n\rangle \le \bias(n, \gamma, r). $$

\textbf{Then }, for the residual term we use Lemma~\ref{lem.stoch.rec}. Using that~: $$ \eta^0_n - \n{n}^{init}= (I-\gamma K_{x_n} \otimes K_{x_n}) (\eta^0_n - \n{n}^{init}) + \gamma (K_{x_n} \otimes K_{x_n}-\Sigma) \eta_{n-1}^{0},$$
 we may apply \textbf{Lemma~\ref{lem.stoch.rec}}  to the recursion above with $\alpha_n=\eta^0_n - \eta_n^{init}$ and $\Xi_n=(K_{x_n} \otimes K_{x_n}-\Sigma) \eta_{n-1}^{0}$.
 That is (as $\alpha_0=0$):
 
\begin{equation}
  \E \langle \nb{n}^0- \nb{n}^{noise}, \Sigma(\nb{n}^0 - \nb{n}^{noise}) \rangle\le \frac{1}{1-\gamma R^2}
\frac{\gamma}{n} \E\left[\sum_{k=1}^n \|\Xi_k\|^2\right] \label{initresiduel}.
\end{equation}
 
 Now
  \begin{eqnarray*}
 \E \|\Xi_k\|^2&=& \E\left[\big \langle \n{0}, (I-\gamma \Sigma)^k (\Sigma-x_k\otimes x_k)^2(I-\gamma \Sigma)^k \n{0}\big \rangle\right] \\
 &\le& \big \langle \n{0}, (I-\gamma \Sigma)^k R^2 \Sigma(I-\gamma \Sigma)^k\n{0} \big \rangle\\
  &\le& R^2 \big \langle \n{0}, (I-\gamma \Sigma)^{2k} \Sigma\n{0} \big \rangle .
\end{eqnarray*} 

Thus~: 
\begin{eqnarray*}
\frac{\gamma}{n} \E\left[\sum_{k=1}^n \|\Xi_k\|^2 \right]&\le& \frac{\gamma R^2}{n} \big \langle \n{0},  \sum_{k=1}^n (I-\gamma \Sigma)^{2k} \Sigma\n{0} \big \rangle\\
&\le& \frac{\gamma R^2}{n} \bigg|\bigg| \left(\sum_{k=1}^n (I-\gamma \Sigma)^{2k} \Sigma^{2r}\right)^{1/2}  \Sigma^{1/2-r}\n{0}\bigg|\bigg|^2\\
 &\le &  \frac{\gamma R^2}{n} \gamma^{-2r} \bigg|\bigg|\bigg| \sum_{k=1}^n (I-\gamma \Sigma)^{2k} (\gamma\Sigma)^{2r}\bigg|\bigg|\bigg| \ \|\Sigma^{-r} \n{0}\|_{\Ld}^2.
\end{eqnarray*}

 $\||A^{1/2}\||^2=\||A\||$. Moreover, as $\Sigma$ is self adjoint, we have:
\begin{eqnarray*}
 \bigg|\bigg|\bigg| \sum_{k=1}^n (I-\gamma \Sigma)^{2k} (\gamma\Sigma)^{2r}\bigg|\bigg|\bigg| &\le& \sup_{0\le x\le 1} \sum_{k=1}^n (1-x)^{2k} (x)^{2r} \\
 &\le& \sup_{0\le x\le 1} \frac{1- (1-x)^{2n}}{1-(1-x)^2}  (x)^{2r} \\
  &\le& \sup_{0\le x\le 1} \frac{1- (x)^{2n}}{1-x^2}  (1-x)^{2r} \\
  &\le& \sup_{0\le x\le 1} \frac{1- (x)^{2n}}{1+x}  (1-x)^{2r-1} \\
 &\le& \sup_{0\le x\le 1} (1- (1-x)^{2n} ) (x)^{2r-1} \\
 &\le & n^{1-2r} 
 \end{eqnarray*} 

Where we have used inequality \eqref{maj_ineq}, if $r\le 1/2$. However, this result does not stand anymore if $r\geq 1/2$. To deal with this particular case, we use the fact that, 
\begin{eqnarray*}
 \E\langle\bar \eta_n - \eta_\ast, 
 \Sigma ( \bar \eta_n  - \eta_\ast) \rangle  &\le&  (1 + (R^{2\alpha} \gamma^{1+\alpha} n s^2)^{\frac{2r-1}{\alpha}} )  \frac{ \| \Sigma^{-r} \n{0} \|^{2}_{L2}}{(\gamma n)^{2r}} .
 \end{eqnarray*}
 This result's proof is postponed to Lemma~\ref{lem:fixproof}.

 So that we would get, replacing our result in \eqref{initresiduel}~:
 \begin{equation}
  \E \langle \nb{n}^0- \nb{n}^{noise}, \Sigma(\nb{n}^0 - \nb{n}^{noise}) \rangle\le \frac{1}{1-\gamma R^2}
\frac{\gamma R^2 }{(\gamma n)^{2r}} \|\Sigma^{-r} \n{0}\|_{\Ld}^2.
\end{equation}

 \paragraph*{Conclusion}
Summing both bounds we get from \eqref{minkow_init}~:

\begin{equation}
\left(\E\left[\langle \nb{n}^{init}, \Sigma \nb{n}^{init} \rangle \right]\right) ^{1/2} \le \left(\frac{1}{1-\gamma R^2}
\frac{\gamma R^2}{(\gamma n)^{2r}} \|\Sigma^{-r} \n{0}\|_{\Ld}^2 \right)^{1/2} + \left(Bias(n, \gamma, g_\H, \alpha)\right)^{1/2}. \label{init}
\end{equation}

\subsubsection{Conclusion}

These two parts allow us to show Theorem \ref{prop.dinf.rand}~: using \eqref{init} and  \eqref{noise} in \eqref{init+sto}, and Lemmas \ref{bias_gam_const} and \ref{var_gam_const} we have the final result.

 Assuming \textbf{A1-6}~:  

\begin{enumerate}
\item If $r <1$ 
 \begin{eqnarray*}
 \left(\ \E \left[ \epsilon\left(\t{n}\right)-\epsilon({g_\H}) \right]\right) ^{1/2} &\le&  \frac{1}{1-\sqrt{\gamma R^2}} \left(C(\alpha) \ s^{\frac{2}{\alpha}}\   \sigma^2  \frac{\gamma^{\frac{1}{\alpha}}}  {n^{1-\frac{1}{\alpha}}}  +\frac{\sigma^2}{n}\right)^{1/2}  \\
 & & \hspace{3em}+ \left( \|\Sigma^{-r} {g_\H}\|_{\Ld}^2  \left( \frac{1}{(n\gamma)^{2r}} \right) \right)^{1/2} \\&&\hspace{3em}+ \left(\frac{1}{1-\gamma R^2}
\frac{\gamma R^2}{(\gamma n)^{2r}} \|\Sigma^{-r} \n{0}\|_{\Ld}^2 \right)^{1/2}.
 \end{eqnarray*}

\item If $r>1$ 
 \begin{eqnarray*}
 \left(\ \E \left[ \epsilon\left(\t{n}\right)-\epsilon({g_\H}) \right]\right) ^{1/2} &\le& \frac{1}{1-\sqrt{\gamma R^2}} \left(C(\alpha) \  s^{\frac{2}{\alpha}}\  \sigma^2  \frac{\gamma^{\frac{1}{\alpha}}}  {n^{1-\frac{1}{\alpha}}} + \frac{\sigma^2}{n}\right)^{1/2}\\ &&\hspace{3em} +\left( \|\Sigma^{-r} {g_\H}\|_{\Ld}^2  \left( \frac{1}{n^2 \gamma^{2r}} \right)\right)^{1/2} \\ & &\hspace{3em}+ \left(\frac{1}{1-\gamma R^2}
\frac{\gamma R^2}{(\gamma n)^{2r}} \|\Sigma^{-r} \n{0}\|_{\Ld}^2 \right)^{1/2}. 
 \end{eqnarray*}

\end{enumerate}

Regrouping terms, we get : 
\begin{Th}[Complete bound, $\gamma$ constant, finite horizon]\label{prop.dinf.rand_nonsimpl} Assume \textbf{(A1-6)} and $\gamma_i=\gamma = \Gamma(n)$, for $1\le i\le n$. 
 We have,  with $C(\alpha)=\frac{2 \alpha^2 }{(\alpha+1)(2\alpha-1)}$:
\begin{eqnarray*}
 \(\E \| \bar{g}_n -  g_\H \|_{\Ld}^{2}\)^{1/2} &  \le &   \frac{\sigma / \sqrt{n} }{1-\sqrt{\gamma R^2}} \left( 1 + C(\alpha)   s^{\frac{2}{\alpha}}     (\gamma n )^{\frac{1}{\alpha}}  \right)^{\frac{1}{2}}  \\
 & & \hspace*{2cm} +
\frac{ \| L_K^{-r} g_\H\| _{\Ld} }{\gamma^rn^{  \min\{r,1\}} }
\bigg( 1 + \frac{\sqrt{\gamma R^2 }}{\sqrt{ 1 - \gamma R^2}} \bigg).
\end{eqnarray*}
  \end{Th}

Then bounding $C(\alpha)$ by 1 and simplifying under assumption $\gamma R^{2}\le 1/4$, we exactly get Theorem~\ref{prop.dinf.rand} in the main text. In order to derive corollaries, one just has to chose $\gamma=\Gamma(n)$ in order to balance the main terms.

\subsection{Complete proof, Theorem \ref{prop.dinf.rand.onl} (on-line setting)}
\label{subsec:completeproofONL}
The sketch of the proof is exactly the same. We just have to check that changing a constant step into  a decreasing sequence of step-size does not change to much. However as most calculations make appear some weird constants, we will only look for asymptotics. The sketch of the decomposition is given in Table~\ref{tab:online}.

\vspace{1em}
\begin{table}
\makebox[\textwidth][c]{
\begin{tabular}{|ccccc|}
   \hline
  &    \multicolumn{3}{c}{Complete recursion $\eta_n$ } &      \\ 
   
     &    $\swarrow$ &   & $\searrow$ &      \\ 
   
    \multicolumn{2}{|c}{noise term $\eta^{noise}_n$} &  | &     \multicolumn{2}{c|}{ bias term $\eta^{init}_n$}  \\ 
   
    \multicolumn{2}{|c}{$\downarrow$  } &  | &    \multicolumn{2}{c|}{ $\downarrow$ }   \\ 
   
    \multicolumn{2}{|c}{multiple recursion } &  | &   \multicolumn{2}{c|}{ semi stochastic variant}   \\ 
   
   $\swarrow$ &    $\searrow$ & |  & $\swarrow$ &   $\searrow$ \\ 
   
   main terms $\eta^r_n$, $r\geq 1$ &    residual term $\eta^{noise}_n - \sum \eta^r_n$ &  | & main term $\eta^0_n $ &    residual term $\eta^{init}_n -  \eta^0_n$ \\ 
   
  satisfying semi-sto recursions &    satisf. stochastic recursion & |  & satisf. semi-sto recursion &    satisf. stochastic recursion \\ 
   
  Lemma \ref{lem.ssto.rec} &    Lemma \ref{lem.stoch.rec.onl}&  | & $\downarrow$  &   Lemma \ref{lem.stoch.rec.onl} \\ 
   
    $\downarrow$ &    $\swarrow \qquad \searrow$ &  | & $\downarrow$  &    $\swarrow \qquad \searrow$ \\ 
    
   $\le C$ Variance term &    $\rightarrow_{r\rightarrow \infty} 0\qquad + \qquad\rightarrow_{r\rightarrow \infty} 0$ & |  & $\le $ Bias term &    Resid. term 1 + Resid term 2 \\ 
  
  &&&& \\
    \multicolumn{2}{|c}{ \ $\qquad\qquad$ Lemma \ref{var_gam_var} $ \searrow$} &   & \multicolumn{2}{c|}{$\swarrow$ Lemma \ref{bias_gam_var}}     \\ 
 
     &    \multicolumn{3}{c}{Theorem \ref{prop.dinf.rand.onl}} &      \\ 
   \hline
   
\end{tabular} 
}\vspace{0.5em}
\caption{Sketch of the proof, on-line setting.} \label{tab:online}
\end{table}

\subsubsection{A Lemma on stochastic recursions - on-line}

We want to derive a Lemma comparable to Lemma \ref{lem.stoch.rec} in the online setting. That is considering a sequence $(\gamma_n)_n $ and the recursion  $\alpha_n=(I-\gamma_n K_{x_n} \otimes K_{x_n}) \alpha_{n-1} + \gamma_n \Xi_n$ we would like to have a bound on $\E\big\langle \overline{\alpha}_{n-1}, \Sigma\overline{\alpha}_{n-1} \big\rangle $.

 \begin{Lem}\label{lem.stoch.rec.onl}
Assume $(x_n, \Xi_n) \in \H \times \H$ are $\mathcal{F}_n$ measurable for a sequence of increasing $\sigma$-fields $ \left( \mathcal{F}_n\right) $. Assume that $ \E\left[\Xi_n|\mathcal{F}_{n-1}\right]=0$,  $ \E\left[\|\Xi_n\|^2|\mathcal{F}_{n-1}\right]$ is finite and $ \E\left[\|K_{x_n}\|^2 \x|\mathcal{F}_{n-1}\right]\lec R^2 \Sigma$, with $ \E \left[K_{x_n} \otimes K_{x_n} | \mathcal{F}_{n-1}\right]=\Sigma$ for all $ n\geq 1 $ , for some $R>0$ and invertible operator $\Sigma$. Consider the recursion $\alpha_n=(I-\gamma_n K_{x_n} \otimes K_{x_n}) \alpha_{n-1} + \gamma_n \Xi_n$, with $(\gamma_n)_n$ a sequence such that for any $n$, $\gamma_n R^2\le 1$. Then if $\alpha_0=0$, we have So that if $\alpha_0=0$~:
\begin{equation}
\E\left[\big\langle \overline{\alpha}_{n-1}, \Sigma\overline{\alpha}_{n-1} \big\rangle\right] \le \frac{1}{2   n (1- \gamma_0 R^2)} \( \sum_{i=1}^{n-1} \|\alpha_i\|^2 {\(-\frac{1}{\gamma_i} +\frac{1}{\gamma_{i+1}} \)} + \sum_{k=1}^n \gamma_k \E \|\Xi_k\|^2\).
\end{equation}
 \end{Lem} 
 
\begin{proof}
\begin{equation}
2 \gamma_n (1- \gamma_n R^2) \E \langle \Sigma \alpha_{n-1}, \alpha_{n-1} \rangle \le \E\( \|\alpha_{n-1}\|^2 - \|\alpha_{n}\|^2 + \gamma_n^2 \|\Xi_n\|^2 \)
\end{equation}
So that, if we assume that $(\gamma_n) $  is non increasing: 
\begin{equation}
 \E \langle \Sigma \alpha_{n-1}, \alpha_{n-1} \rangle \le \frac{1}{2 \gamma_n (1- \gamma_0 R^2)} \E\( \|\alpha_{n-1}\|^2 - \|\alpha_{n}\|^2 + \gamma_n^2 \|\Xi_n\|^2 \)
\end{equation}
Using convexity~:
\begin{eqnarray*}
\E\left[\big\langle \overline{\alpha}_{n-1}, \Sigma\overline{\alpha}_{n-1} \big\rangle\right] &\le& \frac{1}{2   n (1- \gamma_0 R^2)} \Bigg( \frac{\|\alpha_0\|^2}{\gamma_1} + \sum_{i=1}^{n-1} \|\alpha_i\|^2 \underbrace{\(-\frac{1}{\gamma_i} 
+\frac{1}{\gamma_{i+1}} \)}_{\geq 0}  \\
& & \hspace{6em}- \frac{\|\alpha_{n}\|^2}{\gamma_n}+ \sum_{k=1}^n \gamma_k \E \|\Xi_k\|^2\Bigg).
\end{eqnarray*}

So that if $\alpha_0=0$~:
\begin{equation}
\E\left[\big\langle \overline{\alpha}_{n-1}, \Sigma\overline{\alpha}_{n-1} \big\rangle\right] \le \frac{1}{2   n (1- \gamma_0 R^2)} \( \sum_{i=1}^{n-1} \|\alpha_i\|^2 {\(-\frac{1}{\gamma_i} +\frac{1}{\gamma_{i+1}} \)} + \sum_{k=1}^n \gamma_k \E \|\Xi_k\|^2\).
\end{equation}

Note that it may  be interesting to consider the weighted average $\tilde{\alpha}_n= \frac{\sum \gamma_i \alpha_i} {\sum \gamma_i}$, which would satisfy be convexity 
\begin{equation}
\E\left[\big\langle \tilde{\alpha}_{n-1}, \Sigma\tilde{\alpha}_{n-1} \big\rangle\right] \le \frac{1}{2   (\sum \gamma_i) (1- \gamma_0 R^2)} \( \frac{\|\alpha_0\|^2}{\gamma_1} - \frac{\|\alpha_{n}\|^2}{\gamma_n}  +  \sum_{k=1}^n \gamma^2_k \E \|\Xi_k\|^2\).
\end{equation}
\end{proof}

\subsubsection{Noise process}
We remind that $(\n{n}^{noise})_n $ is defined by~: 
\begin{equation} 
 \n{0}^{noise} =0 \mbox{ and } \n{n}^{noise}=(I-\gamma \x)\n{n-1}^{noise}   +\gamma \Xi_n.
 \end{equation}
 
As before, for any $r\geq 0$  we define a sequence $(\n{n}^{noise,r})_n$ by~:
$$ \n{0}^{noise,r} =0 \mbox{ and } \n{n}^{noise,r}=(I-\gamma \Sigma)\n{n-1}^{noise,r}   +\gamma \Xi_n^r,$$ $$ \quad \mbox{ with } \Xi_n^r = (\Sigma- \x) \n{n-1}^{noise,r-1}.$$

And we want to use the following upper bound
\begin{equation}
\left(\E\left[\| \nb{n}^{noise}\|^2_{\Ld}\right]\right)^{1/2}\le   \sum_{i=0}^r \left( \E\left[\| \nb{n}^{noise,i}\|^2_{\Ld}\right]\right)^{1/2}+\left(\E\left[\bigg|\bigg| \nb{n}^{noise}- \sum_{i=0}^r \nb{n}^{noise,i}\bigg|\bigg|^2_{\Ld}\right]\right)^{1/2}. \label{minkows.onl}
\end{equation}

 So that we had to upper bound the noise~:
 
\begin{Lem} \label{moment.bound.onl}
For any $r\geq 0$ we have $\ \E \left[\Xi_n^r \otimes \Xi_n^r \right]\lec R^{2r} \gamma_0^r \sigma^2 \Sigma
$ and $\E \left[\n{n}^{noise,r }\otimes \n{n}^{noise,r} \right] \lec \gamma_0^{r+1} R^{2r} \sigma^2 I $.
\end{Lem}

\begin{proof}[\textbf{Lemma~\ref{moment.bound.onl}}]We make an induction on $ n $.
\ \\
We note that~: \begin{eqnarray}
 \sum_{k=1}^{n} D(n,k+1,  (\gamma_k)_k) \gamma_k^2 \Sigma D(n,k+1,  (\gamma_k)_k) &\le& \gamma_0 \sum_{k=1}^{n} D(n,k+1,  (\gamma_k)_k) \gamma_k \Sigma \nonumber\\
  &\le& \gamma_0 \sum_{k=1}^{n} D(n,k+1,  (\gamma_k)_k) - D(n,k,  (\gamma_k)_k) \nonumber \\
  &\le& \gamma_0 (I- D(n,1,  (\gamma_k)_k))  \nonumber \\
  &\le & \gamma_0 I \label{remark_trick}
\end{eqnarray}

Where we have used that~: $D(n,k+1,  (\gamma_k)_k) - D(n,k,  (\gamma_k)_k) = D(n,k+1,  (\gamma_k)_k) \gamma_k \Sigma$.

\underline{Initialisation}~: for $r=0$ we have by \textbf{A6} that $\E \left[\Xi_n^0 \otimes \Xi_n^0 \right]\lec \sigma^2 \Sigma
$. Moreover   $\n{n}^0= \sum_{k=1}^n D( n,k+1, (\gamma_k)_k) \gamma_k \Xi^0_k$.

\begin{eqnarray*}
\E(\n{n}^0\otimes \n{n}^0) &=&   \sum_{k=1}^{n} D( n,k+1, (\gamma_k)_k) \gamma_k^2 \E \left[\Xi_k^0 \otimes \Xi_k^0 \right] D( k+1,n, (\gamma_k)_k)\\
 &\lec& \sigma^2 \sum_{k=1}^{n} D( n,k+1, (\gamma_k)_k) \gamma_k^2  \Sigma D( k+1,n, (\gamma_k)_k)\\
 &\lec& \sigma^2 \gamma_0 I, \quad \text{ by \eqref{remark_trick}}
\end{eqnarray*}  

\underline{Induction }: If we assume $\forall n \geq 0, \quad \E \left[\Xi_n^r \otimes \Xi_n^r \right]\lec \gamma_0^r R^{2r} \sigma^2 \Sigma
$ and $\E \left[\n{n}^{r }\otimes \n{n}^{r} \right] \lec \gamma_0^{r+1} R^{2r} \sigma^2 I $ then: $ \forall n \geq 0, $
\begin{eqnarray*}
\E \left[\Xi_n^{r+1} \otimes \Xi_n^{r+1} \right]&\lec& \E \left[(\Sigma-\x) \n{n-1}^{r} \otimes \n{n-1}^{r}  (\Sigma-\x)\right]\\
&=& \E \left[(\Sigma-\x) \E\left[\n{n-1}^{r}\otimes \n{n-1}^{r}\right]  (\Sigma-\x)\right] \\
&& \hspace{15em}
\quad (\text{as } \n{n-1} \in \mathcal{F}_{n-1})\\
&\lec & \gamma_0^{r+1} R^{2r} \sigma^2 \E \left[(\Sigma-\x)^2\right] \\
&\lec & \gamma_0^{r+1} R^{2r+2} \sigma^2  \Sigma. \\
\end{eqnarray*}
Once again we have $\n{n}^{r+1}= \sum_{k=1}^n D( n,k+1, (\gamma_k)_k) \gamma_k \Xi^{r+1}_k$, for any $n$: 
\begin{eqnarray*}
\E \left[\n{n}^{r+1}\otimes \n{n}^{r+1} \right] &\lec& \gamma^2 \E \left[\sum_{k=1}^n (I-\gamma \Sigma) ^{n-1-k} \Xi_n^{r+1} \otimes \Xi_n^{r+1} (I-\gamma \Sigma) ^{n-1-k}\right]\\ 
&\lec& \sigma^2 \gamma_0^{r+1} R^{2r} \sum_{k=1}^{n} D( n,k+1, (\gamma_k)_k) \gamma_k^2  \Sigma D( k+1,n, (\gamma_k)_k)\\
 &\lec&\sigma^2 \gamma_0^{r+2} R^{2r} I, \quad \text{ by \eqref{remark_trick}}
\end{eqnarray*}
\end{proof}

With the bound on $\E \left[\Xi_n^r \otimes \Xi_n^r \right]$ and as we have said, with Lemma \ref{lem.ssto.rec}:
 \begin{equation}
 \E\left[\| \nb{n}^{noise,i}\|^2_{\Ld}\right]=\E\left[\langle \nb{n}^{i}, \Sigma \nb{n}^{i} \rangle \right] \le  \var(n, \gamma, \alpha, \gamma^i_0 R^{2i} \sigma, s ) =\gamma_0^i R^{2i} \var(n, \gamma, \alpha, \sigma, s )  \label{etai.onl}.\end{equation}

Moreover, using  the Lemma on stochastic recursions  (Lemma~\ref{lem.stoch.rec.onl}) for $(\alpha^r_n)_n= (\n{n}^{noise}-\sum_{i=0}^r \n{n}^{i})_n $ (all conditions are satisfied) we have~:
\begin{eqnarray}
2(1-\gamma_0 R^2) \ \E\left[\Big\langle \overline{\alpha}^r_n, \Sigma\overline{\alpha}^r_n \Big\rangle \right]&\le& \frac{1}{n} \( \sum_{i=1}^{n-1} \E\|\alpha^r_i\|^2 {\(-\frac{1}{\gamma_i} +\frac{1}{\gamma_{i+1}} \)} + \sum_{k=1}^n \gamma_k \E \|\Xi^{r+1}_k\|^2\). \nonumber
\end{eqnarray}

We are going to show that both these terms goes to 0 when $r$ goes to infinity.  Indeed~:
 \begin{eqnarray}
\sum_{k=1}^n \gamma_k \E \|\Xi^{r+1}_k\|^2 &\le& \sum_{k=1}^n \gamma_k   \tr \left(\E \left[ \Xi_k^{r+1} \otimes \Xi_k^{r+1}\right]\right)\nonumber\\
&\le&  \sum_{k=1}^n \gamma_k   \gamma_0^{r+1} R^{2r+2} \sigma^2  \tr(\Sigma) \nonumber \\
&\le& n  \gamma_0^{r+2} R^{2r+2} \sigma^2  \tr(\Sigma) \nonumber 
\end{eqnarray}

Moreover, if we assume $\gamma_i=\frac{1}{i^\zeta}$~:
 \begin{equation*}
\frac{1}{n} \sum_{i=1}^{n-1} \E\|\alpha^r_i\|^2 {\(-\frac{1}{\gamma_i} +\frac{1}{\gamma_{i+1}} \)} \le 2\zeta \frac{1}{n} \sum_{i=1}^{n-1}  \frac{\gamma_i}{i } \E\|\alpha^r_i\|^2
\end{equation*}

And \begin{equation*}
\alpha_i^r= (I-\gamma_i \widetilde{K_{x_i} \otimes K_{x_i}})\alpha_{i-1}^r + \gamma_i \Xi_i
\end{equation*}
So that~:
\begin{eqnarray*}
\|\alpha_i^r\|&\le& \||(I-\gamma_i \widetilde{K_{x_i} \otimes K_{x_i}})\|| \ \|\alpha_{i-1}^r\| + \gamma_i \ \|\Xi_i\|\\
&\le& \ \|\alpha_{i-1}^r\| + \gamma_i \ \|\Xi_i\|\\
&\le& \sum_{k=1}^i  \gamma_k \ \|\Xi_k\|. \\
\text{thus~: } \|\alpha_i^r\|^2 &\le& \sum_{k=1}^i  \gamma_k \sum_{k=1}^i  \gamma_k \ \|\Xi_k\|^2  \\
\E \|\alpha_i^r\|^2 &\le& \sum_{k=1}^i  \gamma_k \sum_{k=1}^i  \gamma_k \ \E \|\Xi_k\|^2 \\
\ \E \|\alpha_i^r\|^2 &\le&  C_1 \ \ i \gamma_i  \ \ i  \gamma_0^{r+2} R^{2r+2} \sigma^2  \tr(\Sigma)\\
\frac{\gamma_i}{i }\ \E \|\alpha_i^r\|^2 &\le& C_2 \ \  i  \gamma^2_i  \ \ (\gamma_0 R^2)^ {r+2} \\
\frac{1}{n} \sum_{i=1}^{n-1} \E\|\alpha^r_i\|^2 {\(-\frac{1}{\gamma_i} +\frac{1}{\gamma_{i+1}} \)} & \le& C_3 \ \ n\gamma_n^2  \ \  (\gamma_0 R^2)^ {r+2}  .
\end{eqnarray*}

That is :
\begin{equation}
E\left[\bigg|\bigg| \nb{n}^{noise}- \sum_{i=0}^r \nb{n}^{noise,i}\bigg|\bigg|^2_{\Ld}\right]\le  (\gamma_0 R^2)^ {r+2}    \( \sigma^2  \tr(\Sigma)+   C_3 n\gamma_n^2\) . \label{diff.onl}
\end{equation}

With \eqref{minkows.onl}, \eqref{etai.onl},\eqref{diff.onl}, we get~:
\begin{eqnarray}
\left(\E\left[\| \nb{n}^{noise}\|^2_{\Ld}\right]\right)^{1/2}&\le&   \sum_{i=0}^r \left(\gamma_0^i R^{2i} \var(n, \gamma, \alpha, \sigma, s )\right)^{1/2}\nonumber\\
&&\hspace{5em}+\left( (\gamma_0 R^2)^ {r+2}    \( \sigma^2  \tr(\Sigma)+   C_3  \ n\gamma_n^2\) \right)^{1/2}.
\end{eqnarray}
So that, with $r\rightarrow \infty$~:
\begin{equation}
\left(\E\left[\| \nb{n}^{noise}\|^2_{\Ld}\right]\right)^{1/2}\le \(C \var(n, \gamma, \alpha, \sigma, s )\)^{1/2}  \label{compl.noise.onl}.
\end{equation}

\subsubsection{Initial conditions}
Exactly as before, we can separate the effect of initial conditions and of noise~: 
We are  interested in getting such a bound for $\E\left[\langle \nb{n}^{init}, \Sigma \nb{n}^{init} \rangle \right]$.
We remind that~: $$ \n{0}^{init} = {g_\H} \mbox{ and } \n{n}^{init}=(I-\gamma_n \x)\n{n-1}^{init} .$$ 
and define  $(\eta_n^{0})_{n\in \N}$ so that~: $$\eta_0^{0}=    {g_\H}, \quad \eta_n^{0}= (I-\gamma_n \Sigma) \eta_{n-1}^{0} .$$
 
\paragraph*{Minkowski's again~:}

As above   \begin{equation}
\left(\E\left[\langle \nb{n}^{init}, \Sigma \nb{n}^{init} \rangle \right]\right) ^{1/2} \le \left(\E\left[\langle \nb{n}^{init}- \nb{n}^{0}, \Sigma \left(\nb{n}^{init}- \nb{n}^{0}  \right)\rangle \right]\right) ^{1/2} + \left(\E\left[\langle \nb{n}^{0}, \Sigma \nb{n}^{0} \rangle \right] \right)^{1/2}.
\label{minkows.init.onl}
\end{equation}
 
\textbf{First} for $\overline{\eta}^0_n$ we have a semi-stochastic recursion, with $\Xi_n \equiv 0$ so that we have \begin{equation}
 \langle \overline{\eta}^0_n , \Sigma\overline{\eta}^0_n \rangle \le \Bias(n, (\gamma_n)_n, g_\H, r). \label{princ.init.onl} \end{equation}

\textbf{Then }, for the residual term we use \textbf{Lemma}~\ref{lem.stoch.rec.onl} for the recursion above with $\alpha_n=\eta^0_n - \eta_n^{init}$. Using that~: $$ \eta^0_n - \n{n}^{init}= (I-\gamma K_{x_n} \otimes K_{x_n}) (\eta^0_n - \n{n}^{init}) + \gamma_n (K_{x_n} \otimes K_{x_n}-\Sigma) \eta_{n-1}^{0},$$
 That is (as $\alpha_0=0$):
\begin{eqnarray}
  \E \langle \nb{n}^0- \nb{n}^{noise}, \Sigma(\nb{n}^0 - \nb{n}^{noise}) \rangle&\le& \frac{1}{2   n (1- \gamma_0 R^2)} \Bigg( \sum_{i=1}^{n-1} \E \|\alpha_i\|^2 {\(-\frac{1}{\gamma_i} +\frac{1}{\gamma_{i+1}} \)} \nonumber\\
 &&\hspace{10em} + \sum_{k=1}^n \gamma_k \E \|\Xi_k\|^2 \Bigg). \label{initresiduel.onl}
\end{eqnarray}
 
 Now 
  \begin{eqnarray*}
 \E \|\Xi_k\|^2&=& \E\left[\big \langle \n{0}, D(n,1, (\gamma_i)_i) (\Sigma-x_k\otimes x_k)^2 D(1,n, (\gamma_i)_i) \n{0}\big \rangle\right] \\
 &\le& R^2 \big \langle \n{0}, D(1,n, (\gamma_i)_i)^2 \Sigma \n{0} \big \rangle .
\end{eqnarray*} 

Thus~: 
\begin{eqnarray}
\E\left[\sum_{k=1}^n \gamma_k \|\Xi_k\|^2 \right]&\le&  R^2 \big \langle \n{0}, \sum_{k=1}^n \gamma_k D(1,n, (\gamma_i)_i)^2 \Sigma \n{0} \big \rangle \nonumber\\
&\le&  R^2 \bigg|\bigg| \left( \sum_{k=1}^n \gamma_k D(1,n, (\gamma_i)_i)^2  \Sigma^{2r}\right)^{1/2}  \Sigma^{1/2-r}\n{0}\bigg|\bigg|^2 \nonumber\\
 &\le &  R^2  \bigg|\bigg|\bigg| \sum_{k=1}^n  D(1,n, (\gamma_i)_i)^2 \gamma_k \Sigma^{2r}\bigg|\bigg|\bigg| \ \|\Sigma^{-r} \n{0}\|_{\Ld}^2 \label{1}.
\end{eqnarray}

Now~:
\begin{eqnarray}
 \bigg|\bigg|\bigg| \sum_{k=1}^n  D(1,n, (\gamma_i)_i)^2 \gamma_k \Sigma^{2r}\bigg|\bigg|\bigg|&\le & \sup_{0\le x\le 1/ \gamma_0} \sum_{k=1}^n \prod_{i=1}^n (1-\gamma_i x)^2 \gamma_k x^{2r} \nonumber\\ 
 &\le & \sup_{0\le x\le 1/ \gamma_0} \sum_{k=1}^n \exp\(-\sum_{i=1}^k \gamma_i x \) \gamma_k x^{2r} \nonumber\\
&\le & \sup_{0\le x\le 1/ \gamma_0} \sum_{k=1}^n \exp\(-k \gamma_k x \) \gamma_k x^{2r}    \quad \text{if $(\gamma_k)_k$ is decreasing} \nonumber\\
&\le & \gamma_0 \sup_{0\le x\le 1/ \gamma_0} \sum_{k=1}^n \exp\(-k \gamma_k x \)  x^{2r}  \nonumber \\
 &\le & \gamma_0 \sup_{0\le x\le 1/ \gamma_0} \sum_{k=1}^n \exp\(-k^{1-\rho}  \gamma_0 x \)  x^{2r} \quad \text{if $(\gamma_k)_i=\frac{\gamma_0}{k^\rho}$}\nonumber\\
 &\le & \gamma_0 \sup_{0\le x\le 1/ \gamma_0} x^{2r} \int_{u=0}^{n} \exp\(-u^{1-\rho} \gamma_0 x \)  du \quad \nonumber \\
\int_{u=0}^{n-1} \exp\(-u^{1-\rho}  \gamma_0 x \)  du &\le &  n    \quad \text{clearly, but also} \nonumber\\
\int_{u=0}^{n-1} \exp\(-u^{1-\rho} \gamma_0 x \)  du &\le &  \int_{t=0}^{\infty} \exp\(-t^{1-\rho} \)  (x\gamma_0)^{-\frac{1}{1-\rho}} dt \quad \text{changing variables. So that~:}    \nonumber \\
\bigg|\bigg|\bigg| \sum_{k=1}^n  D(1,n, (\gamma_i)_i)^2 \gamma_k \Sigma^{2r}\bigg|\bigg|\bigg|&\le & \gamma_0 \sup_{0\le x\le 1/ \gamma_0} x^{2r} \(n \wedge  I (x\gamma_0)^{-\frac{1}{1-\rho}} \) \nonumber\\
&\le & \gamma_0 C_1 \sup_{0\le x\le 1/ \gamma_0}  \(n x^{2r}\wedge   x^{2r-\frac{1}{1-\rho}} \) \text{ and if ${2r-\frac{1}{1-\rho}}<0$} \nonumber\\
&\le & \gamma_0 C_1 n^{1-2r(1-\rho)} . \label{2} 
 \end{eqnarray} 

And finally, using \eqref{1},  \eqref{2}~:
\begin{eqnarray}
\frac{1}{2   n (1- \gamma_0 R^2)}  \sum_{k=1}^n \gamma_k \E \|\Xi_k\|^2 &\le& \frac{\gamma_0 C_1 \ \|\Sigma^{-r} \n{0}\|_{\Ld}^2 R^2 }{2   (1- \gamma_0 R^2)}  (n\gamma_n)^{-2r} \nonumber  \\
&\le & K (n\gamma_n)^{-2r} .
\end{eqnarray}

To conclude, we have to upper bound~:  $$\frac{1}{2   n (1- \gamma_0 R^2)} \sum_{i=1}^{n-1} \E \|\alpha_i\|^2 {\(-\frac{1}{\gamma_i} +\frac{1}{\gamma_{i+1}} \)}.$$

 By the induction we make to get \textbf{Lemma} \ref{lem.stoch.rec.onl}, we have~:\begin{eqnarray*}
\|\alpha_i\|^2&\le&   \|\alpha_{i-1}\|^2 + \gamma_i^{2} \|\Xi_i\|^{2}\\
&\le&  \sum_{k=1}^{i } \gamma_k^{2} \|\Xi_k\|^{2}
\\
&\le&  \sum_{k=1}^{i } \gamma_k \|\Xi_k\|^{2}\\
&\le&  C i \ (i\gamma_i)^{-2r}.
 \end{eqnarray*}
 
 So that (C  changes during calculation)~:
   \begin{eqnarray*}
 \frac{1}{2   n (1- \gamma_0 R^2)} \sum_{i=1}^{n-1} \E \|\alpha_i\|^2 {\(-\frac{1}{\gamma_i} +\frac{1}{\gamma_{i+1}} \)} &\le & C \frac{1}{n} \sum_{i=1}^{n-1} \E \|\alpha_i\|^2 \frac{\gamma_i}{i}\\
 &\le & C \frac{1}{n} \sum_{i=1}^{n-1} i \ (i\gamma_i)^{-2r} \frac{\gamma_i}{i}\\
  &\le & C \frac{1}{n} \sum_{i=1}^{n-1} \ (i\gamma_i)^{-2r} {\gamma_i}\\
   &\le & C \frac{\gamma_n}{(n\gamma_n)^{2r}}.
 \end{eqnarray*}

 So that we would get, replacing our result in \eqref{initresiduel.onl}~:
 \begin{equation}
  \E \langle \nb{n}^0- \nb{n}^{noise}, \Sigma(\nb{n}^0 - \nb{n}^{noise}) \rangle= O\(\frac{1}{n \gamma_n}\)^{2r} +O\(\frac{\gamma_n}{n \gamma_n}\)^{2r} = O\(\frac{1}{n \gamma_n}\)^{2r} .\label{resid.init.onl}
\end{equation}
 
And finally, with \eqref{princ.init.onl} and \eqref{resid.init.onl} in \eqref{minkows.init.onl}, 
 \begin{eqnarray}
\left(\E\left[\langle \nb{n}^{init}, \Sigma \nb{n}^{init} \rangle \right]\right) ^{1/2} &\le& \left(\E\left[\langle \nb{n}^{init}- \nb{n}^{0}, \Sigma \left(\nb{n}^{init}- \nb{n}^{0}  \right)\rangle \right]\right) ^{1/2} + \left(\E\left[\langle \nb{n}^{0}, \Sigma \nb{n}^{0} \rangle \right] \right)^{1/2} \nonumber \\
&\le&  \(O\(\frac{1}{n \gamma_n}\)^{2r}\)^{1/2} + \bias(n, (\gamma_n)_n, g_\H, r)^{1/2}. \label{compl.init.onl}
\end{eqnarray} 

\subsubsection{Conclusion}
We conclude with both \eqref{compl.noise.onl}  and \eqref{compl.init.onl} in \eqref{init+sto}~: 
\begin{equation}
 \left(\E\left[\| \nb{n}\|^2_{\Ld} \right]\right)^{1/2} \le  \(C \var(n, \gamma, \alpha, \sigma, s )\)^{1/2}+ \(O\(\frac{1}{n \gamma_n}\)^{2r}\)^{1/2} + \bias(n, (\gamma_n)_n, g_\H, r)^{1/2}.
\end{equation}

Which gives Theorem \ref{prop.dinf.rand.onl} using Lemmas \ref{bias_gam_var} and \ref{var_gam_var}. Once again, deriving corollaries is simple.

\subsection{A lemma on stochastic recursion, $r\geq1/2$}

\begin{Lem} \label{lem:fixproof} If we consider the recursion $\eta_{n+1} = (I-\gamma x_n \otimes x_n) \eta_n$, with $\eta_0=g_\H$, we have 
\begin{eqnarray*}
 \E\langle\bar \eta_n , 
 \Sigma  \bar \eta_n \rangle  &\le&  (1 + (R^{2\alpha} \gamma^{1+\alpha} n s^2)^{\frac{2r-1}{\alpha}} )  \frac{ \| \Sigma^{-r} \n{0} \|^{2}_{L2}}{(\gamma n)^{2r}} .
 \end{eqnarray*}
 
\end{Lem}

Let us first state a few properties of symmetric matrices that are useful here. 
First we recall that the order $\lecc$ is defined by $M\lecc N$ if $N-M$ is sdp. This is an order on $S_n$, which is not a total order.
We say that a function $f$ is matrix increasing if $M\lecc N$ implies $f(M)\lecc f(N).$ We have the following special cases :
\begin{itemize}
\item The function $M\mapsto M^{2}$ is not matrix increasing on $S_n^{+}$.
\item For any $q\in [0;1]$, the function $M\mapsto M^{q}$ is matrix increasing on  $S_n^{+}$.
\item For any $N$, the function $M\mapsto N^{\top} M N$ is matrix increasing.
\item For some  $N$,  the function $M\mapsto N^{\top} M +  M N$ is not matrix increasing.
\item For any  $N$,  the function $M\mapsto( N\otimes \idm +  \idm \otimes N)^{-1}$ is  matrix increasing.
\item $\exp$ is not  matrix increasing, $\log$ is matrix increasing.
\item if $A\lecc B$, $A\lecc C$ and $BC=CB$, then for any $q \in [0;1]$,  $A\lecc B^{q}C^{1-q}.$
\end{itemize}
It is important to notice that it often occurs that $f$ is matrix increasing and its inverse $f^{-1} $ is not (square/square root, exp/log, left and right multiplication).

Notation $\wedge $ is somehow ill defined as $A\wedge B $ may be neither $A$ nor $B$. However, we use this notation as a shortcut for a matrix $C$ such that $C\lecc A$ and $C\lecc B$.

To prove Lemma~\ref{lem:fixproof}, we consider the stochastic recursion, in the case $r\geq1/2$.

We consider a full expansion of the  function value  $\Vert \Sigma^{1/2}(\bar \eta_{n})\Vert^{2}$.
 This corresponds to 
 $$ \eta_n= \theta_n - \theta_0 =   M(n,1) ( \eta_0) =   M(n,1) ( \theta_0 -g_\H) $$
Where the matrix $M(n,k) $ is $(I-\gamma x_n\otimes x_n) \cdots(I-\gamma x_k\otimes x_k)  $, for $k\le n$ 
 We consider the recursion without noise and rely on an explicit decomposition.

 \begin{eqnarray*}
 n^2 \E\langle\bar \eta_n  , 
 \Sigma ( \bar \eta_n   ) \rangle
 & = & \E \sum_{i=0}^n \sum_{j=0}^n \langle \eta_i  , 
 \Sigma ( \eta_j  ) \rangle\\
 & = & \E \sum_{i=0}^n  \langle \eta_i   ,
  \Sigma ( \eta_i  ) \rangle
 + 2 \E \sum_{i=0}^{n-1} \sum_{j=i+1}^n \langle  \eta_i  ,
 \Sigma ( \eta_j  ) \rangle.
 \end{eqnarray*}
 
  Moreover,
 \begin{eqnarray*}
  & & \E \sum_{i=0}^{n-1} \sum_{j=i+1}^n\langle  \eta_i  ,
 \Sigma ( \eta_j  ) \rangle\\
 & = & 
 \E \sum_{i=0}^{n-1} \sum_{j=i+1}^n \bigg\langle  \eta_i  ,  \Sigma \bigg[
 M(j,i+1) ( \eta_i  )
 \bigg] \bigg\rangle
\\
& = & 
 \E \sum_{i=0}^{n-1} \sum_{j=i+1}^n \langle \eta_i  , 
 \Sigma  
( \idm - \gamma \Sigma )^{j-i}
( \eta_i  ) \rangle \mbox{ because } M(j,i+1) \mbox{ and } \eta_i \mbox{ are independent,}
\\
& = & 
 \E \sum_{i=0}^{n-1} \Big\langle \eta_i  ,
( \gamma^{-1} \big[  ( \idm - \gamma \Sigma   ) - ( \idm - \gamma \Sigma)^{n-i+1} \big] \wedge n \Sigma( \idm - \gamma \Sigma   ) )
( \eta_i  ) \Big\rangle
\\
& \le &   \E \sum_{i=0}^{n}  \langle \eta_i   , A_{i,n}
 ( \eta_i    ) \rangle
  - \E \sum_{i=0}^{n}  \langle \eta_i   , \Sigma ( \eta_i   ) \rangle.
 \end{eqnarray*}
 with $A_{i,n }\lecc  (\gamma^{-1} I \wedge  n \Sigma)$ (meaning  $A_{i,n }\lecc  \gamma^{-1} I $ and  $A_{i,n }\lecc n \Sigma$). As for $\i \in [0;n]$,  $A_{i,n}\lecc A_{0,n}=:A$, we only need to get an upper bound on  :
  $  \E \sum_{i=0}^{n}  \langle \eta_i  ,  A
 ( \eta_i    ) \rangle$, to get a bound on
$n^2 \E \| \Sigma^{1/2} ( \bar \eta_n  ) \|^2$.

However, \begin{eqnarray*}
  \E \sum_{i=0}^{n}  \langle \eta_i   , A ( \eta_i    ) \rangle &=&    \E \sum_{i=0}^{n}  \langle   ( \eta_0  ), (M(i,1))^{* } A  M(i,1) ( \eta_0  )\rangle \\&=& \langle\langle  \E \sum_{i=0}^{n}    (M(i,1))^{\top } A  M(i,1)  , E_0\rangle\rangle\end{eqnarray*}
  as $\eta_i   = M(i,1) ( \eta_0  )$, 
 with $E_0=  ( \eta_0  ) ( \eta_0  )^\ast .$ and $\langle\langle\cdot, \cdot\rangle\rangle$ denoting the Froebenius scalar product between matrices.
 
 We consider the operator $T$ from symmetric matrices to symmetric matrices defined as 
 $$
 TA =\Sigma A +   A\Sigma- \gamma 
 E \big[ x_n \otimes x_n A  x_n \otimes x_n\big].
 $$
 of the form $TA =\Sigma A +   A\Sigma- \gamma  S A$. 
 
 We can make the following remarks : \begin{itemize}
 \item Operator $(I-\gamma T) $ is matrix increasing, as it is by definition  $M\mapsto \E [x_n\otimes x_n \Sigma x_n \otimes x_n]$. 
\item
 The operator $S$ is self-adjoint and positive.
 \end{itemize}
 
  We have for any symmetric matrix $A$:
 $$
 \E M(i,1)^\ast A M(i,1) =  ( \idm - \gamma T )^i A.
 $$
 Thus, 
  $$
 \E \sum_{i=0}^{n}   M(i,1)^\ast A
  M(i,1)  
  =   \sum_{i=0}^{n}
    ( \idm - \gamma T )^i A $$


  We have from previous calculations (case $r=1/2$ previously) that the main quantity of interest $ \sum_{i=0}^{n}   ( \idm - \gamma T )^i  A$ satisfies :

\begin{eqnarray}
 \gamma \sum_{i=0}^{n}   ( \idm - \gamma T )^i  A \lecc n  I. \label{majorP}
\end{eqnarray}
  
  We now show the following lemma :  
  \begin{Lem}
 \begin{eqnarray} \gamma  \sum_{i=0}^{n}   ( \idm - \gamma T )^i  A \le \gamma^{-1} \Sigma^{-1} + (n\gamma)^{1/\alpha} \tr(\Sigma^{\alpha})  \Sigma^{-1} . \label{majrP} \end{eqnarray}
  \end{Lem}
    
\begin{proof}
Let us denote P the quantity of interest $P=  \sum_{i=0}^{n}   ( \idm - \gamma T )^i  A $. We clearly have that $P\lecc n A$, and $P\lecc \gamma^{-1 }  T^{-1} A$.
As a consequence,  we first consider an upper bound on $ M_{A}:=T^{-1} A$.

We have : 
\begin{eqnarray}
A & = &   \Sigma M_A + M_A \Sigma - \gamma S M_A. 
\end{eqnarray}
Thus :
 \begin{eqnarray*}
M_A & = &  \big[\Sigma \otimes \idm  + \idm \otimes \Sigma \big]^{-1} A + \gamma  \big[ \Sigma  \otimes \idm + \idm \otimes  \Sigma  \big]^{-1} S M_A 
\end{eqnarray*}
 As $  \big[\Sigma \otimes \idm  + \idm \otimes \Sigma \big]^{-1}$ is a matrix increasing operator, we have that : 
 $  \big[\Sigma \otimes \idm  + \idm \otimes \Sigma \big]^{-1} A \lecc   \big[\Sigma \otimes \idm  + \idm \otimes \Sigma \big]^{-1} n \Sigma = \frac{n}{2} \idm$ and $  \big[\Sigma \otimes \idm  + \idm \otimes \Sigma \big]^{-1} A \lecc   \big[\Sigma \otimes \idm  + \idm \otimes \Sigma \big]^{-1} \gamma^{-1 } \idm = \frac{1}{2\gamma} \Sigma^{-1}$.
 
 Moreover, 
\begin{eqnarray}
S M_A  &\lecc & \tr(SM_A) \idm
\end{eqnarray}

Moreover we can upper bound $\tr( SM ) :$ as 
$$ \tr ( A )= 2 \tr (\Sigma M_A)- \gamma \tr  \E ( x_n \otimes x_n M _A  x_n \otimes x_n)  $$
 And
 $$\tr  \E ( x_n \otimes x_n M _A  x_n \otimes x_n) \leqslant R^2 \tr M_A \Sigma .$$
This implies
 \begin{eqnarray*}
 \tr A & \geqslant &  \frac{1} {R^2}   \tr  S M_A. \end{eqnarray*}
And as $A\lecc n^{1/\alpha} \gamma^{-1/\alpha} \Sigma^{1/\alpha}$, we finally have :
 
\begin{eqnarray*}
S M_A  &\lecc &  R^2  n^{1/\alpha} \gamma^{-1+1/\alpha} \tr(\Sigma^{1/\alpha}) \idm
\end{eqnarray*}

Thus : \begin{eqnarray*} P&\lecc &\gamma^{-2} \Sigma^{-1} +   R^2  n^{1/\alpha} \gamma^{-1+1/\alpha} \tr(\Sigma^{1/\alpha}) \Sigma^{-1} \\
\gamma P&\lecc& ( \gamma^{-1} +   R^2  n^{1/\alpha} \gamma^{1/\alpha} \tr(\Sigma^{1/\alpha})) \Sigma^{-1} . \end{eqnarray*}
\end{proof}  
  
Combining equation \eqref{majorP} and \eqref{majrP} we get, for any $q\in [0;1]$ : \begin{eqnarray} \gamma P &\lecc& ( \gamma^{-1}  + R^2 (n^{1/\alpha} \gamma^{1/\alpha} \tr(\Sigma^{1/\alpha})) ^{q } n^{1-q} \Sigma^{-q}
\end{eqnarray}

Thus with $q=-1+2r$, for $r \in [1/2;1]$, we get
\begin{eqnarray}  \gamma P&\lecc& ( \gamma^{-1}  + R^2(n^{1/\alpha} \gamma^{1/\alpha} \tr(\Sigma^{1/\alpha})) ^{2r-1 } n ^{2-2r}\Sigma^{1-2r}
\end{eqnarray}

Thus \begin{eqnarray*} \E    \langle\langle P,  E_0  \rangle\rangle  &\le&\frac{1}{\gamma}  (\gamma^{-1} + R^2 \gamma^{1/\alpha} n^{1/\alpha} \tr(\Sigma^{1/\alpha}))^{2r-1}  n ^{2-2r}  \| \Sigma^{-r} \eta_0 \|^{2}_{L2} \\
&=&\frac{1}{\gamma^{2r}}  ( 1 + R^2 \gamma^{1+1/\alpha} n^{1/\alpha} \tr(\Sigma^{1/\alpha}))^{2r-1}  n ^{2-2r}  \| \Sigma^{-r} \eta_0 \|^{2}_{L2} \end{eqnarray*}
And the quantity  $R^2 \gamma^{1+1/\alpha} n^{1/\alpha} $ is always going to 0 for the optimal choice of $\gamma$. Moreover, the exponent $2r-1$ is logically vanishing when $r\rightarrow 1/2.$

And we have 
\begin{eqnarray*} \E\langle\bar \eta_n  , 
 \Sigma ( \bar \eta_n   ) \rangle &\le&  (1 + R^2 \gamma^{1+1/\alpha} n^{1/\alpha} \tr(\Sigma^{1/\alpha}))^{2r-1}  \frac{ \| \Sigma^{-r} \eta_0 \|^{2}_{L2}}{(\gamma n)^{2r}} \\
 &\le&  (1 + (R^{2\alpha} \gamma^{1+\alpha} n s^2)^{\frac{2r-1}{\alpha}} )  \frac{ \| \Sigma^{-r} \eta_0 \|^{2}_{L2}}{(\gamma n)^{2r}} 
 \end{eqnarray*}

which concludes the proof of the Lemma.

\subsection{Some quantities}   
\label{subsec:technicalcal}
\label{some_quantities}
 In this section, we bound the main quantities which are involved above.
 
\subsubsection{Lemma \ref{bias_gam_const}}

\begin{proof}[\textbf{Lemma~\ref{bias_gam_const}}]
\ \\

If $0\le r\le 1$ :  
\vspace*{-1em}
\begin{eqnarray*}
\Bias(n, \gamma, g_\H , r)&=&\frac{1}{n^2}\big\langle \sum_{k=0}^{n-1} (I-\gamma \Sigma)^k g_\H,  \sum_{k=0}^{n-1} (I-\gamma \Sigma)^k\ \  \Sigma g_\H\big\rangle\\
&=&\frac{1}{n^2}\big\langle \sum_{k=0}^{n-1} (I-\gamma \Sigma)^k \Sigma^{2r} \Sigma^{-r+1/2} g_\H,  \sum_{k=0}^{n-1} (I-\gamma \Sigma)^k\ \   \Sigma^{-r+1/2}g_\H\big\rangle \\
&=&\frac{1}{n^2} \bigg|\bigg| \sum_{k=0}^{n-1} (I-\gamma \Sigma)^k \Sigma^{r} (\Sigma^{-r+1/2} g_\H)\bigg|\bigg|^2\\
&\le & \frac{1}{n^2} \bigg|\bigg|\bigg| \sum_{k=0}^{n-1} (I-\gamma \Sigma)^k \Sigma^{r} \bigg|\bigg|\bigg|^2\ \  \bigg|\bigg|\Sigma^{-r+1/2} g_\H\bigg|\bigg| ^2\\
&=& \frac{1}{n^2} \gamma^{-2r} \bigg|\bigg|\bigg| \sum_{k=0}^{n-1} (I-\gamma \Sigma)^k \gamma^r\Sigma^{r} \bigg|\bigg|\bigg|^2\ \  \bigg|\bigg|\Sigma^{-r} g_\H\bigg|\bigg|^2_{\mathcal{L}^2_\rho}\\
&\le& \frac{1}{n^2} \gamma^{-2r} \sup_{0\le x\le 1} \left(\sum_{k=0}^{n-1} (1-x)^k x^{r}\right)^2 \bigg|\bigg|\Sigma^{-r} g_\H\bigg|\bigg|^2_{\mathcal{L}^2_\rho}\\
&\le & \left( \frac{1}{(n\gamma)^{2r}} \right)\bigg|\bigg|\Sigma^{-r} g_\H\bigg|\bigg|_{\Ld}^2 .
\end{eqnarray*}
Using the inequality : \begin{equation}
 \sup_{0\le x\le 1} \left(\sum_{k=0}^{n-1} (1-x)^k x^{r}\right) \le n^{1-r} \label{maj_ineq}.
\end{equation} 
Indeed :
\begin{eqnarray*}
\left(\sum_{k=0}^{n-1} (1-x)^k x^{r}\right)&=& \frac{1-(1-x)^n}{x} x^{r}\\
&=& (1-(1-x)^n) x^{r-1}.
\end{eqnarray*}
And we have, for any $n\in \N, r \in [\,0;1], x\in [\,0;1]$ : $ (1-(1-x)^n)  \le (nx)^{1-r} $ : \begin{enumerate}
\item if $ nx \le 1 $ then $ (1-(1-x)^n)  \le nx \le (nx)^{1-r} $ (the first inequality can be proved by deriving the difference).
\item if $ nx \geq 1 $ then $ (1-(1-x)^n)  \le 1 \le (nx)^{1-r} $ .
\end{enumerate}

If $ r \geq 1 $,  $ x\mapsto(1-(1-x)^n) $ is increasing on $ [\,0;1] $ so $ \sup_{0\le x\le 1} \left(\sum_{k=0}^{n-1} (1-x)^k x^{r}\right)  =1$ : there is no improvement in comparison to $ r=1 $ :

\begin{eqnarray*}
\Bias(n, \gamma, g_\H, r)&\le & \left( \frac{1}{n^2 \gamma^{2r}} \right)\bigg|\bigg|\Sigma^{-r} g_\H\bigg|\bigg|_{\Ld}^2 .
\end{eqnarray*}
\end{proof}
 
\subsubsection{Lemma \ref{var_gam_const}}
\begin{proof} [\textbf{Lemma~\ref{var_gam_const}} ]
\ \\
In the following proof, we consider $ s=1 $. It's easy to get the complete result replacing in the proof below `` $ \gamma  $'' by `` $ s^2 \gamma $''. 
We have, for $j \in \N $, still assuming $ \gamma \Sigma \lem I$, and by a comparison to the integral :
\begin{eqnarray*}
 \tr \left(I-(I-\gamma \Sigma)^{j}\right)^2\Sigma^{-1} C&=& \sigma^2 \tr \left(I-(I-\gamma \Sigma)^{j}\right)^2 \\
 &\le & 1+\sigma^2 \int_{u=1}^{\infty} \left(1-\left(1-\frac{\gamma}{u^\alpha}\right)^j\right)^2 du \\ \qquad & & \mbox{(1 stands for the first term in the sum)}\\ 
&=&1+\sigma^2 \int_{u=1}^{(\gamma j)^{\frac{1}{\alpha}}} \left(1-\left(1-\frac{\gamma}{u^\alpha}\right)^j\right)^2 du \\
&&\hspace{3em}+\sigma^2 \int_{u=(\gamma j)^{\frac{1}{\alpha}}}^{\infty} \left(1-\left(1-\frac{\gamma}{u^\alpha}\right)^j\right)^2 du. \\
\end{eqnarray*}
Note that the first integral may be empty if $\gamma j \le 1$.
We also have: 
\begin{eqnarray*}
\tr \left(I-(I-\gamma \Sigma)^{j}\right)^2\Sigma^{-1} C &\geq & \sigma^2 \int_{u=1}^{\infty} \left(1-\left(1-\frac{\gamma}{u^\alpha}\right)^j\right)^2  du .\\
\end{eqnarray*}
Considering that $g_j: u \mapsto \left(1-\left(1-\frac{\gamma}{u^\alpha}\right)^j\right)^2$ is a decreasing function of $u$ we get :
$$ \forall u \in [1; (\gamma j)^{\frac{1}{\alpha}}], \quad (1-e^{-1})^2 \le g_j(u) \le 1. $$
Where we have used the fact that $ \left(1-\frac{1}{j}\right)^j \le e^{-1}$ for the left hand side inequality. Thus we have proved :
$$
(1-e^{-1})^2 (\gamma j)^{\frac{1}{\alpha}}  \le \int_{u=1}^{(\gamma j)^{\frac{1}{\alpha}}} \left(1-\left(1-\frac{\gamma}{u^\alpha}\right)^j\right)^2 du \le  (\gamma j)^{\frac{1}{\alpha}}. 
$$
For the other part of the sum, we consider $h_j: u \mapsto \left(\frac{1-\left(1-\frac{\gamma}{u^\alpha}\right)^j}{\frac{\gamma}{u^\alpha}}\right)^2$ which is an increasing function of u. So :
$$ \forall u \in [(\gamma j)^{\frac{1}{\alpha}}; +\infty], \quad (1-e^{-1})^2 j^2 \le h_j(u) \le j^2 , $$ using the same trick as above.
Thus : 
\begin{eqnarray*}
 \int_{u=(\gamma j)^{\frac{1}{\alpha}} }^{\infty} \left(1-\left(1-\frac{\gamma}{u^\alpha}\right)^j\right)^2 du &=& \int_{u=(\gamma j)^{\frac{1}{\alpha}} }^{\infty} h_j(u) \left(\frac{\gamma}{u^\alpha}\right)^2 du\\
 &\le & j^2  \int_{u=(\gamma j)^{\frac{1}{\alpha}} }^{\infty} \left(\frac{\gamma}{u^\alpha}\right)^2 du\\
 &\le & j^2 \gamma^2 \int_{u=(\gamma j)^{\frac{1}{\alpha}}  }^{\infty} \left(\frac{1}{u^\alpha} \right)^2 du\\ 
 &= & j^2 \gamma^2 \left[ \frac{1}{(1-2\alpha) u^{2\alpha-1}} \right]_{u=(\gamma j)^{\frac{1}{\alpha}}  }^{\infty}\\ 
&= & j^2 \gamma^2  \frac{1}{(2\alpha-1)  ((\gamma j)^{\frac{1}{\alpha}} )^{2\alpha-1}} \\ 
&= & \frac{1}{(2\alpha-1)} (j \gamma)^{\frac{1}{\alpha}}.
\end{eqnarray*}
And we could get, by a similar calculation  :
$$  \int_{u=(\gamma j)^{\frac{1}{\alpha}} +1}^{\infty} \left(1-\left(1-\frac{\gamma}{u^\alpha}\right)^j\right)^2 du \geq(1-e^{-1})^2 \frac{1}{(2\alpha-1)} (j \gamma)^{\frac{1}{\alpha}}.$$
Finally, we have shown that : 
$$ C_1 (j \gamma)^{\frac{1}{\alpha}} \le  \tr \left(I-(I-\gamma \Sigma)^{j}\right)^2 \le C_2 (j \gamma)^{\frac{1}{\alpha}}+1.$$
Where $C_1= (1-e^{-1})^2 \ (1+\frac{1}{(2\alpha-1)})$ and $C_2= (1+\frac{1}{(2\alpha-1)})$ are real constants.
To get the complete variance term we have to calculate :
$ \frac{\sigma^2}{n^2} \sum_{j=1}^{n-1}  \tr \left(I-(I-\gamma \Sigma)\right)^{j} .$ 
We have :
\begin{eqnarray*} 
\frac{\sigma^2}{n^2} \sum_{j=1}^{n-1}  \tr \left(I-(I-\gamma \Sigma)^{j}\right)^2 &\le & \frac{\sigma^2}{n^2}   \sum_{j=1}^{n-1} \left(C_2(j \gamma)^{\frac{1}{\alpha}}+ 1\right)\\
&\le & \frac{\sigma^2}{n^2} C_2 \gamma^{\frac{1}{\alpha}} \int_{u=2}^{n} u^{\frac{1}{\alpha}} du + \frac{\sigma^2}{n}\\
&\le & \frac{\sigma^2}{n^2} C_2 \gamma^{\frac{1}{\alpha}} \frac{\alpha}{\alpha+1} n^{\frac{\alpha+1}{\alpha}} + \frac{\sigma^2}{n}\\
&\le & \frac{\alpha \ \sigma^2  \ C_2}{\alpha+1} \  \frac{\gamma^{\frac{1}{\alpha}}}  {n^{1-\frac{1}{\alpha}}} + \frac{\sigma^2}{n}.
\end{eqnarray*} 
That is :$$ (1-e^{-1})^2 \  C(\alpha) \  \sigma^2  \frac{\gamma^{\frac{1}{\alpha}}}  {(n-1)^{1-\frac{1}{\alpha}}} \le\Var(n,\gamma,\alpha,\sigma^2)\le C(\alpha) \  \sigma^2  \frac{\gamma^{\frac{1}{\alpha}}}  {n^{1-\frac{1}{\alpha}}} + \frac{\sigma^2}{n},$$
with $C(\alpha)=\frac{2 \alpha^2  }{(\alpha+1)(2\alpha-1)}$.

\end{proof}

\subsubsection{Lemma \ref{bias_gam_var}}
\begin{proof}

\begin{eqnarray*}
\frac{1}{n^2} \bigg\| \Sigma^{1/2} \sum_{k=1}^n \prod_{i=1}^k \(I-\gamma_i \Sigma \) g_\H \bigg\|_K^2
&\le &  \frac{1}{n^2} \bigg|\bigg|\bigg|\sum_{k=1}^n \prod_{i=1}^k \(I-\gamma_i \Sigma \) \Sigma^r\bigg|\bigg|\bigg|^2 ||\Sigma^{1/2-r} g_\H||_K^2 \\
&\le &  \frac{1}{n^2} \bigg|\bigg|\bigg|\sum_{k=1}^n \prod_{i=1}^k \(I-\gamma_i \Sigma \) \Sigma^r\bigg|\bigg|\bigg|^2 ||\Sigma^{-r} g_\H||_{\Ld}^2 .
\end{eqnarray*}

Moreover :
\begin{eqnarray*}
\bigg|\bigg|\bigg|\sum_{k=1}^n \prod_{i=1}^k \(I-\gamma_i \Sigma \) \Sigma^r\bigg|\bigg|\bigg| &\le& \sup_{0\le x\le 1} \sum_{k=1}^n \prod_{i=1}^k \(I-\gamma_i x \) x^r\\
&\le & \sup_{0\le x\le 1} \sum_{k=1}^n \exp\(-\sum_{i=1}^k \gamma_i x \) \gamma_k x^{r} \nonumber\\
&\le & \sup_{0\le x\le 1} \sum_{k=1}^n \exp\(-k \gamma_k x \) \gamma_k x^{r}    \quad \text{if $(\gamma_k)_k$ is decreasing} \nonumber\\
&\le & \sup_{0\le x\le 1} \sum_{k=1}^n \exp\(-k \gamma_k x \)  x^{r}  \nonumber \\
 &\le & \sup_{0\le x\le 1} \sum_{k=1}^n \exp\(-k^{1-\zeta} x \)  x^{r} \quad \text{if $(\gamma_k)_i=\frac{1}{k^\zeta}$}\nonumber\\
 &\le &  \sup_{0\le x\le 1} x^{r} \int_{u=0}^{n} \exp\(-u^{1-\zeta}  x \)  du \quad \text{by comparison to the integral }\nonumber \\
\int_{u=0}^{n} \exp\(-u^{1-\zeta}   x \)  du &\le &  n    \quad \text{clearly, but also} \nonumber\\
\int_{u=0}^{n} \exp\(-u^{1-\zeta}  x \)  du &\le &  \int_{t=0}^{\infty} \exp\(-t^{1-\zeta} \)  (x)^{-\frac{1}{1-\zeta}} dt \quad \text{changing variables. So that :}    \nonumber \\
\bigg|\bigg|\bigg|\sum_{k=1}^n \prod_{i=1}^k \(I-\gamma_i \Sigma \) \Sigma^r\bigg|\bigg|\bigg| &\le & K \sup_{0\le x\le 1} x^{r} \(n \wedge x^{-\frac{1}{1-\zeta}} \) \nonumber\\
&\le & K \sup_{0\le x\le 1 }  \(n x^{r}\wedge   x^{r-\frac{1}{1-\zeta}} \) \text{ and if ${r-\frac{1}{1-\zeta}}<0$} \nonumber\\
&\le & K n^{1-r(1-\zeta)} . \label{2bis}
 \end{eqnarray*} 

So that : 
\begin{eqnarray*}
\frac{1}{n^2}\left\langle \sum_{k=1}^n \prod_{i=1}^k \(I-\gamma_i \Sigma \) g_\H , \sum_{k=1}^n \prod_{i=1}^k \(I-\gamma_i \Sigma \) \Sigma g_\H \right\rangle 
&\le &  \frac{1}{n^2}  \(K n^{1-r(1-\zeta)}\)^2 ||\Sigma^{-r} g_\H||_{\Ld}^2 \\
&\le& K^2 ||\Sigma^{-r} g_\H||_{\Ld}^2 n^{-2r(1-\zeta)}.
\end{eqnarray*}

Else if ${r-\frac{1}{1-\zeta}}>0$, then  $\sup_{0\le x\le 1 }  \(n x^{r}\wedge   x^{r-\frac{1}{1-\zeta}} \)=1$, so that 
\begin{eqnarray*}
\frac{1}{n^2}\left\langle \sum_{k=1}^n \prod_{i=1}^k \(I-\gamma_i \Sigma \) g_\H , \sum_{k=1}^n \prod_{i=1}^k \(I-\gamma_i \Sigma \) \Sigma g_\H \right\rangle 
&= & O\(\frac{||\Sigma^{-r} g_\H||_{\Ld}^2}{n^2}\).
\end{eqnarray*}
\end{proof}

\subsubsection{Lemma \ref{var_gam_var}}

\begin{proof}
 To get corollary \ref{var_gam_var}, we will just replace in the following calculations $\gamma$ by $s^2 \gamma$
We remind that : 

\begin{equation}
\Var\Big(n, (\gamma_i)_i, \Sigma, (\xi_i)_i\Big)=\frac{1}{n^2} \E \left\langle   \sum_{j=1}^n  \sum_{k=1}^j \left[ \prod_{i=k+1}^j (I-\gamma_i \Sigma)\right] \gamma_k \xi_k, \Sigma  \sum_{j=1}^n  \sum_{k=1}^j \left[ \prod_{i=k+1}^j (I-\gamma_i \Sigma)\right] \gamma_k \xi_k \right\rangle.
\end{equation}

For shorter notation, in the following proof, we note $\Var(n)=\Var\Big(n, (\gamma_i)_i, \Sigma, (\xi_i)_i\Big)$.

\begin{eqnarray*}
\Var(n)&=& \frac{1}{n^2} \E \left\langle   \sum_{j=1}^n  \sum_{k=1}^j \[ \prod_{i=k+1}^j (I-\gamma_i \Sigma)\] \gamma_k \xi_k, \Sigma  \sum_{j=1}^n  \sum_{k=1}^j \[ \prod_{i=k+1}^j (I-\gamma_i \Sigma)\] \gamma_k \xi_k \right\rangle\\
&=& \frac{1}{n^2} \E \left\langle   \sum_{k=1}^n  \(\sum_{j=k}^n \[ \prod_{i=k+1}^j (I-\gamma_i \Sigma)\]\) \gamma_k \xi_k, \Sigma  \sum_{k=1}^n  \(\sum_{j=k}^n \[ \prod_{i=k+1}^j (I-\gamma_i \Sigma)\] \)\gamma_k \xi_k \right\rangle\\
&=&\frac{1}{n^2} \sum_{k=1}^n \E \left\langle     \(\sum_{j=k}^n \[ \prod_{i=k+1}^j (I-\gamma_i \Sigma)\]\) \gamma_k \xi_k, \Sigma   \(\sum_{j=k}^n \[ \prod_{i=k+1}^j (I-\gamma_i \Sigma)\] \)\gamma_k \xi_k \right\rangle\\
&=&\frac{1}{n^2} \sum_{k=1}^n \E \left\langle    M_{n,k} \gamma_k \xi_k, \Sigma   M_{n,k} \gamma_k \xi_k \right\rangle \qquad \mbox{with} \  M_{n,k}:= \(\sum_{j=k}^n \[ \prod_{i=k+1}^j (I-\gamma_i \Sigma)\]\)\\
&=&\frac{1}{n^2} \sum_{k=1}^n \gamma_k^2 \ \E \left\langle    M_{n,k}  \xi_k, \Sigma   M_{n,k}  \xi_k \right\rangle=\frac{1}{n^2} \sum_{k=1}^n \gamma_k^2 \ \E \tr \(M_{n,k} \Sigma   M_{n,k}  \xi_k \otimes \xi_k \)\\
&\le& \frac{1}{n^2} \sum_{k=1}^n \gamma_k^2 \sigma^2  \tr \(M_{n,k}^2 \Sigma   \Sigma\)\\
&\le& \frac{1}{n^2} \sum_{k=1}^n \gamma_k^2 \sigma^2  \tr \(\(\sum_{j=k}^n \[ \prod_{i=k+1}^j (I-\gamma_i \Sigma)\]\) \Sigma  \)^2\\
&\le  & \frac{1}{n^2} \sum_{k=1}^n \gamma_k^2 \sigma^2  \sum_{t=1}^\infty \(\(\sum_{j=k}^n \[ \prod_{i=k+1}^j \(1-\gamma_i \frac{1}{t^\alpha}\)\]\) \frac{1}{t^\alpha} \)^2.
\end{eqnarray*}

Let's first upper bound: 
\begin{eqnarray*}
  \[ \prod_{i=k+1}^j \(1-\gamma_i \frac{1}{t^\alpha}\)\]  &\le& \exp \sum _{i=k+1}^j (\gamma_i \frac{1}{t^\alpha}) \\
  &=&   \exp -\sum _{i=k+1}^j \(\frac{1}{i^{\zeta}} \frac{1}{t^\alpha}\) \text{ if } \gamma_i=\frac{1}{i^\zeta} \\
  &\le&  \exp -\frac{1}{t^\alpha} \int _{u=k+1}^{j+1} \(\frac{1}{u^{\zeta}} du \)\\
    &\le&  \exp -\frac{1}{t^\alpha} \frac{(j+1)^{1-\zeta}-(k+1)^{1-\zeta} }{1-\zeta}.
\end{eqnarray*}
Then
\begin{eqnarray*}
\sum_{j=k}^n \prod_{i=k+1}^j \(1-\gamma_i \frac{1}{t^\alpha}\) &\le& \sum_{j=k}^n \exp -\frac{1}{t^\alpha} \frac{(j+1)^{1-\zeta}-(k+1)^{1-\zeta} }{1-\zeta}\\
&\le& \int_{u=k}^n \exp -\frac{1}{t^\alpha} \frac{(u+1)^{1-\zeta}-(k+1)^{1-\zeta} }{1-\zeta} du\\
 &\le& (n-k) \quad \mbox{clearly}
 \end{eqnarray*}
(this upper bound is good when $t >> n^{1-\zeta}$), but we also have:

\begin{eqnarray*}
\int_{u=k}^n \exp -\frac{1}{t^\alpha} \frac{(u+1)^{1-\zeta}-(k+1)^{1-\zeta} }{1-\zeta} du &=& \int_{u=k+1}^{n+1} \exp -\frac{1}{t^\alpha} \frac{u^{1-\zeta}-(k+1)^{1-\zeta} }{1-\zeta} du .
\end{eqnarray*}

With $\rho=1-\zeta, K_\zeta:=\frac{1}{(1-\zeta)^{1/\rho}t^{\alpha/\rho}}$ and 
\begin{eqnarray*}
v^\rho &=&\frac{1}{t^\alpha} \frac{(u)^{\rho}-(k+1)^{\rho} }{(1-\zeta)} \\
 v &=&\frac{1}{(1-\zeta)^{1/\rho} t^{\alpha/\rho}} \((u)^{\rho}-(k+1)^{\rho}\)^{1/\rho}  \\
  dv &=&K_\zeta  \frac{1}{\rho} \(u^{\rho}-(k+1)^{\rho}\)^{1/\rho-1} \rho u^{\rho-1} du\\
   dv &=&K_\zeta  \(1-\(\frac{k+1}{u}\)^{\rho}\)^{1/\rho-1}  du\\
      dv \frac{1}{K_\zeta  \(1-\(\frac{(k+1)^{\rho}}{t^\alpha C v^\rho+(k+1)^\rho}\)\)^{1/\rho-1}}&=& du \\
       dv \frac{1}{K_\zeta}\(\frac{t^\alpha C v^\rho+(k+1)^\rho}{  t^\alpha C v^\rho+(k+1)^\rho-{(k+1)^{\rho}}}\)^{1/\rho-1}&=& du \\
      dv \frac{1}{K_\zeta}\(\frac{t^\alpha C v^\rho+(k+1)^\rho}{  t^\alpha C v^\rho}\)^{1/\rho-1}&=& du   \\ 
            dv \frac{1}{K_\zeta}\(1+\frac{(k+1)^\rho}{  t^\alpha C v^\rho}\)^{1/\rho-1}&=& du 
\end{eqnarray*}

\begin{eqnarray*}
\int_{u=k}^n \exp -\frac{1}{t^\alpha} \frac{(u+1)^{\frac{\alpha}{\alpha+\beta}}-(k+1)^{\frac{\alpha}{\alpha+\beta}} }{(1-\zeta)} du &\le& \int_{0}^{\infty}  \frac{1}{K_\zeta}\(1+\frac{(k+1)^\rho}{  t^\alpha C v^\rho}\)^{1/\rho-1} \exp \(-v^\rho\)dv\\
&\hspace{-15em}\le& \hspace{-8em} \frac{2^{1/\rho-1}}{K_\zeta}  \int_{0}^{\infty} \(1\vee\frac{(k+1)^\rho}{  t^\alpha C v^\rho}\)^{1/\rho-1} \exp \(-v^\rho\)dv\\
&\hspace{-15em}\le& \hspace{-8em} 2^{1/\rho-1} (1-\zeta)^{1/\rho}t^{\alpha/\rho} \int_{0}^{\infty} \(1\vee\frac{(k+1)^{1-\rho}}{  (t^\alpha C)^{1/\rho-1} v^{1-\rho}}\) \exp \(-v^\rho\)dv.\\
&\hspace{-15em}\le& \hspace{-8em}  K t^{\alpha/\rho} \(I_1 \vee I_2 \frac{(k+1)^{1-\rho}}{  (t^\alpha )^{1/\rho-1} }\) \\
&\hspace{-15em}\le& \hspace{-8em}  K \(t^{\frac{\alpha}{1-\zeta}}  \vee t^\alpha (k+1)^{\zeta}\).
\end{eqnarray*}

Finally :
\begin{eqnarray*}
\Var(n)&\le&\frac{1}{n^2} \sum_{k=1}^n \gamma_k^2 \sigma^2  \sum_{t=1}^\infty \frac{1}{t^{2\alpha}} \((n-k) \wedge K \(t^{\frac{\alpha}{1-\zeta}}  \vee t^\alpha (k+1)^{\zeta}\)\)^2 \\
\Var(n)&\le&\frac{1}{n^2} \sum_{k=1}^n \gamma_k^2 \sigma^2  \sum_{t=1}^\infty \frac{1}{t^{2\alpha}} \((n-k)^2 \wedge K \(t^{2\frac{\alpha}{1-\zeta}} + t^{2\alpha} k^{2\zeta}\)\) \\
&\le&\underbrace{\frac{1}{n^2} \sum_{k=1}^n \gamma_k^2 \sigma^2 \sum_{t=1}^\infty \frac{1}{t^{2\alpha}} \((n-k)^2 \wedge K \(t^{2\frac{\alpha}{1-\zeta}} \)\)}_{S_1} \\& & \hspace{8em}
+\underbrace{\frac{1}{n^2}\sum_{k=1}^n \gamma_k^2 \sigma^2 \sum_{t=1}^\infty \frac{1}{t^{2\alpha}} \((n-k)^2 \wedge t^{2\alpha} k^{2\zeta}\) }_{S_2}\\
S_1&\le& K \frac{1}{n^2} \sum_{k=1}^n \gamma_k^2 \sigma^2 \( \sum_{t=1}^{(n-k)^\frac{1-\zeta}{\alpha}} \frac{1}{t^{2\alpha}}  t^{2\frac{\alpha}{1-\zeta}}   + \sum_{t=(n-k)^\frac{1-\zeta}{\alpha}}^\infty \frac{1}{t^{2\alpha}} (n-k) ^2\)\\
&\le& K \frac{1}{n^2} \sum_{k=1}^n \gamma_k^2 \sigma^2 \( \sum_{t=1}^{(n-k)^\frac{1-\zeta}{\alpha}}  t^{\frac{2\alpha \zeta}{1-\zeta}}   + (n-k) ^2 \sum_{t=(n-k)^\frac{1-\zeta}{\alpha}}^\infty \frac{1}{t^{2\alpha}} \)\\
&\le& G  \frac{1}{n^2} \sum_{k=1}^n \gamma_k^2 \sigma^2 \(   (n-k)^{\frac{1-\zeta}{\alpha}{(\frac{2\alpha \zeta}{1-\zeta} + 1)}}   + (n-k) ^2  \frac{1}{(n-k)^{\frac{1-\zeta}{\alpha}({2\alpha-1})}} \)\\
&\le& G  \frac{1}{n^2} \sum_{k=1}^n \gamma_k^2 \sigma^2 \(   (n-k)^{\frac{(2\alpha-1)\zeta +1}{\alpha}}  + (n-k) ^{ 2-\frac{1-\zeta}{\alpha}({2\alpha-1}) } \)\\
&=& 2 G \sigma^2  \frac{1}{n^2} \sum_{k=1}^n \frac{1}{k^{2\zeta}}   (n-k)^\frac{(2\alpha-1)\zeta +1}{\alpha}  \\
&\le& 2 G \sigma^2  \frac{1}{n^2} \sum_{k=1}^n     \(\frac{n}{k}-1\)^\frac{(2\alpha-1)\zeta +1}{\alpha}   k^{\frac{1-\zeta}{\alpha}}\\
&=&2 G \sigma^2  n^{-1+{\frac{1-\zeta}{\alpha}}} \frac{1}{n} \sum_{k=1}^n     \(\frac{1}{k/n}-1\)^\frac{(2\alpha-1)\zeta +1}{\alpha}   \(\frac{k}{n}\)^{\frac{1-\zeta}{\alpha}}\\
&=&2 G \sigma^2  n^{-1+\frac{1-\zeta}{\alpha}} \(\frac{1}{n} \sum_{k=1}^n     \(\frac{1}{k/n}-1\)^\frac{(2\alpha-1)\zeta +1}{\alpha}    \(\frac{k}{n}\)^\frac{1-\zeta}{\alpha} \)\\
&=&2 G \sigma^2  n^{-1+\frac{1-\zeta}{\alpha}} \(\frac{1}{n} \sum_{k=1}^n     \(\frac{1}{k/n}-1\)^{2\zeta}    \(1-\frac{k}{n}\)^\frac{1-\zeta}{\alpha} \).
\end{eqnarray*}

If $ \zeta<\frac{1}{2} $ then 
$$ \int_0^1 \(\frac{1}{x}-1\)^{2\zeta} (1-x) ^\frac{1-\zeta}{\alpha} dx< \infty $$
and 
\begin{eqnarray*}
S_1&\le&H  n^{-1+\frac{1-\zeta}{\alpha}} \(\frac{1}{n} \sum_{k=1}^n     \(\frac{1}{k/n}-1\)^{2\zeta}    \(1-\frac{k}{n}\)^\frac{1-\zeta}{\alpha} \)\\
&\le& H'  n^{-1+\frac{1-\zeta}{\alpha}} .
\end{eqnarray*}

If $ \zeta>\frac{1}{2} $ then 
$$ \int_0^1 \(\frac{1}{x}-1\)^{2\zeta} (1-x) ^\frac{1-\zeta}{\alpha} - \(\frac{1}{x}\)^{2\zeta} dx< \infty .$$

and 
\begin{eqnarray*}
S_1&\le& H n^{-1+\frac{1-\zeta}{\alpha}} \(\frac{1}{n} \sum_{k=1}^n     \(\frac{1}{k/n}-1\)^{2\zeta}    \(1-\frac{k}{n}\)^\frac{1-\zeta}{\alpha} - \(\frac{n}{k}\)^{2\zeta} + \frac{1}{n	} \sum_{k=1}^{n}  \(\frac{n}{k}\)^{2\zeta}  \) \\
&\le&H n^{-1+\frac{1-\zeta}{\alpha}}\(C+ n^{2\zeta-1}\)\\
&\le& C n^{-1+\frac{1-\zeta+\alpha(2\zeta-1)}{\alpha}} .
\end{eqnarray*}

\begin{eqnarray*}
S_2&=& \frac{1}{n^2}\sum_{k=1}^n \gamma_k^2 \sigma^2 \sum_{t=1}^\infty \frac{1}{t^{2\alpha}} \((n-k)^2 \wedge t^{2\alpha} k^{2\zeta}\) \\ 
&\le& \frac{1}{n^2}\sum_{k=1}^n \gamma_k^2 \sigma^2 \(\sum_{t=1}^{t_\ell} \frac{1}{t^{2\alpha}} t^{2\alpha} k^{2\zeta} +\sum_{t=t_\ell}^{\infty} \frac{1}{t^{2\alpha}} (n-k)^2 \) \quad \mbox{with } \quad t_\ell =\frac{\(n-k\)^{\frac{1}{\alpha}}}{k^\frac{\zeta}{\alpha}}\\
&\le& \frac{1}{n^2}\sum_{k=1}^n \gamma_k^2 \sigma^2 \(k^{2\zeta}\sum_{t=1}^{t_\ell} 1 +(n-k)^2 \sum_{t=t_\ell}^{\infty} \frac{1}{t^{2\alpha}} \)\\
&\le& \frac{1}{n^2}\sum_{k=1}^n \gamma_k^2 \sigma^2 \(k^{2\zeta} \frac{\(n-k\)^{\frac{1}{\alpha}}}{k^\frac{\zeta}{\alpha}}  +(n-k)^2 \(\frac{\(n-k\)^{\frac{1}{\alpha}}}{k^\frac{\zeta}{\alpha}}\)^{1-2\alpha} \) \\
&=& \frac{1}{n^2}\sum_{k=1}^n \gamma_k^2 \sigma^2 \(k^{2\zeta-\frac{\zeta}{\alpha}} \(n-k\)^{\frac{1}{\alpha}}  +(n-k)^{\frac{1}{\alpha}} k^{\frac{\zeta}{\alpha}(2\alpha-1)}\) \\
&=& \frac{2\sigma^2}{n^2}\sum_{k=1}^n \frac{1}{k^{2\zeta}} (n-k)^{\frac{1}{\alpha}} k^{\frac{\zeta}{\alpha}(2\alpha-1)} \\
&=& \frac{2\sigma^2}{n^2}\sum_{k=1}^n  k^{-\frac{\zeta}{\alpha}} \(n-k\)^{\frac{1}{\alpha}}  \\
&=& {2\sigma^2} n^{\(-1+ {-\frac{\zeta}{\alpha}} + {\frac{1}{\alpha}}\) } \frac{1}{n}\sum_{k=1}^n  \(\frac{k}{n}\)^{-\frac{\zeta}{\alpha}} \(1-\frac{k}{n}\)^{\frac{1}{\alpha}}  \\
&\le& {K n^{\(-1+ \frac{1-\zeta}{\alpha} \) }}.
\end{eqnarray*}
As we have a Riemann sum which converges.

Finally we get  : if $ 0<\zeta<\frac{1}{2} $ then

\begin{eqnarray*} \Var(n) &=& O\(\sigma^2 n^{-1+ \frac{1-\zeta}{\alpha}}\)\\
&=&O\(\sigma^2 \frac{\sigma^2 (s^2 \gamma_n){1/\alpha}}{n^{1-1/\alpha}} n^{-1+ \frac{1-\zeta}{\alpha}}\)
\end{eqnarray*}
where we have re-used the constants $s$ by formaly replacing in the proof $\gamma$ by $\gamma s^2$.

 and if $ \zeta>\frac{1}{2} $ then 

\begin{equation*} \Var(n) =  O\(\sigma^2 n^{-1+ \frac{1-\zeta}{\alpha}+ 2\zeta-1 } \).
\end{equation*}

Which is substantially Lemma~\ref{var_gam_var}.

\end{proof}

\end{appendices}

\bibliographystyle{ieeetr}
\bibliography{Css}

\end{document}